%% file: MainFile.tex
\documentclass[preprint,3p]{elsarticle}

\usepackage{amssymb,amsmath,amsthm}
\usepackage{array}
\usepackage{xcolor}
\usepackage{bm}
\usepackage{subfigure}

\DeclareMathAccent{\maxvec}{\mathord}{letters}{"7E}

\usepackage{xcolor}

\newtheorem{lemma}{Lemma}

\newtheorem{remark}{Remark}

\newcommand{\REV}[1]{\textcolor{black}{{#1}}}
\newcommand{\REVpk}[1]{\textcolor{black}{{#1}}}

\journal{Computer Methods in Applied Mechanics and Engineering}
\begin{document}
\begin{frontmatter}

                    
\title{Explicit synchronous partitioned scheme for coupled reduced order models based on composite reduced bases}


\cortext[cor1]{Corresponding author}
                    
\fntext[SC]{Department of Mathematics, Clemson University, USA}
\fntext[CCR]{Center for Computing Research, Sandia National Laboratories, Albuquerque, NM 87185, USA}
\fntext[SNLCA]{Quantitative Modeling \& Software Engineering Department, Sandia National Laboratories, P.O. Box 969, Livermore, CA 94551-0969, USA}
\fntext[sand-blurb]{\REVpk{This paper describes objective technical results and analysis. Any subjective views or opinions that might be expressed in the paper do not necessarily represent the views of the U.S. Department of Energy or the United States Government. Sandia National Laboratories is a multimission laboratory managed and operated by National Technology and Engineering Solutions of Sandia, LLC., a wholly owned subsidiary of Honeywell International, Inc., for the U.S. Department of Energy's National Nuclear Security Administration under contract DE-NA-0003525.}}
\author[SC,CCR]{Amy\ de Castro}
\ead{agmurda@clemson.edu}

\author[CCR]{Pavel Bochev\corref{cor1}}
\ead{pbboche@sandia.gov}

\author[CCR]{Paul Kuberry}
\ead{pakuber@sandia.gov}

\author[SNLCA]{Irina Tezaur}
\ead{ikalash@sandia.gov}

\begin{abstract}

This paper formulates, analyzes and demonstrates numerically a method for the \REV{explicit} partitioned solution of coupled interface problems involving combinations of projection-based reduced order models (ROM) and/or full order \REV{models} (FOMs). The method builds on the partitioned scheme developed in \cite{AdC:CAMWA}, which starts from a well-posed formulation of the coupled interface problem and uses its dual Schur complement to obtain an approximation of the interface flux. Explicit time integration of this problem decouples its subdomain equations and enables their independent solution on each subdomain. 
Extension of this partitioned scheme to coupled ROM-ROM or ROM-FOM problems requires formulations with non-singular Schur complements. To obtain these problems, we project a well-posed coupled FOM-FOM problem onto a composite reduced basis comprising separate sets of basis vectors for the interface and interior variables, and use the interface reduced basis as a Lagrange multiplier. 
Our analysis confirms that the resulting coupled ROM-ROM and ROM-FOM problems have provably non-singular Schur complements, independent of the mesh size and the reduced basis size. In the ROM-FOM case, analysis shows that one can also use the interface FOM space as a Lagrange multiplier.
We illustrate the theoretical and computational properties of the partitioned scheme through reproductive and predictive tests for a model advection-diffusion transmission problem.

\end{abstract}

\begin{keyword}
partitioned scheme \sep 
projection-based reduced order model (ROM) \sep 
interface \sep
transmission problem\sep 
\textit{inf-sup} condition \sep
proper orthogonal decomposition (POD) \sep
Galerkin method
\end{keyword}
\end{frontmatter}

\input{sec_Intro}


\input{sec_Notation}


\input{sec_ModelProb}


\input{sec_POD}
\input{sec_IVR_SplitBasis}


\input{sec_Analysis}


\input{sec_NumResults}


\input{sec_Conclusion}

\section*{Acknowledgments}
This material is based upon work supported by the U.S. Department of Energy, Office of Science, Office of Advanced Scientific Computing Research, Mathematical Multifaceted Integrated Capability Centers (MMICCs) program, under Field Work Proposal 22-025291 (Multifaceted Mathematics for Predictive Digital Twins (M2dt)), Field Work Proposal 20-020467, and the Laboratory Directed Research and Development program at Sandia National Laboratories.  
 The writing of this manuscript was funded in part by the fourth author's (Irina Tezaur's)
Presidential Early Career Award for Scientists and Engineers (PECASE). 

This article has been authored by an employee of National Technology \& Engineering Solutions of Sandia, LLC under Contract No. DE-NA0003525 with the U.S. Department of Energy (DOE). The employee owns all right, title and interest in and to the article and is solely responsible for its contents. The United States Government retains and the publisher, by accepting the article for publication, acknowledges that the United States Government retains a non-exclusive, paid-up, irrevocable, world-wide license to publish or reproduce the published form of this article or allow others to do so, for United States Government purposes. The DOE will provide public access to these results of federally sponsored research in accordance with the DOE Public Access Plan https://www.energy.gov/downloads/doe-public-access-plan.

\bibliographystyle{elsarticle-num}
\bibliography{AmydeCastro}

\end{document}

%% file: sec_Intro.tex
\section{Introduction} \label{AdC:sec:intro}
Partitioned methods  are an attractive alternative to monolithic approaches for \REV{both single and multi-physics applications. In the first case such schemes \REVpk{can increase the concurrency of the simulation, improving computational efficiency,} by using an artificial interface to split the computational domain into several subdomains. In the second case, where the interface is physical,  partitioned schemes enable both increased concurrency and} reuse of existing codes for the constituent physics components; see, e.g., \cite{AdC:Felippa_01_CMAME} for an expository survey. Because each individual component is solved independently, the codes can run at their ``sweet spots'' utilizing, e.g., multi-rate time integrators \cite{AdC:Gravouil_01_IJNME}. 
Performance of partitioned schemes \REV{in both simulation contexts} can be further enhanced by replacing the full-fidelity models in one or more subdomains by computationally efficient projection-based reduced order models (ROMs). 

This work continues our efforts in \cite{DeCastro_23_INPROC} to extend the partitioned schemes in \cite{AdC:CAMWA} and \cite{AdC:RINAM} to interface problems in which a projection-based ROM on one of the subdomains is coupled to either a full order model (FOM) or a ROM on the other subdomain. In \cite{DeCastro_23_INPROC}, we defined the subdomain ROMs by utilizing full subdomain bases obtained by performing proper orthogonal decomposition (POD) 
\cite{AdC:Sirovich1987, AdC:Holmes1996}
on a collection of snapshots containing both the interior and interface degrees of freedom (DoFs). While this strategy is common in \REV{applications of} domain decomposition \REV{ideas to} ROM (see, e.g., \cite{AdC:Maday_02_CRM}), it does not guarantee that the dual Schur complement system for the Lagrange multiplier is non-singular. Unique solvability of this system is essential for the extension of the partitioned schemes in \cite{AdC:CAMWA} and \cite{AdC:RINAM} because the Lagrange multiplier defines a Neumann boundary condition on the interface, which allows us to \REV{obtain} well-posed subdomain equations that can be solved independently.

\REV{The main contribution of this paper is the formulation and analysis of an} alternative approach which utilizes a \emph{composite} reduced basis, comprising  independently constructed ROM bases for the interfacial and interior DoFs, instead of the conventional full subdomain reduced basis. To couple two ROMs across an interface, we then use the \emph{interface} part of the composite basis as a reduced order Lagrange multiplier space to enforce the interface conditions. Our analysis reveals that this approach leads to a provably non-singular Schur complement, independent of the underlying mesh size and/or composite reduced basis dimension. For the coupling of a ROM to a FOM, represented by a finite element model (FEM), this analysis indicates that one can use either the interface part of the composite basis from the ROM side or the interface finite element space from the FEM side of the interface as a Lagrange multiplier. 
We provide numerical results that corroborate numerically our theoretical findings. Results are shown on a two-dimensional (2D) time-dependent advection-diffusion  problem in the advection-dominated (high P\'{e}clet) regime.

\subsection{Related work}  \label{sec:lit_overview}

During the past two decades, the idea of coupling projection-based ROMs
with each other and with FOMs has been explored by a number of authors.
The bulk of the literature presents ROM-ROM or ROM-FOM coupling as a means for ``gluing" or
``tiling" these ROMs and/or FOMs together.
The focus is hence primarily on using domain decomposition (DD) as 
a vehicle to \emph{improve the efficiency of model order 
reduction} (MOR) for extreme scale problems and decomposable problems. 
The coupling approaches in the literature fall into roughly two categories: (1) monolithic 
coupling methods, and (2) iterative coupling methods.  
\REV{We succinctly review the literature on both method categories in Sections \ref{sec:monolithic} and \ref{sec:iterative}, respectively, and then in Section \ref{sec:organization} we highlight the key distinctions and contributions of this work.}



\subsubsection{Monolithic coupling methods} \label{sec:monolithic}

The majority of monolithic coupling methods in the MOR
community employ Lagrange multipliers to enforce compatibility
constraints.  Among the earliest works exploring DD 
to perform \REV{coupling of POD-based ROMs in an effort to improve the predictive 
accuracy}
 is the work of Lucia \textit{et al.} \cite{Lucia:2003}. 
Another early monolithic method 
for DD-based coupling 
of ROMs, this time constructed using the Reduced Basis Element (RBE) method,
is the work Maday \textit{et al.} \cite{AdC:Maday_02_CRM,AdC:Maday_04_SISC}. 
\REV{These methods are different from ours in that they rely} on
Lagrange multipliers represented by low-order polynomials \REV{(vs. POD modes)}
for imposing compatibility in a mortar-type method that ``glues" together non-overlapping subdomains. 
In \cite{Wicke_2009}, Wicke \textit{et al.}
present an approach for stitching together ``composable" ROM ``tiles", 
precomputed given specific boundary conditions, with the promise that the tiles can be assembled
in arbitrary ways at runtime. 
\REV{Continuity between tiles is enforced by duplicating the DoFs on the interfaces and constraining their normal components to be equal.}
\REV{ The Reduced basis method  with DD and Finite elements (RDF) \cite{AdC:Iapichino_16_CMA} is a conceptually similar, non-overlapping DD approach for gluing together networks of repetitive blocks. RDF uses standard finite element bases on the interfaces between the ROM domains to both enforce continuity and provide a sort of finite element enrichment.}
%
%
In \cite{AdC:Hoang}, Hoang \textit{et al.} present an algebraically non-overlapping method for coupling Least-Squares Petrov-Galerkin (LSPG)
ROMs with each other. 
\REV{This approach shares some commonality with the method developed herein, in that it 
considers several different types of subdomain and interface bases.}
\REV{Unlike our approach, \cite{AdC:Hoang} explores the use of both strong or weak compatibility}
conditions imposed at the
subdomain interfaces using Lagrange multipliers.  
%


It is also possible to \REV{effect} monolithic couplings without Lagrange multipliers through a judicious construction of the underlying discrete solution spaces. \REV{While these formulations are fundamentally different from our Lagrange multiplier-based 
approach, we succinctly review several methods falling into this category here for completeness.} 
\REV{The works  \cite{AdC:Iapichino_16_CMA} and \cite{Baiges_CMAME_2013}
propose monolithic DD-based strategies for ROM-ROM and
FOM-ROM coupling via local reduced order bases, which are carefully constructed 
to ensure automatic solution continuity across different subdomains.}
Another recent work that accomplishes \REV{monolithic} ROM-FOM coupling 
without Lagrange multipliers is \cite{Riffaud:2021}.
\REV{Unlike our approach, which does not require that any specific discretization 
method be used to discretize the governing PDE in space, the 
method in \cite{Riffaud:2021} is based on a discontinuous Galerkin (DG) formulation,
in which coupling is achieved through the definition of
numerical fluxes at discrete cell boundaries.}



It is interesting to remark that several DD-based ROM-ROM and ROM-FOM 
coupling methods with on-the-fly basis and/or DD adaptation have been
proposed in recent years.  \REV{While online model adaptation goes beyond the scope 
of the present manuscript, it may be considered in a future publication.}
In \cite{Corigliano:2015} and \cite{Corigliano:2013}, Corigliano \textit{et al.} develop a non-overlapping Lagrange 
multiplier-based coupling method for nonlinear elasto-plastic and multi-physics problems, 
\REV{in which  
on-the-fly ROM adaptation and ROM/FOM switching is performed through a plastic check during the reduced analysis.}
\REV{
Hybrid ROM-FOM coupling in the context of solid mechanics applications
is considered also in \cite{Kerfriden:2012, Kerfriden:2013}, where   
 a local/global
model reduction strategy for the simulation of quasi-brittle fracture is developed.}
An adaptive sub-structuring (domain decomposition)
non-overlapping approach for ROM-ROM and ROM-FOM coupling in solid mechanics 
is also presented in \cite{Radermacher:2014}.  
This method 
not only enables on-the-fly adaptation of the ROM basis, but also on-the-fly
substructuring/DD changes.
\REV{Furthermore, in} the pre-print \cite{Huang:2022}, Huang \textit{et al.} 
develop a component-based modeling
framework that can flexibly integrate ROMs and FOMs 
for different components or domain decompositions, 
towards modeling accuracy and efficiency for complex, large-scale
combustion problems. It is demonstrated that accuracy 
can be enhanced by incorporating basis adaptation ideas from 
\cite{Peherstorfer:2020, Huang:2023}. 

Finally, it is worth mentioning that another recent direction for hybrid ROM-FOM and ROM-ROM coupling 
involves the integration of ideas from machine learning into the coupling formulation.  
For example, in \cite{Ahmed:2021}, 
Ahmed \textit{et al.} present a hybrid ROM-FOM
approach in which 
a long short-term memory
network is introduced at the interface and subsequently used to perform the 
multi-model coupling.

\subsubsection{Iterative coupling methods} \label{sec:iterative}

\REV{While iterative coupling methods are fundamentally different from the monolithic couplings developed
in the present work, we overview several recent efforts falling into the iterative coupling category here 
for completeness.}
The majority of iterative methods for ROM-ROM and ROM-FOM coupling 
are based on the Schwarz alternating method \cite{Schwarz:1870}. 
Iterative coupling methods have the advantage that they are often 
less intrusive to implement in existing high-performance computing (HPC) codes \cite{Mota:2017, Mota:2022}; 
however, the methods' 
iterative nature can add to the total CPU time required to complete a simulation.
 
\REV{Among the earliest Schwarz-based DD approaches for coupling FOMs with ROMs is the work of Buffoni \textit{et al.}
 \cite{Buffoni:2007}, which focuses on Galerkin-free POD ROMs
developed for the Laplace equation and the compressible Euler equations.}
%
%
\REV{Galerkin-free FOM-ROM and ROM-ROM couplings are also considered by} Cinquegrana \textit{et al.}
\cite{Cinquegrana:2011} and Bergmann \textit{et al.} \cite{Bergmann_13_HAL}. 
The former approach \cite{Cinquegrana:2011} 
\REV{considers} overlapping DD in the context
of a Schwarz-like iteration scheme, but, unlike our approach, requires matching meshes at
the subdomain interfaces. 
The latter approach \cite{Bergmann_13_HAL}, termed zonal Galerkin-free POD, 
defines an optimization problem which minimizes the difference between the POD reconstruction
and its corresponding FOM solution in the overlapping region between a ROM and a FOM
domain.
A true POD-Greedy/Galerkin non-overlapping
Schwarz method for the coupling of projection-based ROMs developed for the specific case
of symmetric elliptic PDEs is presented by Maier \textit{et al.} in \cite{Maier:2014}.  
A more 
general POD/Galerkin ROM-ROM and ROM-FOM coupling method based on the 
overlapping or non-overlapping Schwarz method is developed in \cite{Barnett:2022}.
This work is an extension of an alternating Schwarz-based
 concurrent multi-scale FOM-FOM coupling method developed earlier
in \cite{Mota:2017, Mota:2022}.    
\REV{A recent work by Iollo \textit{et al.} \cite{Iollo:2023} on component-based model reduction via overlapping alternating Schwarz shows that, for  linear elliptic PDEs, the latter can be interpreted as an optimization-based coupling \cite{Du_00_SINUM}. By solving the optimization problem directly\REVpk{,} one obtains a ``One-Shot Schwarz" procedure.}

Alternative ROM-FOM approaches include the domain decomposition non-intrusive reduced-order model (DDNIR)
\cite{Xiao:2019}.  
 Here, a radial basis function
interpolation method is used to construct a set of hypersurfaces for 
iterative solution transfer between neighboring subdomains.

While \REV{our} focus is restricted to projection-based ROMs, it is worth noting that
the Schwarz alternating method has recently been extended to \REV{DD} coupling of Physics
Informed Neural Networks (PINNs) in \cite{LiDeepDDM, LiD3M}. The methods
proposed in these works, termed D3M \cite{LiD3M} and DeepDDM \cite{LiDeepDDM}, inherit the benefits of 
DD-based ROM-ROM couplings, but are developed primarily for the purpose of improving the
efficiency of the neural network training process and reducing the risk of over-fitting, 
both of which
are due to the global nature of the neural network ``basis functions".  
The Schwarz alternating method has also been used for online coupling of independently pre-trained subdomain-localized neural network-based 
models, e.g. in \cite{Wang:2022}, which develops a transferable framework for solving boundary value problems 
(BVPs) via deep neural networks that can
be trained once and used forever for various unseen domains and boundary conditions (BCs).  

\subsection{Differentiating contributions and organization} \label{sec:organization} 


The partitioned schemes for coupled ROM-ROM and ROM-FOM problems in this paper have some commonality with both the monolithic Lagrange multiplier-based coupling approaches described succinctly in Section \ref{sec:monolithic}, e.g., \cite{AdC:Maday_02_CRM,AdC:Maday_04_SISC,AdC:Iapichino_16_CMA,AdC:Huynh_13_M2AN, AdC:Hoang}, and the iterative coupling schemes summarized in Section \ref{sec:iterative} \REV{in the sense that they all focus on couplings between ROMs and/or ROMs and FOMs.} However, the work presented here differs from both of these types of methods in several important ways. 

Compared to the methods in Section \ref{sec:monolithic}, our main focus is \REV{on \emph{improving the simulation efficiency} for both single and multi-physics problems through \emph{explicit partitioned solution} of their coupled ROM-ROM and ROM-FOM formulations, rather than on \emph{improving the efficiency of the model order reduction} process  through domain-decomposition ideas.}
%
%
Second, although both DD-ROM methods and our partitioned scheme utilize Lagrange multipliers, the variational setting for the coupled ROM-ROM and ROM-FOM formulations in this paper differs from the one in a typical DD-based ROM\REVpk{; this is because ours} is designed to provide a provably non-singular Schur complement when the coupled ROM-ROM or ROM-FOM problems are discretized by an explicit time integrator. As explained in Remark \ref{rem:DD}, this makes our setting more ``forgiving'' to variational crimes and results in Schur complements whose conditioning is independent of the mesh size underpinning the FOM or the size of the composite reduced basis defining the ROM. 

Insofar as the iterative coupling methods are concerned, both our partitioned schemes and the methods utilizing the Schwarz alternating algorithm perform independent solves of decoupled subdomain problems. Additionally,  in both cases, the decoupling is effected by specifying boundary conditions on the interface that ``close'' the subdomain equations and make their independent solution possible. However, in the case of the Schwarz alternating method, one usually starts with an initial guess for the boundary condition and iterates until the subdomain solutions have converged sufficiently. The rate of convergence generally depends on the size of the overlap between the subdomains, which makes this type of methods more difficult to extend  to multiphysics problems where different subdomains may have different sets of governing equations. 
In contrast, our partitioned schemes define the interface boundary conditions by solving a Schur complement equation that provides a highly accurate estimate of the interface flux. Conceptually, this approach is similar to the techniques in \cite{Carey_85_CMAME} where one solves an additional problem to obtain more accurate approximation of the boundary flux than afforded by simply inserting the finite element solution into the flux function.

The remainder of this paper is organized as follows. 
\REV{Section \ref{AdC:sec:notation} introduces the bulk of the notation used in the paper.}
In Section \ref{AdC:sec:model}, we describe our model transmission problem (a transient \REV{scalar}
advection-diffusion problem) and define the coupled FOM-FOM formulation\REV{\REVpk{, }which provides the basis for the development of the coupled ROM-ROM and ROM-FOM problems.}
\REV{For convenience, in Section \ref{sec:IVR}, we also briefly summarize the partitioned Implicit Value Recovery (IVR) scheme \cite{AdC:CAMWA}.}
 Section \ref{sec:POD} overviews 
projection-based model reduction using the POD/Galerkin method and introduces the reduced order basis spaces that will be used in the paper.
Section \ref{sec:ROM-ROM} is the core of this paper, where we use the composite reduced basis idea to formulate the  IVR scheme for the partitioned solution of two ROMs coupled across an interface. Next, in Section \ref{sec:ROM-FOM}, we describe how the IVR scheme can be extended to the partitioned solution of coupled ROM-FOM problems.  
In Section \ref{sec:analysis}, we use variational techniques to prove that the coupled ROM-ROM and ROM-FOM formulations have non-singular Schur complements, which is the key prerequisite for the extension of the IVR scheme to these couplings. 
Numerical results demonstrating the proposed scheme's accuracy 
and efficiency for reproductive as well as predictive ROMs are presented in Section \ref{sec:num}.  
Finally, conclusions are offered in Section \ref{sec:conc}.

%% file: sec_Notation.tex
\section{\REV{Notation}}\label{AdC:sec:notation}
\begin{figure}[t!]
  \begin{center}
    \includegraphics[width=0.45\textwidth]{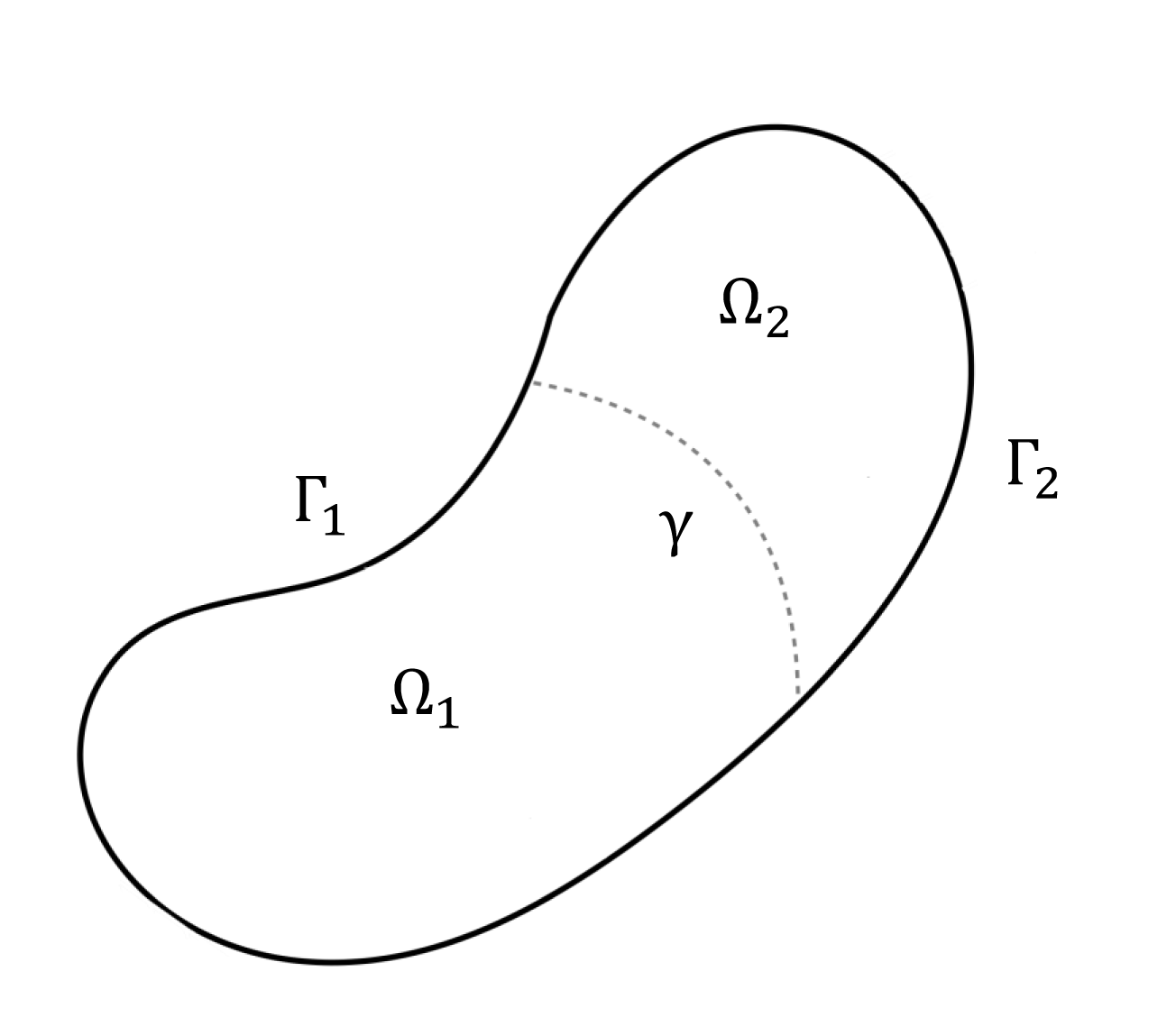}
  \end{center}
\vspace{-3ex}
\caption{Example partitioning of 2D domain $\Omega$ into two non-overlapping subdomains, $\Omega_1$ and 
$\Omega_2$, with interface boundary $\gamma$.} \label{fig:dd}
\end{figure}
For the convenience of the reader and ease of reference, this section summarizes the bulk of the notation used throughout the paper.
We consider a bounded region $\Omega \in \mathbb{R}^\nu$, $\nu = 2,3$ divided into two\footnote{This configuration of the model transmission problem possesses all the characteristics relevant for the development of the partitioned schemes. Extension of these schemes to transmission problems with more than two domains is conceptually similar to the case of two subdomains.}
 non-overlapping subdomains $\Omega_1$ and $\Omega_2$ by an interface $\gamma$, as shown in Figure \ref{fig:dd}. Without a loss of generality, we assume that the unit normal $\bm{n}_\gamma$ points towards $\Omega_2$, and set $\Gamma_i := \partial \Omega_i  \backslash \gamma$, $i=1,2$. 

We use the standard notation $L^2(\Omega_i)$ for the space of all square integrable functions in $\Omega_i$ with norm and inner product denoted by $\|\cdot\|_{0,\Omega_i}$ and $(\cdot,\cdot)_{0,\Omega_i}$, respectively. Likewise, ${H}^1(\Omega_i)$ will denote the Sobolev space of all square integrable scalar functions on $\Omega_i$ whose first derivatives are also square integrable and ${H}^1_{D}(\Omega_i)$ will be the subspace of  ${H}^1(\Omega_i)$ whose elements  vanish on  $\Gamma_i$. Restrictions of ${H}^1(\Omega_i)$ functions to $\gamma$ form the fractional Sobolev space $H^{1/2}(\gamma)$, with dual $H^{-1/2}(\gamma)$ and duality pairing $\langle\cdot,\cdot\rangle_{\gamma}$.

In this paper we consider quasi-uniform partitions $\Omega^h_i$  of  $\Omega_i$ with mesh parameter $h_i$, vertices  $\bm{x}_{i,r}$, and elements $K_{i,s}$. We assume that each subdomain is meshed independently and denote the finite element partition of the interface induced by $\Omega^h_i$ as $\gamma^h_i$. To avoid technicalities that are not germane to the subject of this paper, we shall assume that $\gamma^h_1$ and $\gamma^h_2$ are spatially coincident; however, their vertices are not required to match. Similarly, $\Gamma^h_i$ will be the finite element partition of the Dirichlet boundary induced by $\Omega^h_i$. 

\begin{remark}\label{rem:dbc}
If a node $\bm{x}_{i,r}$ lies on both $\Gamma_i$ and $\gamma$, it is treated as belonging to the finite element partition $\Gamma^h_i$ of the Dirichlet boundary rather than that of the interface; see Figure~\ref{fig:nodes}. The reason for this is that the DoFs at these nodes are assigned the nodal values of the boundary data and they contribute to the right-hand sides of the discrete equations, i.e., these DoFs are not unknown solution coefficients. 
\end{remark}

\begin{figure}[t!]
  \begin{center}
    \includegraphics[width=0.8\textwidth]{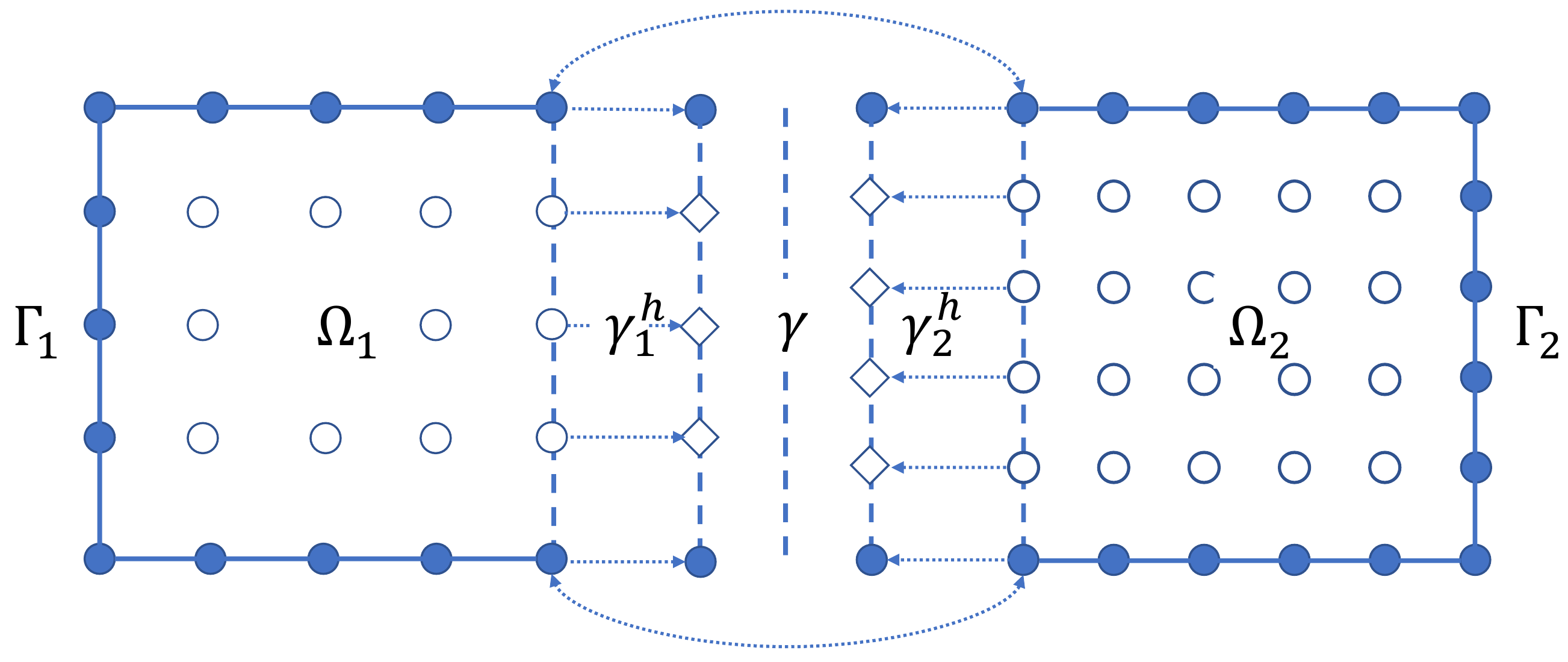}
  \end{center}
  \vspace{-3ex}
\caption{Node partitions of the finite element mesh. Solid disks depict Dirichlet nodes forming the space $S^h_{i,\Gamma}$; circles depict interior and interface nodes forming the space $S^h_{i,D}$; diamonds depict interface nodes forming the spaces $S^h_{i,\gamma}$ and $G^h_i$. Arrows indicate duplicate nodes that correspond to the same DoFs.} \label{fig:nodes}
\end{figure}

In what follows, $S_{i}^h$ will denote the lowest-order  nodal $C^0$ conforming finite element  subspace of ${H}^1(\Omega_i)$, defined with respect to $\Omega^h_i$; see, e.g., \cite{Ciarlet_02_BOOK}. We equip $S_{i}^h$ with a standard Lagrangian basis $\{N_{i,r}\}$, i.e., $N_{i,r}(\bm{x}_{i,s}) = \delta_{rs}$, where $\delta_{rs}$ is the Kronecker $\delta$-symbol.
We denote the subspace of $S_{i}^h$ comprising all finite element functions that vanish on $\Gamma_i$ by $S_{i,D}^h$. This subspace is a conforming approximation of the Sobolev space ${H}^1_{D}(\Omega_i)$. 
Let $n_{i,\Gamma}$, $n_{i,\gamma}$ and $n_{i,0}$ denote the numbers of mesh nodes on the Dirichlet boundary $\Gamma_i$, the interface $\gamma$, and the interior of the subdomain $\Omega_i$, respectively; see Figure~\ref{fig:nodes}.  
Thus, $n_{i,D}= n_{i,0}+n_{i,\gamma}$ and $n_{i} = n_{i,D} + n_{i,\Gamma} = n_{i,0}+n_{i,\gamma}+n_{i,\Gamma}$.

The coefficients of a finite element function $u^h_i\in S_{i}^h$ form a vector $\bm{u}_i\in\mathbb{R}^{n_i}$. 
Without loss of generality, we can assume that the nodes of $\Omega^h_i$ are numbered such that this coefficient vector has the form  
$\bm{u}_i = (\bm{u}_{i,\gamma},\bm{u}_{i,0},\bm{u}_{i,\Gamma})$, where 
$\bm{u}_{i,\gamma}\in\mathbb{R}^{n_{i,\gamma}}$, 
$\bm{u}_{i,0}\in\mathbb{R}^{n_{i,0}}$, and 
$\bm{u}_{i,\Gamma}\in\mathbb{R}^{n_{i,\Gamma}}$
are vectors of interface, interior, and Dirichlet coefficients, respectively. With this convention,
 it is easy to see that the coefficient vector of a finite element function $u^h_{i,D}\in S^h_{i,D}$ has the form $\bm{u}_{i,D}=(\bm{u}_{i,\gamma},\bm{u}_{i,0},\bm{0})$.

Besides $S^h_{i,D}$, we shall need three additional subspaces of $S^h_i$. The first one contains all finite element functions whose coefficient vectors have the form $(\bm{u}_{i,\gamma},\bm{0},\bm{0})$, i.e., they vanish at all but the interface nodes. We denote this space by $S^h_{i,\gamma}$ and call it the \emph{interface} part of $S^h_i$. The finite element functions in the second subspace have coefficient vectors $(\bm{0},\bm{u}_{i,0},\bm{0})$. These functions are identically zero on both the interface and the Dirichlet boundary. We term this subspace the \emph{interior} part of $S^h_i$ and denote it by $S^h_{i,0}$. The last subspace of $S^h_i$ contains all finite element functions with coefficients $(\bm{0},\bm{0},\bm{u}_{i,\Gamma})$. These functions vanish at all nodes except those on the Dirichlet boundary $\Gamma_i$. We denote this space by $S^h_{i,\Gamma}$ and call it the \emph{boundary} part of $S^h_i$. Note that  
$$
S^h_i = S_{i,\gamma}^h\cup S_{i,0}^h\cup S^h_{i,\Gamma}\
\quad\mbox{and}\quad
S^h_{i,D} = S_{i,\gamma}^h\cup S_{i,0}^h\,.
$$

Formally, the subspaces  $S^h_{i,D}$, $S^h_{i,\gamma}$, $S^h_{i,0}$,  and $S^h_{i,\Gamma}$ have the same dimension $n_i$ as their parent space $S^h_i$. However, when discussing the assembled algebraic forms of the 
weak formulations, it will be more convenient to remove the zero blocks from the coefficient vectors and associate  $u^h_{i,D}\in S^h_{i,D}$, $u^h_{i,\gamma}\in S^h_{i,\gamma}$, $u^h_{i,0}\in S^h_{i,0}$,  and $u^h_{i,\Gamma}\in S^h_{i,\Gamma}$ with coefficient vectors 
$\bm{u}_{i,D}\in\mathbb{R}^{n_{i,D}}$, 
$\bm{u}_{i,\gamma}\in\mathbb{R}^{n_{i,\gamma}}$, 
$\bm{u}_{i,0}\in\mathbb{R}^{n_{i,0}}$, and 
$\bm{u}_{i,\Gamma}\in\mathbb{R}^{n_{i,\Gamma}}$, respectively. 
In this context, we will refer to $n_{i,D}$, $n_{i,\gamma}$, $n_{i,0}$, and $n_{i,\Gamma}$ as the \emph{effective} dimensions of their respective finite element subspaces. 
Finally, we define the \emph{induced} interface finite element space $G_i^h$ as the trace of the interface part of $S^h_i$, i.e., $G_{i}^h = S_{i,\gamma}^h \big|_\gamma$.  Since $dim\, G_{i}^h = dim\, S_{i,\gamma}^h = n_{i,\gamma}$, the coefficient spaces of $G_{i}^h$ and $S_{i,\gamma}^h$ are isomorphic, i.e., a vector $\bm{c}\in \mathbb{R}^{n_{i,\gamma}}$ can be mapped to a function $u^{h,\bm{c}}_{i,\gamma}\in S_{i,\gamma}^h$, or a function $\lambda^{h,\bm{c}}_{i}\in G^h_i$.

%% file: sec_ModelProb.tex
\section{\REV{Model problem and its coupled FOM-FOM formulation}}\label{AdC:sec:model}
This section defines the model transmission problem, the associated weak coupled formulation and its semi-discretization in space\footnote{We emphasize that, while the high-fidelity models herein are assumed to be constructed
using the finite element method, our partitioned solution approach is easily extensible to FOMs 
constructed using alternate discretization approaches such as finite volume and finite difference methods.} by finite elements.
\REV{\REVpk{The} coupled FOM-FOM problem is key to the development of partitioned schemes for coupled ROM-ROM and ROM-FOM problems in this paper. For completeness, Section \ref{sec:IVR} briefly reviews the IVR scheme for the FOM-FOM problem.}

\REV{We consider the scalar} advection-diffusion transmission equation
\begin{align}\label{AdC:strongForm}
\begin{split}
\dot{u}_i - \nabla \cdot F_i (u_i) &= f_i \hspace{5mm} \text{ on } \Omega_i \times [0,T] \\
u_i &= g_i     \hspace{5mm} \text{ on } \Gamma_i \times [0,T] \\
u_i & = u_{i,0}\hspace{2mm} \text{ on } \Omega_i, \hspace{2mm} t=0
\end{split}
 \;\quad i=1,2, 
\end{align}
where the over-dot notation denotes differentiation in time, the unknown $u_i:=u_i(\bm{x},t)$ is a 
scalar field, $F_i(u_i) = \kappa_i \nabla u_i - \mathbf{a} u_i$ is the total flux function, $f_i:=f_i(\bm{x},t)$ is a source term, $g_i:=g_i(\bm{x},t)$ is prescribed boundary data, $u_{i,0}:=u_{i,0}(\bm{x})$, is a prescribed initial condition, $\kappa_i:=\kappa_i(\bm{x},t) > 0$ is the diffusion coefficient in $\Omega_i$, and $\mathbf{a}:=\mathbf{a}(\bm{x},t)$ is the advection field.  
%
%
Along the interface $\gamma$, we enforce continuity of the ``velocities'' $\dot{u}_i$ and continuity of the total flux, giving rise to the following interface conditions:
\begin{align}\label{AdC:interfaceConditions}
\dot{u}_1(\bm{x},t) - \dot{u}_2(\bm{x},t) = 0 \hspace{2mm} \text{ and } F_1(\bm{x},t) \cdot \mathbf{n}_\gamma = F_2(\bm{x},t) \cdot \mathbf{n}_\gamma \hspace{2mm} \text { on } \gamma \times [0,T].
\end{align}
\REV{We choose this problem because it allows us to conveniently demonstrate the partitioned methods developed in this paper in both simulation contexts by setting $\kappa_1=\kappa_2$ and  $\kappa_1\neq\kappa_2$, respectively.}

\begin{remark}\label{rem:index1}
The use of ``velocity'' continuity  in lieu of the more conventional continuity of the states coupling condition is required to obtain a  coupled \REV{FOM-FOM} formulation that, under some conditions on the Lagrange multiplier space, is a Hessenberg Index-1 Differential Algebraic Equation (DAE) \cite{Ascher_98_BOOK}. In such DAEs the algebraic variable  (the Lagrange multiplier) is an implicit function of the differential variables (the subdomain states). This fact is at the core of the IVR formulation as it allows one to solve for the Lagrange multiplier in terms of the subdomain states. It also motivates the term ``implicit'' in the name of the scheme. 
\end{remark}

We define the coupled weak formulation of \eqref{AdC:strongForm} by using a Lagrange multiplier to enforce the first constraint in  \eqref{AdC:interfaceConditions}. To \REV{that end,} we write the solution as 
 $u_{i} = u_{i,D}+g_{i}$ where $u_{i,D}\in {H}^1_{D}(\Omega_i)$ is the interior (unknown) component of $u_{i}$ and, with some abuse of notation, $g_i\in {H}^1(\Omega_i)$ is a lifting of the boundary data. The weak form of \eqref{AdC:strongForm}--\eqref{AdC:interfaceConditions} is then given by the variational equation: \emph{seek} $\{u_{1,D}, u_{2,D}, \lambda\}: (0,T]\mapsto {H}^1_{D}(\Omega_1)\times {H}^1_{D}(\Omega_2)\times {H}^{-1/2}(\gamma)$  \emph{such that}  $u_i  = u_{i,0}$ for $t=0$, $i=1,2$, and for $t>0$
\begin{equation}\label{eq:weaF_form_LM}
\begin{array}{rcll}
   \left( \dot{u}_{1,D} , v_1 \right)_{0,\Omega_1} + \left \langle \lambda,  v_1 \right\rangle_{\gamma}
   & =& \left( f_1, v_1 \right)_{0,\Omega_1} - \left( F_1(u_{1,D}), \nabla v_1\right)_{0,\Omega_1} 
   - Q_1(\dot{g}_{1},g_1;v_1)
   &\quad \forall v_1 \in H^1_{\Gamma}(\Omega_1)
   \\[1ex]
   \left( \dot{u}_{2,D}, v_2 \right)_{0,\Omega_2} - \left \langle \lambda,  v_2 \right\rangle_{\gamma}
   & = &\left( f_2, v_2 \right)_{0,\Omega_2} - \left( F_2(u_{2,D}), \nabla v_2\right)_{0,\Omega_2}  
  - Q_2(\dot{g}_{2},g_2;v_2)
   &\quad \forall v_2 \in H^1_{\Gamma}(\Omega_2)
   \\[1ex]
   \left\langle \dot{u}_{1,D}- \dot{u}_{2,D}, \mu \right \rangle_\gamma  
   &= &  \left\langle \dot{g}_{1}- \dot{g}_{2}, \mu \right \rangle_\gamma  
   &\quad \forall \mu \in H^{-1/2}(\gamma).
 \end{array}
\end{equation}
In \eqref{eq:weaF_form_LM}, the $Q_i$, for $i=1,2$, denote bilinear forms producing the contributions from the boundary data to the right-hand sides of the subdomain equations. The variational equation \eqref{eq:weaF_form_LM} is of the mixed type.  Using the theory in \cite{Brezzi_74_RAIRO}, one can show that  \eqref{eq:weaF_form_LM} is well-posed. 

\begin{remark}\label{rem:interface-BC}
Since the Dirichlet data are supposed to satisfy the first coupling condition in \eqref{AdC:interfaceConditions}, the boundary data contribution $\left\langle \dot{g}_{1}- \dot{g}_{2}, \mu \right \rangle_\gamma$ to the right-hand side of the constraint equation in \eqref{eq:weaF_form_LM} is identically zero. However, in general, a discretized version of this term will not be identically zero and boundary contributions need to be properly accounted for in the assembled discrete problem.  
\end{remark}

\subsection{\REV{The coupled FOM-FOM problem}}\label{sec:FOM-FOM}
\REV{To obtain the coupled FOM-FOM problem, we discretize the weak formulation \eqref{eq:weaF_form_LM} in space by approximating the subdomain states $\{u_1,u_2\}$ and the Lagrange multiplier $\lambda$ with the finite element spaces $V^h = S^h_{1}\times S^h_{2}$ and  $W^h = G^h_k$, $k=1$ or $k=2$, respectively.} This choice of $W^h$ is common for mortar element methods  \cite{Bernardi_05_GAMM,Bernardi_94_INPROC} and it also ensures  the well-posedness of the IVR scheme for the FOM-FOM problem.
%
%

To handle the Dirichlet boundary conditions,
 we proceed similarly to  \eqref{eq:weaF_form_LM} and write the finite element solution as $u^h_i = u^h_{i,D}+ g^h_{i}$, where $u^h_{i,D}\in S^h_{i,D}$ is the unknown part of $u^h_i$ and $g^h_{i}\in S^h_{i,\Gamma}$ is the finite element interpolant of the boundary data. Thus, the coefficient vector of $u^h_i$ is given by $\bm{u}_i = (\bm{u}_{i,D},\bm{g}_i)$ where $\bm{u}_{i,D}\in \mathbb{R}^{n_{i,D}}$  is the coefficient vector containing the unknown nodal values of the solution, and $\bm{g}_{i}\in \mathbb{R}^{n_{i,\Gamma}}$ is a coefficient vector containing the known nodal values of the boundary data $g_i$.
To obtain the \REV{coupled FOM-FOM} problem, we approximate the duality pairing $\langle\cdot,\cdot\rangle_{\gamma}$ by the $L^2$ inner product\footnote{Remark \ref{rem:DD} provides some additional information about this choice for the interface inner product.} $(\cdot,\cdot)_{0,\gamma}$ and restrict \eqref{eq:weaF_form_LM} to $V^h$ and $W^h$. The resulting problem can be written in the following compact matrix form:
\begin{align}\label{AdC:eq:Idx1System}
\begin{bmatrix}
M_{1,D} & 0 & G_{1,D}^T \\[1ex]
0 & M_{2,D} & -G_{2,D}^T\\[1ex]
G_{1,D} & -G_{2,D} & 0
\end{bmatrix}
\begin{bmatrix}
\bm{\dot{u}}_{1,D} \\[1ex]
\bm{\dot{u}}_{2,D} \\[1ex]
\bm{\lambda}
\end{bmatrix}
= 
\begin{bmatrix}
\bm{f}_{1,D} - F_{1,D} \bm{u}_{1,D} - Q_{1,\Gamma}( \dot{\bm{g}}_{1}, \bm{g}_{1}) \\[1ex]
\bm{f}_{2,D} - F_{2,D} \bm{u}_{2,D} - Q_{2,\Gamma}( \dot{\bm{g}}_{2}, \bm{g}_{2}) \\[1ex]
- Q_{\gamma,\Gamma}(\dot{\bm{g}}_{1},\bm{\dot{g}}_{2})
\end{bmatrix},
\end{align}
where, for $i=1,2$, $M_{i,D}$ and $F_{i,D}$ are $n_{i,D}\times n_{i,D}$ mass and flux matrices, respectively,  
$G_{i,D}$ are $n_{k,\gamma}\times n_{i,D}$ matrices defining the algebraic form of the ``velocity'' constraint in \eqref{AdC:interfaceConditions}, $\bm{\lambda}\in\mathbb{R}^{n_{k,\gamma}}$ is the coefficient vector of the discrete Lagrange multiplier, and $\bm{f}_{i,D}\in \mathbb{R}^{n_{i,D}}$ is the coefficient vector of the source term.
The terms 
$$
Q_{i,\Gamma}( \dot{\bm{g}}_{i}, \bm{g}_{i})= M_{i,\Gamma} \dot{\bm{g}}_{i} + F_{i,\Gamma} \bm{g}_{i} , \quad i=1,2
\quad\mbox{and}\quad
Q_{\gamma,\Gamma}(\dot{\bm{g}}_{1},\bm{\dot{g}}_{2})
 = G_{1,\Gamma}\dot{\bm{g}}_{1} - G_{2,\Gamma} \bm{\dot{g}}_{2},
$$
where $M_{i,\Gamma}$ and $F_{i,\Gamma}$ are $n_{i,D}\times n_{i,\Gamma}$ ``partial'' mass and flux matrices, and $G_{i,\Gamma}$ are $n_{k,\gamma}\times n_{i,\Gamma}$ ``partial'' constraint matrices, provide the contributions from the boundary data interpolants to the right-hand sides of the subdomain equations. Note that the matrix blocks in \eqref{AdC:eq:Idx1System} are dimensioned using the effective dimensions of the finite element spaces. 

\begin{remark}\label{rem:partial-mat}
The ``partial'' constraint matrices $G_{i,\Gamma}$ are, in general, very sparse. For example, in two-dimensions, each $G_{i,\Gamma}$ will have at most two non-zero elements corresponding to the two Dirichlet nodes at the endpoints of $\gamma$; see Figure~\ref{fig:nodes}. Although $\dot{\bm{g}}_1 = \dot{\bm{g}}_2$ at these nodes, the integrals of the finite element interpolants $\dot{g}^h_1$ and $\dot{g}^h_2$ against the Lagrange multiplier basis functions will not be identical unless the nodes adjacent to the endpoints of $\gamma$ match on both sides of the interface. This is to be contrasted with the continuous problem where the Dirichlet data does not contribute to the constraint equation; see Remark \ref{rem:interface-BC}.
\end{remark}

\subsection{\REV{The Implicit Value Recovery (IVR) scheme for the coupled FOM-FOM problem}}\label{sec:IVR}
\REV{In this section, we briefly review the IVR scheme \cite{AdC:CAMWA} for the coupled FOM-FOM problem \eqref{AdC:eq:Idx1System}. This scheme solves the linear system}
\begin{equation}\label{eq:FOM-FOM-flux}
\begin{array}{rcll}
S\bm{\lambda} & =& G_{1,D} M_{1,D}^{-1} \left[\bm{f}_{1,D} - F_{1,D} \bm{u}_{1,D} - Q_{1,\Gamma}( \dot{\bm{g}}_{1}, \bm{g}_{1})\right]\\
&&\; - G_{2,D} M_{2,D}^{-1} \left[\bm{f}_{2,D} - F_{2,D} \bm{u}_{2,D} - Q_{2,\Gamma}( \dot{\bm{g}}_{2}, \bm{g}_{2}) \right] + Q_{\gamma,\Gamma}(\bm{\dot{g}}_1, \bm{\dot{g}}_2),
\end{array}
\end{equation}
\REV{where}
\begin{equation}\label{eq:IVR-schur}
S = G_{1,D}M^{-1}_{1,D}G^T_{1,D} + G_{2,D}M^{-1}_{2,D}G^T_{2,D}
\end{equation}
\REV{is the dual Schur complement of the matrix on the left hand side of \eqref{AdC:eq:Idx1System}.  
This is used 
to compute a highly accurate approximation of the interface flux\footnote{\REV{In contrast, ``loosely coupled'' partitioned schemes use the ``raw'' solution state from each side of the interface to specify boundary conditions that close each subdomain equation and make possible its independent solution. Mathematically, such schemes can be viewed as performing a single step of a non-overlapping alternating Schwarz iterative coupling procedure; see Section \ref{sec:iterative}. This is also the root cause for some of the stability and accuracy issues experienced by these methods.}} $\bm{\lambda}$, which then serves as a Neumann boundary condition for the subdomain equations.}
\REV{As a result, the well-posedness of the IVR scheme hinges on the invertibility of the Schur complement  matrix $S$. A sufficient condition for \eqref{eq:IVR-schur} to be symmetric and positive definite is that the transpose constraint matrix  has a full column rank. One can show that  if the Lagrange multiplier space $W^h$ is defined as in Section \ref{sec:FOM-FOM}, i.e., as the trace $G^h_k$ of the interface finite element space $S^h_{k,\gamma}$ on either of $\Omega_1$ or $\Omega_2$, the matrix $G^T =(G_{1,D}, -G_{2,D})^T$ does indeed have this property.} 
%

\REV{Assuming a non-singular Schur complement \eqref{eq:IVR-schur}, the IVR scheme for \eqref{AdC:strongForm} comprises the following two steps. First, one solves \eqref{eq:FOM-FOM-flux} for the Lagrange multiplier $\bm{\lambda}$ and eliminates it from \eqref{AdC:eq:Idx1System}. This reduces the coupled FOM-FOM problem to a coupled system of two ordinary differential equations (ODEs):}
\begin{align}\label{AdC:eq:Idx1ODESystem}
\begin{bmatrix}
M_{1,D} & 0 \\
0 & M_{2,D} \\
\end{bmatrix}
\begin{bmatrix}
\bm{\dot{u}}_{1,D} \\
\bm{\dot{u}}_{2,D} \\
\end{bmatrix}
= 
\begin{bmatrix}
\bm{f}_{1,D} - F_{1,D} \bm{u}_{1,D} - G^T_{1,D} \bm{\lambda}(\bm{u}_{1,D},\bm{u}_{2,D};\dot{\bm{g}}_{1},\dot{\bm{g}}_{2}) 
-  Q_{1,\Gamma}( \dot{\bm{g}}_{1}, \bm{g}_{1}) \\
\bm{f}_{2,D} - F_{2,D} \bm{u}_{2,D} +G^T_{2,D}   \bm{\lambda}(\bm{u}_{1,D},\bm{u}_{2,D};\dot{\bm{g}}_{1},\dot{\bm{g}}_{2}) -  Q_{2,\Gamma}( \dot{\bm{g}}_{2}, \bm{g}_{2}) \\
\end{bmatrix} \,.
\end{align}
The ODE sub-systems in \eqref{AdC:eq:Idx1ODESystem} define the associated \emph{subdomain} FOMs. 
It is easy to see that an explicit time discretization of \eqref{AdC:eq:Idx1ODESystem} decouples this problem and allows one to advance the solution to the next time step by solving the subdomain FOMs \emph{independently}; see \cite{AdC:CAMWA}. \REV{Thus, the second step of IVR consists of applying explicit time integrators to each subdomain FOM.} The subdomain time integrators are not required to be the same and they can also use different time steps over shared synchronization time intervals.

Because decoupling of \eqref{AdC:eq:Idx1ODESystem} is effected solely by \REV{explicit time integration, it is not accompanied by} any splitting errors as is the case with traditional loosely coupled partitioned schemes. In particular, the IVR scheme fully retains the stability and the accuracy properties of the underlying coupled problem. In fact, one can show that, for some settings, the IVR solution is identical to the solution of the coupled problem.

%% file: sec_POD.tex
\section{Projection-based model order reduction (MOR)} \label{sec:POD}
In this section, we briefly review the basic concepts of the MOR 
approach used in this paper to develop the IVR scheme for the partitioned solution of coupled problems involving
subdomain-local ROMs coupled to other subdomain-local ROMs or to FOMs. We then specialize some aspects of the generic MOR process to the model transmission problem that is the focus of this paper.

\subsection{A generic POD-based MOR workflow}\label{sec:pod-RB}
The approach for constructing a projection-based ROM consists of  two critical steps: (1) calculation of a reduced basis (RB), and (2) projection of the governing equations onto the reduced basis. These two steps are described succinctly in the following paragraphs.

\paragraph{Reduced basis construction via the POD}
One of the most popular approaches for calculating a reduced basis
 is the POD \cite{AdC:Holmes1996,
AdC:Sirovich1987}.  To discuss POD, consider a generic FOM given by
\begin{equation} \label{eq:genericFOM}
M \dot{\bm{u}} = \bm{f}(\bm{u}), 
\end{equation} 
where $M \in \mathbb{R}^{n\times n}$ and $\bm{u}, \bm{f} \in \mathbb{R}^{n}$.  The FOM \eqref{eq:genericFOM} can be thought of as resulting from a spatial discretization of some set of governing PDEs. 

To obtain the POD basis, one simulates \eqref{eq:genericFOM} and collects its solutions $\bm{u}^m$, $m=1,\ldots,r$ into an $n\times r$ snapshot matrix  $X.$ 
Typically, the snapshots $\bm{u}^m$ are taken to be the primary solution field at different times and/or 
different parameter values. POD works by first computing the singular value decomposition (SVD)
$X=U\Sigma V^T$ of the snapshot matrix. Then, one chooses a positive integer $0< d \le n$ that defines the accuracy of the reduced basis. The value of $d$ is typically selected using a ``snapshot energy" criterion,
where the ``snapshot energy'' is defined as 
\begin{align}\label{AdC:eq:snapEnergy}
	\mathcal{E}:=\frac{\sum_{i=1}^{d} \sigma_i^2}{\sum_{i=1}^n \sigma_i^2},
\end{align} 
with $\sigma_i$ denoting the $i^{th}$ singular value of $X$. Specifically, let $\delta$ be a desired threshold for the retention of the snapshot energy. The integer $d$ is then defined as the smallest integer such that 
\begin{equation}\label{eq:delta}
\sum_{i=1}^{d} \sigma_{i}^2 \geq \Big(1 - \delta \Big)\sum_{i=1}^{n} \sigma_{i}^2.
\end{equation}
Typically, one seeks a reduced basis that  captures 95\% or 99\% of the snapshot energy, i.e., $\mathcal{E} \approx 0.95$ or $\mathcal{E} \approx 0.99$. This corresponds to thresholds $\delta=0.05$ or $\delta=0.01$. 

Once $d$ is determined according to \eqref{eq:delta} the  $n\times d$ 
POD reduced basis  matrix, denoted herein by ${\REV{\Phi}}$, is defined by taking the first $d$ left singular vectors of $X$, i.e., the first $d$ columns of $U$. Construction of the POD basis 
can be interpreted as an approximation of the snapshot set $X$ by its truncated SVD: $X\approx  \widetilde{X} = {\REV{\Phi}}\widetilde{\Sigma}\widetilde{V}^T$.
In order to achieve a meaningful order reduction of the FOM, $d$ must be 
much smaller than the dimension $n$ of the FOM. 
We remark that this requires a 
sharp decay of the singular values, which holds for our model problem 
but is not true in general for problems with a slow decay of the so-called 
Kolmogorov $n$-width \cite{Greif:2019}; see, e.g., \cite{Abgrall_16_AMSES,Abgrall_18_IJNMF,Crisovan_19_JCAM}.

Once the reduced basis ${\REV{\Phi}}$ is calculated using the above workflow, the 
FOM solution $\bm{u}$ is approximated as a linear combination of these reduced basis modes and unknown 
time-dependent modal amplitudes $\widetilde{\bm{u}}\in\mathbb{R}^d$: 
\begin{equation} \label{eq:modal_decomp}
\bm{u}(t) \approx  \bar{\bm{u}} +
 {\REV{\Phi}} \widetilde{\bm{u}}(t).
\end{equation}
In \eqref{eq:modal_decomp}, $\bar{\bm{u}}\in\mathbb{R}^n$ is a reference state, commonly 
selected as the initial condition, base flow or snapshot mean.  
An important detail in our formulation is that $\bar{\bm{u}}$ can be a
function of time, i.e., $\bar{\bm{u}}:=\bar{\bm{u}}(t)$.  
As we show 
below, $\bar{\bm{u}}$ can also be used to enforce time-varying 
Dirichlet boundary conditions strongly within the POD-based ROM.  


\paragraph{Galerkin projection} 
Given a reduced basis ${\REV{\Phi}}$, the next step is to project the FOM onto this  basis.  
Here, we restrict our discussion to an approach called ``discrete Galerkin projection", 
where the governing equations in their semi-discretized form \eqref{eq:genericFOM}  (i.e., discretized only in space)
are projected onto the POD basis in the discrete $l^2$ inner product.  
Projecting \eqref{eq:genericFOM} onto a reduced basis ${\REV{\Phi}}$ and substituting the 
modal decomposition \eqref{eq:modal_decomp} yields:
\begin{equation} \label{eq:genericROM}
{\REV{\Phi}}^T M{\REV{\Phi}} \dot{\widetilde{\bm{u}}} = 
{\REV{\Phi}}^T\bm{f}(\bar{\bm{u}} + {\REV{\Phi}}\widetilde{\bm{u}}) 
- {\REV{\Phi}}^T M \dot{\bar{\bm{u}}}.
\end{equation}
%
In the present context,  the role of the FOM \eqref{eq:genericFOM} will be played by the coupled FOM-FOM problem \eqref{AdC:eq:Idx1System}. Details of the Galerkin projection for this problem are given in Sections \ref{sec:ROM-ROM} and \ref{sec:ROM-FOM}. 

\begin{remark}
While we focus our attention on the POD method 
for reduced basis construction and on the Galerkin method for the projection step, we emphasize 
that our approach is not limited to these methods and can be applied to any reduced order formulation of the coupled FOM-FOM problem that has a provably non-singular Schur complement. 
%
%
We note that, for nonlinear problems, a third step, 
known as hyper-reduction, is needed to treat efficiently
the projection of the nonlinear terms in the governing equations.  
A variety of approaches for hyper-reduction exist in the literature, e.g., the
Discrete Empirical Method (DEIM) \cite{DEIM}, gappy POD \cite{Everson:95}, or the Energy Conserving
Sampling and Weighting (ECSW) method \cite{FarhatECSW}.  The partitioned solver developed herein is easily extendable to nonlinear problems, but we omit  a detailed discussion of hyper-reduction as our numerical experiments focus on linear 
problems.   It is straightforward to see that both the reduced basis construction 
and Galerkin projection steps of our model reduction procedure can be precomputed offline, as shown explicitly later, in Sections \ref{sec:ROM-ROM} and \ref{sec:ROM-FOM}.   
\end{remark}

\subsection{Reduced basis sets for the transmission problem}\label{sec:rb-model}
We now specialize the first step of the generic POD-based MOR workflow in Section \ref{sec:pod-RB} 
to obtain the reduced basis sets that will be used in this work. The second step, i.e., the Galerkin projection onto the reduced basis, will be discussed in Sections \ref{sec:ROM-ROM}--\ref{sec:ROM-FOM}.
Let $X_{i}$ denote a set of $r_i$ snapshots on $\Omega_i$, $i=1,2$. 
The columns of  $X_{i}$ are the coefficient vectors $\bm{u}^m_{i}$ of finite 
element solutions $u^{h,m}_{i}\in S^h_{i}$, $m=1,\ldots,r_i$. 
Thus, $X_{i}$ is an $n_i\times r_i$ matrix. 
The columns of $X_i$ can be partitioned as 
$\bm{u}^m_{i} =(\bm{u}^m_{i,\gamma},\bm{u}^m_{i,0},\bm{g}^m_{i})$, where the 
coefficient subvectors are defined in  Section \ref{AdC:sec:model}. 
To handle the Dirichlet boundary conditions, we adopt an approach similar to the one in 
 \cite{AdC:GunzburgerBCs}. Specifically, we remove the subvectors $\bm{g}^m_i$  
corresponding to the Dirichlet nodes to obtain the  $n_{i,D}\times r_i$ adjusted snapshot matrix $X_{i,D}$. The $m^{th}$ column of this matrix is given by the vector $(\bm{u}^m_{i,\gamma},\bm{u}^m_{i,0})$ for $\bm{u}^m_{i,\gamma} \in \mathbb{R}^{n_{i,\gamma}}$, $\bm{u}^m_{i,0} \in \mathbb{R}^{n_{i,0}}$; see Figure~\ref{fig:snapshots}.  
%
\begin{figure}[t!]
  \begin{center}
    \includegraphics[width=0.9\textwidth]{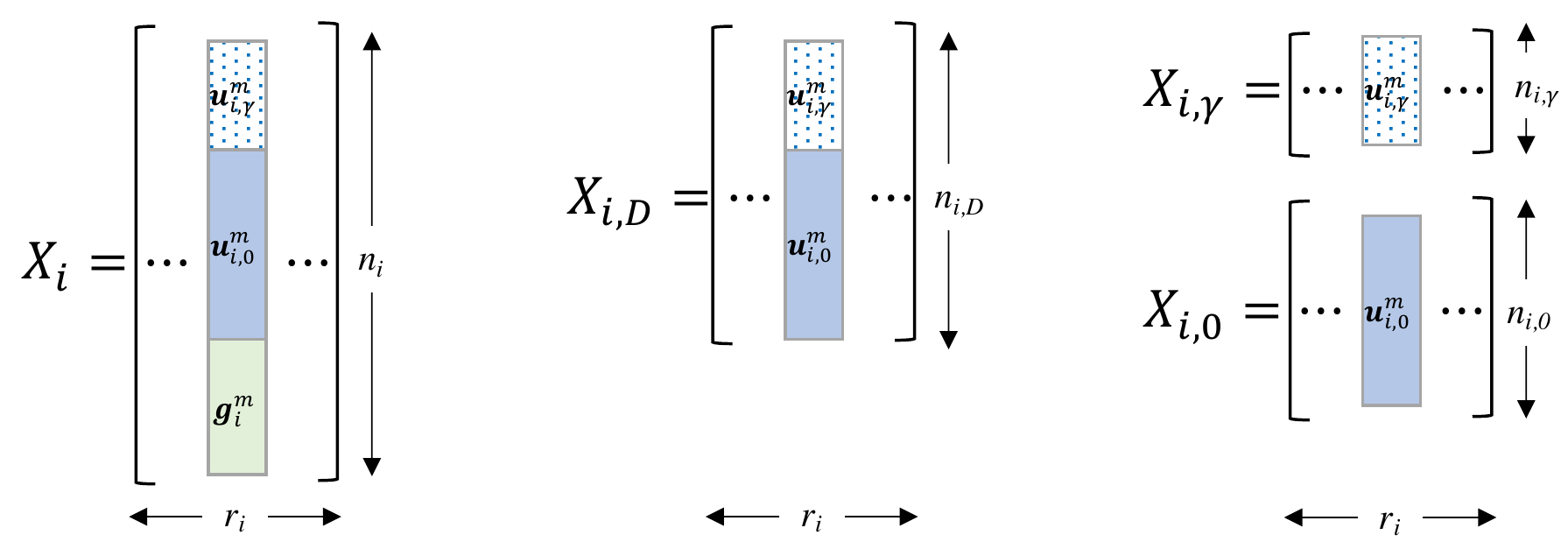}
  \end{center}
		\caption{Left to right: the original snapshot matrix $X_i$, the adjusted snapshot matrix $X_{i,D}$, the interior $X_{i,0}$ and interface $X_{i,\gamma}$ snapshot matrices.} \label{fig:snapshots}
\end{figure}
%
We further split the adjusted snapshot matrix into an $n_{i,0}\times r_i$ submatrix $X_{i,0}$ containing all interior nodal values of the snapshots and an $n_{i,\gamma}\times r_i$ submatrix $X_{i,\gamma}$  containing all interface nodal values of the snapshots. Thus, the columns of $X_{i,0}$ and $X_{i,\gamma}$ are given  by the coefficient vectors $\bm{u}^m_{i,0}$ and  $\bm{u}^m_{i,\gamma}$, respectively.  Figure \ref{fig:snapshots} 
illustrates the construction of these companion snapshot matrices.

Next, for $i=1,2$, we apply the POD basis construction to $X_{i,D}$, $X_{i,0}$ and $X_{i,\gamma}$. Specifically, 
we: (i) compute the SVDs of these matrices, (ii) choose the integers $0<d_{i,D}\ll n_{i,D}$, $0<d_{i,0}\ll n_{i,0}$ ,and $0<d_{i,\gamma}\ll n_{i,\gamma}$  that define the percent snapshot energy captured by the reduced basis for each respective set of snapshots, and (iii) form the $n_{i,D}\times d_{i,D}$, $n_{i,0}\times d_{i,0}$, and $n_{i,\gamma}\times d_{i,\gamma}$ reduced bases 
${\REV{\Phi}}_{i,D}$, ${\REV{\Phi}}_{i,0}$, and ${\REV{\Phi}}_{i,\gamma}$, respectively. 
%
Because the columns of ${\REV{\Phi}}_{i,D}$ 
contain both the interior and interface DoFs on $\Omega_i$, 
in the literature they are usually referred to as the \emph{full subdomain bases} \cite{AdC:Hoang}. 
We include these bases because they are ubiquitous in methods that use DD as a vehicle to improve the efficiency of the MOR workflow; see; e.g.,  \cite{AdC:Iapichino_16_CMA, AdC:Hoang, Wicke_2009}. 
Similarly, we refer to ${\REV{\Phi}}_{i,0}$ and ${\REV{\Phi}}_{i,\gamma}$ as the \emph{interior} and \emph{interface} reduced bases, respectively. Once ${\REV{\Phi}}_{i,D}$,  ${\REV{\Phi}}_{i,0}$, and ${\REV{\Phi}}_{i,\gamma}$ are obtained, one can approximate the coefficients of the FOM solution on $\Omega_i$, for $i=1,2$, as either 
\begin{equation} \label{eq:modal_decomp2}
\bm{u}_i(t) \approx \left( {\REV{\Phi}}_{i,D} \widetilde{\bm{u}}_{i,D}(t), \bm{g}_i(t)\right)
 \quad\mbox{or}\quad
\bm{u}_i(t) \approx  
\left({\REV{\Phi}}_{i,\gamma}
\widetilde{\bm{u}}_{i,\gamma}(t), 
{\REV{\Phi}}_{i,0}
\widetilde{\bm{u}}_{i,0}(t),\bm{g}_i(t)
\right),
 \end{equation}
where $\widetilde{\bm{u}}_{i,D}(t)\in\mathbb{R}^{d_{i,D}}$, $\widetilde{\bm{u}}_{i,\gamma}(t)\in\mathbb{R}^{d_{i,\gamma}}$, and $\widetilde{\bm{u}}_{i,0}(t)\in\mathbb{R}^{d_{i,0}}$
are unknown  time-dependent modal amplitudes. It is straightforward to see that both  ROM solutions in  
\eqref{eq:modal_decomp2} will satisfy the prescribed boundary conditions by construction.
In this context, the reference state in \eqref{eq:modal_decomp}
is given by $\bar{\bm{u}}_i(t) = (\bm{0},\bm{g}_{i}(t))$, with $\bm{0} \in \mathbb{R}^{n_{i,D}}$.

\paragraph{Reduced bases for the Lagrange multiplier} \label{sec:lm}
Because our FOM is given by the coupled problem \eqref{AdC:eq:Idx1System} in which the interface conditions are enforced by Lagrange multipliers, its Galerkin projection also requires a suitable reduced basis for the Lagrange multiplier. Such a basis can be obtained either independently 
from the RB matrices for the states or reusing them in a suitable way. In the first case,
 one collects $r_\gamma$ snapshots from some generic 
Lagrange multiplier space $G^h_\gamma$ into an 
$n_{\gamma}\times r_\gamma$ snapshot matrix $X_{\gamma}$ and then 
follows the same procedure as above  to obtain an 
$n_{\gamma}\times d_{\gamma}$ reduced basis matrix 
${\REV{\Phi}}_{\gamma}$. In this paper, we use solely 
the second approach and define ${\REV{\Phi}}_{\gamma}$ using 
either ${\REV{\Phi}}_{i,D}$ or ${\REV{\Phi}}_{i,\gamma}$.  
The Lagrange multiplier is 
then approximated using its reduced basis as
\begin{equation}\label{eq:LM-proj}
\bm{\lambda}(t) \approx {\REV{\Phi}}_{\gamma} \widetilde{\bm{\lambda}}(t),
\end{equation}
where $\widetilde{\bm{\lambda}}(t)\in\mathbb{R}^{d_{\gamma}}$ is an unknown  time-dependent modal amplitude.

\begin{remark}
Although the construction of the reduced bases is purely algebraic, 
their columns represent coefficient vectors of finite 
element functions in $S^h_i$, for $i=1,2$, and $G^h_\gamma$. Specifically, 
using the columns of ${\REV{\Phi}}_{i,D}$, ${\REV{\Phi}}_{i,\gamma}$,
 ${\REV{\Phi}}_{i,0}$, and ${\REV{\Phi}}_{\gamma}$ as coefficients in 
an expansion in terms of the nodal basis functions, one obtains 
functions belonging in $S^h_{i,D}$, $S^h_{i,\gamma}$, $S^h_{i,0}$, 
and $G^h_\gamma$, respectively. These finite element functions can be viewed as
 basis sets spanning \emph{reduced subspaces} (RS) 
$\widetilde{S}^h_{i,D}$, $\widetilde{S}^h_{i,\gamma}$, $\widetilde{S}^h_{i,0}$, and $\widetilde{G}^h_\gamma$ of their respective parent finite element spaces. 
This functional viewpoint of 
the reduced bases will be convenient when analyzing the properties of the 
partitioned IVR schemes for the coupled ROM-ROM and ROM-FOM formulations.  This analysis 
is deferred until Section \ref{sec:analysis}.   
\end{remark}

%% file: sec_IVR_SplitBasis.tex
\section{An IVR scheme for coupled ROM-ROM problems}\label{sec:ROM-ROM}

In this section, we formulate an IVR scheme for the partitioned solution of two ROMs coupled across an interface. 
Formally such a scheme can be obtained by projecting the coupled FOM-FOM problem \eqref{AdC:eq:Idx1System} onto reduced subspaces $\widetilde{V}^h\subset V^h$ and $\widetilde{W}^h\subset W^h$ and then using the Schur complement of the resulting coupled ROM-ROM problem to calculate an accurate approximation of the interface flux. 

However, successful execution of this plan requires one to take into consideration the fact that Galerkin projection of mixed problems does not automatically preserve their stability properties; see, e.g., \cite{Rozza_07_CMAME}. In the present context this means that the Schur complement of the coupled ROM-ROM problem is not guaranteed to be non-singular even if the Schur complement of its parent coupled FOM-FOM has this property. 

A sufficient condition for \eqref{AdC:eq:Idx1System} to have a non-singular Schur complement was established in \cite{AdC:CAMWA}, and requires every Lagrange multiplier to be a trace of a finite element function from one of the two sides of the interface. In Section \ref{sec:analysis}, we prove that a similar \emph{trace compatibility} condition ensures that Galerkin projection of \eqref{AdC:eq:Idx1System} also has this property; see Remark \ref{lem:trace}.
However, this condition imposes restrictions on the choices of the reduced bases for the subdomain states and the Lagrange multiplier. In particular, the trace compatibility condition makes the full subdomain bases ${\REV{\Phi}}_{i,D}$ (for $i=1,2$) less than an ideal choice for the extension of the IVR scheme to coupled ROM-ROM problems. 
To explain the issues and motivate our approach, let us examine more closely the Galerkin projection of \eqref{AdC:eq:Idx1System} onto the  full subdomain bases. Such a projection uses the first ansatz in \eqref{eq:modal_decomp2} to approximate the subdomain states, i.e., 
\begin{equation}\label{AdC:eq:ansatz}
\bm{u}_{i} := \left({\REV{\Phi}}_{i,D} \widetilde{\bm{u}}_{i,D}, \bm{g}_i\right); \quad i=1,2\,,
\end{equation}
where $\widetilde{\bm{u}}_{i,D}\in \mathbb{R}^{d_{i,D}}$ are time-dependent modal amplitudes (reduced order states). 
One then has to select a \REV{reduced basis (RB)}  ${\REV{\Phi}}_{\gamma}$ for the Lagrange multiplier that is \emph{trace-compatible} with \eqref{AdC:eq:ansatz}. To construct ${\REV{\Phi}}_{\gamma}$ note that every column $\REV{\bm{\phi}}^j_{i,D}\in\mathbb{R}^{n_{i,D}}$ of the full subdomain basis ${\REV{\Phi}}_{i,D}$ can be partitioned as  $\REV{\bm{\phi}}^j_{i,D} = (\REV{\bm{\phi}}^j_{i,D,\gamma},\REV{\bm{\phi}}^j_{i,D,0})$, where $\REV{\bm{\phi}}^j_{i,D,\gamma}\in\mathbb{R}^{n_{i,\gamma}}$ and $\REV{\bm{\phi}}^j_{i,D,0}\in\mathbb{R}^{n_{i,0}}$ are sub-vectors corresponding to interface and interior degrees of freedom, respectively. 
Let ${\REV{\Phi}}_{i,D,\gamma}$ denote the $n_{i,\gamma}\times d_{i,D}$ matrix whose $j$th column is given by $\REV{\bm{\phi}}^j_{i,D,\gamma}$. It is easy to see that ${\REV{\Phi}}_{\gamma} := {\REV{\Phi}}_{k,D,\gamma}$ for $k=1$ or $k=2$ is a trace-compatible RB for the Lagrange multiplier:  every function in the  reduced space $\widetilde{G}^h_\gamma$ spanned by ${\REV{\Phi}}_{\gamma}$ is a trace of a function in the reduced space $\widetilde{S}^h_{k,D}$ spanned by ${\REV{\Phi}}_{k,D}$.
However, since sub-vectors of a linearly independent set of vectors are not necessarily linearly independent on their own, the reduced order basis ${\REV{\Phi}}_{i,D,\gamma}$ is not guaranteed to have a full column rank. 

This is a serious drawback for the development of the IVR scheme as it is easily seen that rank-deficiency of ${\REV{\Phi}}_{i,D,\gamma}$ will lead to rank-deficiency of transposed projected constraint matrix $\widetilde{G}^T = (\widetilde{G}_{1,D},-\widetilde{G}_{2,D})^T$, thereby resulting in coupled ROM-ROM problems whose Schur complement is not invertible. 
Of course, one can prune the redundant basis functions from ${\REV{\Phi}}_{i,D,\gamma}$ by computing its SVD  and throwing away all left singular vectors corresponding to zero singular values. Unfortunately, this solution also suffers from some serious flaws. 
First, the ``pruned'' RB for the Lagrange multiplier may fail to satisfy the original ``snapshot energy" criterion \eqref{eq:delta} used to select the full subdomain basis ${\REV{\Phi}}_{i,D}$. Second, to preserve trace compatibility one would have to project the interface states $\bm{u}_{k,\gamma}$ using the ``pruned'' RB, while the interior states $\bm{u}_{k,0}$ will have to be projected using the RB ${\REV{\Phi}}_{k,D,0}$ whose columns are given by the sub-vectors $\REV{\bm{\phi}}^j_{i,D,0}$. It is clear that ${\REV{\Phi}}_{k,D,0}$ may suffer from the same issues as ${\REV{\Phi}}_{k,D,\gamma}$, that is, it may be rank-deficient. 
%

Instead of trying to extract a trace-compatible RB for the Lagrange multiplier from the full subdomain basis and potentially lose its optimality with respect to the snapshot energy criterion, a more robust strategy is to ensure trace compatibility from the onset by using separate RBs for the interior and interface variables. The following section describes the construction of a coupled ROM-ROM formulation based on this idea.

\subsection{A coupled ROM-ROM based on a composite reduced basis} \label{AdC:sec:split-basis}
Our strategy for securing a coupled ROM-ROM problem with a provably non-singular Schur complement has two key ingredients. The first one is the projection of the state $\bm{u}_{i,D}$ using pairs (thus the term ``composite'' RB) $\{{\REV{\Phi}}_{i,\gamma},{\REV{\Phi}}_{i,0}\}_{i=1,2}$ of \emph{independently} computed interface and interior RBs instead of a single full subdomain RB. 
The second ingredient is achieving trace compatibility by selecting one of the two interface RBs as \REVpk{a} RB for the Lagrange multiplier, i.e., we set ${\REV{\Phi}}_{\gamma} = {\REV{\Phi}}_{k,\gamma}$ for $k=1$ or $k=2$. This choice mimics the  one in \eqref{AdC:eq:Idx1System} and, as we shall prove in Section \ref{sec:analysis}, ensures that the Schur complement of the coupled ROM-ROM problem is non-singular. It also has some similarities with the techniques in \cite{AdC:Eftang_13_IJNME} and \cite{AdC:Hoang}.  

To  project the coupled FOM-FOM \eqref{AdC:eq:Idx1System} using the composite RB, we apply separate projections to the interior and to the interface DoFs. Thus, instead of  \eqref{AdC:eq:ansatz}, we use the second ansatz in \eqref{eq:modal_decomp2} and set
\begin{equation}\label{AdC:eq:ansatz-separate}
\bm{u}_i = 
\left( 
{\REV{\Phi}}_{i,\gamma} \widetilde{\bm{u}}_{i,\gamma},
{\REV{\Phi}}_{i,0} \widetilde{\bm{u}}_{i,0} ,\bm{g}_i
\right)
\quad\mbox{and}\quad
\bm{\lambda} = {\REV{\Phi}}_{k,\gamma} \widetilde{\bm{\lambda}}.
\end{equation}
Here, the time-dependent modal amplitudes $\widetilde{\bm{u}}_{i,\gamma}\in\mathbb{R}^{d_{i,\gamma}}$, $\widetilde{\bm{u}}_{i,0}\in\mathbb{R}^{d_{i,0}}$, and $\widetilde{\bm{\lambda}}\in\mathbb{R}^{d_{k,\gamma}}$ represent the reduced order interface and interior states and the reduced order Lagrange multiplier, respectively. 
Following Section \ref{sec:pod-RB},
 we insert \eqref{AdC:eq:ansatz-separate} into \eqref{AdC:eq:Idx1System} and multiply the blocks corresponding to the interior, interface,
 and Lagrange multiplier DoFs by ${\REV{\Phi}}_{i,0}^T$, ${\REV{\Phi}}_{i,\gamma}^T$, and ${\REV{\Phi}}_{k,\gamma}^T$ respectively. These steps yield the following \emph{composite reduced basis} coupled ROM-ROM formulation:
\begin{align}\label{AdC:eq:SeparatedRR}
\begin{split}
&\begin{bmatrix}
\widetilde{M}_{1,\gamma \gamma} & \widetilde{M}_{1, \gamma 0} & 0  & 0 &  \widetilde{G}_{1,\gamma}^T \\[1ex]
\widetilde{M}_{1,0 \gamma} & \widetilde{M}_{1,0 0} & 0 & 0 & 0\\[1ex]
0 & 0 & \widetilde{M}_{2,\gamma \gamma} & \widetilde{M}_{2,\gamma 0} & -\widetilde{G}_{2,\gamma}^T \\[1ex]
0 & 0 & \widetilde{M}_{2,0 \gamma} & \widetilde{M}_{2,0 0} & 0 \\[1ex]
\widetilde{G}_{1,\gamma} & 0 & -\widetilde{G}_{2,\gamma} & 0 & 0
\end{bmatrix}
\begin{bmatrix}
\dot{\widetilde{\bm{u}}}_{1,\gamma} \\[1.25ex] 
\dot{\widetilde{\bm{u}}}_{1,0} \\[1.25ex]
\dot{\widetilde{\bm{u}}}_{2,\gamma} \\[1.25ex]
\dot{\widetilde{\bm{u}}}_{2,0} \\[1.25ex]
 \widetilde{\bm{\lambda}}
\end{bmatrix}
=
\begin{bmatrix}
 \widetilde{\bm{s}}_{1,\gamma} \\[1.35ex]
  \widetilde{\bm{s}}_{1,0} \\[1.35ex]
  \widetilde{\bm{s}}_{2,\gamma} \\[1.35ex]
  \widetilde{\bm{s}}_{2,0} \\[1.35ex] 
 \widetilde{\bm{s}}_{\gamma}
\end{bmatrix}
\end{split}
\,.
\end{align} 
The block structure of \eqref{AdC:eq:ansatz-separate} is induced by the projection onto the composite RB space. 
Specifically, we have that,  for $i=1,2$,
 \begin{align*}
 \begin{bmatrix}
\widetilde{\bm{s}}_{i,\gamma} \\[1ex] 
\widetilde{\bm{s}}_{i,0}
\end{bmatrix} 
= 
\begin{bmatrix}
\widetilde{\bm{f}}_{i,\gamma} \\[1ex] 
\widetilde{\bm{f}}_{i,0} 
\end{bmatrix}
-
 \begin{bmatrix}
 \widetilde{F}_{i,\gamma \gamma} & \widetilde{F}_{i,\gamma 0} \\[0.5ex]
 \widetilde{F}_{i, 0 \gamma} & \widetilde{F}_{i,0 0}
 \end{bmatrix}
\begin{bmatrix}
 \widetilde{\bm{u}}_{i,\gamma} \\[1ex] \widetilde{\bm{u}}_{i,0}
 \end{bmatrix}  
 -
 \begin{bmatrix}
  {\REV{\Phi}}_{i,\gamma}^T Q_{i,\gamma \Gamma}(\dot{\bm{g}}_i,\bm{g}_i) \\[1ex]
  {\REV{\Phi}}_{i,0}^T Q_{i,0 \Gamma}(\dot{\bm{g}}_i,\bm{g}_i)  
 \end{bmatrix}\,;
 \quad
 \widetilde{\bm{s}}_{\gamma} =  - {\REV{\Phi}}^T_{k,\gamma} Q_{\gamma,\Gamma}(\dot{\bm{g}}_1,\dot{\bm{g}}_2), 
 \end{align*} 
and that, for $i \in \{1,2\}$ and $\{p,q\} \in \{\gamma, 0\}$, 
\begin{equation}\label{eq:split-matrices}
\begin{array}{c}
\displaystyle
\widetilde{M}_{i,pq} := {\REV{\Phi}}_{i,p}^T M_{i,pq} {\REV{\Phi}}_{i,q}, \quad
\widetilde{F}_{i,pq} := {\REV{\Phi}}_{i,p}^T F_{i,pq} {\REV{\Phi}}_{i,q}, \quad
\widetilde{G}_{i,\gamma} := {\REV{\Phi}}_{k,\gamma}^T G_{i,\gamma} {\REV{\Phi}}_{i,\gamma}, \\[2ex]
\widetilde{\bm{f}}_{i,p} := {\REV{\Phi}}_{i,p}^T \bm{f}_{i,p}, \quad
Q_{i,p \Gamma}( \dot{\bm{g}}_{i}, \bm{g}_{i})=  \ M_{i,p \Gamma} \dot{\bm{g}}_{i} + F_{i,p \Gamma} \bm{g}_{i}. 
\end{array}
\end{equation}
For $p\in \{\gamma,0\}$,  the matrices $M_{i,p \Gamma} $ and $F_{i,p \Gamma}$ are the blocks of the ``partial'' mass and flux matrices corresponding to the interface  and interior variables respectively, i.e., 
$$
M_{i,\Gamma} = [M_{i,\gamma \Gamma};M_{i,0 \Gamma}] 
\quad\mbox{and}\quad       
F_{i,\Gamma} = [F_{i,\gamma \Gamma};F_{i,0 \Gamma}] \,.
$$


\subsection{An IVR scheme for the coupled ROM-ROM based on a composite reduced basis}\label{AdC:sec:split-basis-algorithm}
In this section, we formulate an IVR scheme for the partitioned solution of the coupled ROM-ROM problem \eqref{AdC:eq:SeparatedRR}. Analysis in Section \ref{sec:analysis} provides a rigorous mathematical basis for this scheme by showing that the Schur complement of \eqref{AdC:eq:SeparatedRR} is symmetric and positive definite.

Since this IVR scheme is based on a coupled ROM-ROM problem, similar to other ROM methods, it comprises an offline stage where one computes the reduced basis and projects the FOM and an online stage where one uses the resulting ROM to simulate the system of interest. 
Although the offline stage of the IVR ROM-ROM scheme is very similar to that of any POD-based model order reduction scheme, we include it for completeness of the presentation. 

\paragraph{Offline: Computation of the composite basis ROMs}
\begin{enumerate}
\item 
\textbf{Snapshot collection.} 
Solve the transmission problem \eqref{AdC:strongForm} using a suitable full order model to obtain the 
subdomain snapshot matrices ${X}_i$, for $i=1,2$.  Form the adjusted snapshot matrices $X_{i,D}$ and  extract their interior $X_{i,0}$ and interface $X_{i,\gamma}$ parts.
\item \textbf{Reduced basis calculation.} For $i=1,2$, choose accuracy thresholds $\delta_{i,0}, \delta_{i,\gamma} >0$, determine the reduced bases dimensions $d_{i,0}$ and $d_{i,\gamma}$ as in \eqref{eq:delta}, and calculate the reduced bases ${\REV{\Phi}}_{i,0}$, ${\REV{\Phi}}_{i,\gamma}$ following Section \ref{sec:rb-model}. Choose $k=1$ or $k=2$ and set ${\REV{\Phi}}_{\gamma}={\REV{\Phi}}_{k,\gamma}$.
%
\item \textbf{Galerkin projection.}  For $i=1,2$ and $\{p,q\} \in \{0,\gamma \}$, precompute the ROM matrices:
$$
\begin{array}{c}
\widetilde{M}_{i,pq} := {\REV{\Phi}}^T_{i,p} M_{i,pq} {\REV{\Phi}}_{i,q} \in 
\mathbb{R}^{d_{i,p}\times d_{i,q}}; \ \ 
\widetilde{F}_{i,pq} := {\REV{\Phi}}_{i,p}^T F_{i,pq} {\REV{\Phi}}_{i,q} \in 
\mathbb{R}^{d_{i,p}\times d_{i,q}}; \ \ 
\widetilde{G}_i := {\REV{\Phi}}_{k,\gamma}^T G_i {\REV{\Phi}}_{i,\gamma} \in \mathbb{R}^{d_{k,\gamma} \times d_{i,\gamma}}\\[2ex]
\widetilde{Q}_{i,p \Gamma} 
= \{{\REV{\Phi}}_{i,p}^T M_{i,p \Gamma},  {\REV{\Phi}}_{i,p}^T F_{i,p \Gamma} \},
\quad\mbox{and}\quad
\widetilde{Q}_{\gamma,\Gamma} 
= \{{\REV{\Phi}}^T_{k,\gamma} G_{1,\Gamma}, -{\REV{\Phi}}^T_{k,\gamma} G_{2,\Gamma}\}
\end{array}\,.
$$
\end{enumerate}

\paragraph{Online: Partitioned solution of the  coupled ROM-ROM system \eqref{AdC:eq:SeparatedRR}}
\begin{enumerate}
\vspace{1mm}
\item Given a simulation time interval $[0,T]$ 
choose explicit time integration schemes $D^n_{i,t}(\widetilde{\bm{u}})$ on $\Omega_i$, $i=1,2$.
\item For $i=1,2$, $p \in \{0,\gamma \}$ and $n=0,1,\ldots$, compute the right-hand side vectors $\widetilde{\bm{f}}^n_{i,p} := {\REV{\Phi}}_{i,p}^T \bm{f}^n_{i,p}$,
 \begin{align*}
 \begin{bmatrix}
\widetilde{\bm{s}}^n_{i,\gamma} \\[1ex] 
\widetilde{\bm{s}}^n_{i,0}
\end{bmatrix} 
= 
\begin{bmatrix}
\widetilde{\bm{f}}^n_{i,\gamma} \\[1ex] 
\widetilde{\bm{f}}^n_{i,0} 
\end{bmatrix}
-
 \begin{bmatrix}
 \widetilde{F}_{i,\gamma \gamma} & \widetilde{F}_{i,\gamma 0} \\[0.5ex]
 \widetilde{F}_{i, 0 \gamma} & \widetilde{F}_{i,0 0}
 \end{bmatrix}
\begin{bmatrix}
 \widetilde{\bm{u}}^n_{i,\gamma} \\[1ex] \widetilde{\bm{u}}^n_{i,0}
 \end{bmatrix}  
 -
 \begin{bmatrix}
  \widetilde{Q}_{i,\gamma \Gamma}(\dot{\bm{g}}^n_i,\bm{g}^n_i) \\[1ex]
  \widetilde{Q}_{i,0 \Gamma}(\dot{\bm{g}}^n_i,\bm{g}^n_i)  
 \end{bmatrix}\,,
 \quad\mbox{and}\quad
\widetilde{\bm{s}}^n_{\gamma} 
=  - \widetilde{Q}_{\gamma,\Gamma}(\dot{\bm{g}}^n_1,\dot{\bm{g}}^n_2),
 \end{align*} 
where $\dot{\bm{g}}^n_i$ is approximation of the time derivative of the boundary data at the current time step.
\item 
For $i=1,2$, let $\widetilde{M}_i$, $\widetilde{G}_i$, $\widetilde{\bm{s}}^n_i$ and $\widetilde{\bm{u}}^n_i$ denote the $2\times 2$, and $2\times 1$ block matrices and vectors in \eqref{AdC:eq:SeparatedRR}. Solve the Schur complement system for $\widetilde{\bm{\lambda}}^n$:
\begin{equation}\label{eq:split-schur}
\Big(\widetilde{G}_1 \widetilde{M}_1^{-1} \widetilde{G}_1^T + \widetilde{G}_2 \widetilde{M}_2^{-1} \widetilde{G}_2^T\Big) \widetilde{\bm{\lambda}}^n  =
\widetilde{G}_1 \widetilde{M}_1^{-1} \widetilde{\bm{s}}_{1}^n - 
\widetilde{G}_2 \widetilde{M}_2^{-1}  \widetilde{\bm{s}}_{2}^n -
\widetilde{\bm{s}}_{\gamma}^n \,.
\end{equation}
 \item For $i=1,2$, solve the subdomain ROM problems 
$$
 \begin{bmatrix} 
 \widetilde{M}_{i,\gamma \gamma} & \widetilde{M}_{i,\gamma 0} \\[1ex]
 \widetilde{M}_{i,0\gamma} & \widetilde{M}_{i,00} 
 \end{bmatrix} 
 D^n_{i,t}\left(\begin{bmatrix} \widetilde{\bm{u}}_{i,\gamma}^{n+1} \\[1ex]
 \widetilde{\bm{u}}_{i,0}^{n+1} 
 \end{bmatrix} \right) 
 = 
 \begin{bmatrix} \widetilde{\bm{s}}_{i,\gamma}^n + (-1)^i \widetilde{G}_i^T \widetilde{\bm{\lambda}}^n \\[1ex]
 \widetilde{\bm{s}}_{i,0}^n
 \end{bmatrix}
 $$
 for the ROM solution $\widetilde{\bm{u}}_i^{n+1}$ at the new time step.
 \item For $i=1,2$, project the ROM solutions to the state spaces of the FOMs on $\Omega_i$:
 $$
 \bm{u}^{n+1}_i = \left(
{\REV{\Phi}}_{i,\gamma} \widetilde{\bm{u}}^{n+1}_{i,\gamma},
{\REV{\Phi}}_{i,0} \widetilde{\bm{u}}^{n+1}_{i,0} ,\bm{g}_i^{n+1}\right)
 \,.
$$
\end{enumerate}

\section{An IVR scheme for coupled ROM-FOM problems}\label{sec:ROM-FOM}
\REV{This section extends the IVR scheme to the case of a ROM coupled} with a FOM. Such a formulation is relevant to multiple simulation scenarios for the model transmission problem \eqref{AdC:strongForm}. One such scenario is when one of the subdomains is much larger than the other and its FOM  dominates the computational cost. To balance the computational costs across the subdomains, one can replace the FOM on the large domain with a computationally efficient ROM, while retaining the FOM on the small subdomain.   
A second possible scenario is when the governing equations are parameterized 
on just one of the subdomains and simulation of the other subdomain requires a scheme that can handle all admissible inputs. In this case, a ROM would be only appropriate for the domain with the parameterized equations, while on the other subdomain one would still have to use  a FOM.
A related scenario is the case where a coupled ROM-FOM model has the potential of having better predictive accuracy than a model based solely on ROMs \cite{Lucia:2003, AdC:Iapichino_16_CMA, 
Kerfriden:2012, Ammar_2011, LeGresley:2005}. One example of this scenario would be a version of our model transmission problem \eqref{AdC:strongForm} in which the diffusion coefficient on one of the subdomains is allowed to be identically zero.  In this case, the governing equation on that subdomain reduces to a pure advection problem for which  POD-based MOR is not effective; see Section \ref{sec:POD}.

\subsection{A coupled ROM-FOM based on a composite reduced basis} \label{sec:split-basis-ROM-FOM}
As in the coupled ROM-ROM case, to develop an IVR scheme for the partitioned solution of coupled ROM-FOM problems, we first formulate a coupled ROM-FOM problem that is guaranteed to have a symmetric and positive definite Schur complement. To obtain this problem, assume that the FOM should be retained on $\Omega_2$ while simulation on $\Omega_1$ can be performed by a ROM. 
To define the corresponding coupled  ROM-FOM, we start from the coupled FOM-FOM \eqref{AdC:eq:Idx1System}, retain the FOM on $\Omega_2$, and project the state on $\Omega_1$  
${\REV{\Phi}}_{1,C}:=\{{\REV{\Phi}}_{1,\gamma},{\REV{\Phi}}_{1,0}\}$ using the composite RB ansatz
\begin{equation}\label{eq:ansatz-separate-ROM-FOM}
\bm{u}_1 = \left(
{\REV{\Phi}}_{1,\gamma} \widetilde{\bm{u}}_{1,\gamma},
{\REV{\Phi}}_{1,0} \widetilde{\bm{u}}_{1,0} ,\bm{g}_1 \right)\,.
\end{equation}
To complete the coupled ROM-FOM problem, one has to choose a trace-compatible representation for the Lagrange multiplier. In the present setting there are two possible options that satisfy this condition: the interface RB ${\REV{\Phi}}_{1,\gamma}$ from the ROM side of the interface (option ``rLM") or the interface finite element space $G^h_2$ from the FOM side of the interface (option ``fLM"). In the first case $\bm{\lambda} = {\REV{\Phi}}_{1,\gamma} \widetilde{\bm{\lambda}}$ and in the second case $\bm{\lambda}$ is the coefficient vector of a function $\lambda^h\in G^h_2$. Analysis in Section \ref{sec:analysis} will confirm that either one of these two options leads to coupled problems with non-singular Schur complements.

With these choices, we have the following \emph{composite RB} coupled ROM-FOM formulation:
\begin{align}\label{AdC:eq:SeparatedRR-ROM-FOM}
\begin{split}
&\begin{bmatrix}
\widetilde{M}_{1,\gamma \gamma} & \widetilde{M}_{1, \gamma 0} & 0  & 0 &  \widehat{G}_{1,\gamma}^T \\[1ex]
\widetilde{M}_{1,0 \gamma} & \widetilde{M}_{1,0 0} & 0 & 0 & 0\\[1ex]
0 & 0 & {M}_{2,\gamma \gamma} & {M}_{2,\gamma 0} & -\widehat{G}_{2,\gamma}^T \\[1ex]
0 & 0 & {M}_{2,0 \gamma} & {M}_{2,0 0} & 0 \\[1ex]
\widehat{G}_{1,\gamma} & 0 & -\widehat{G}_{2,\gamma} & 0 & 0
\end{bmatrix}
\begin{bmatrix}
\dot{\widetilde{\bm{u}}}_{1,\gamma} \\[1.25ex] 
\dot{\widetilde{\bm{u}}}_{1,0} \\[1.25ex]
\dot{\bm{u}}_{2,\gamma} \\[1.25ex]
\dot{\bm{u}}_{2,0} \\[1.25ex]
 \widehat{\bm{\lambda}}
\end{bmatrix}
=
\begin{bmatrix}
 \widetilde{\bm{s}}_{1,\gamma} \\[1.35ex]
 \widetilde{\bm{s}}_{1,0} \\[1.35ex]
  \bm{s}_{2,\gamma} \\[1.35ex]
  \bm{s}_{2,0} \\[1.35ex] 
 \widehat{\bm{s}}_{\gamma}
\end{bmatrix}
\end{split}
\,,
\end{align} 
where the blocks with the ``tilde" accent are defined as in the coupled ROM-ROM \eqref{AdC:eq:ansatz-separate}, the blocks without accents are defined as in the coupled FOM-FOM \eqref{AdC:eq:Idx1System}, and 
\begin{equation}\label{eq:hat}
\begin{array}{c}
\widehat{\bm{\lambda}}
=
\left\{\!\!
\begin{array}{rl}
\widetilde{\bm{\lambda}}, \!&\! \mbox{for option rLM} \\[1ex]
\bm{\lambda}, \!&\! \mbox{for option fLM}
\end{array}\,;
\right.
\widehat{\bm{s}}_{\gamma} 
=
\left\{\!\!
\begin{array}{rl}
 - {\REV{\Phi}}^T_{1,\gamma} Q_{\gamma,\Gamma}(\dot{\bm{g}}_1,\dot{\bm{g}}_2), 
\!&\! \mbox{for option rLM} \\[1ex]
 -Q_{\gamma,\Gamma}(\dot{\bm{g}}_1,\dot{\bm{g}}_2), 
\!&\! \mbox{for option fLM}
 \end{array}
 \right. ;
\\ \\
\widehat{G}_{1,\gamma}
=
\left\{\!\!
\begin{array}{rl}
{\REV{\Phi}}^T_{1,\gamma}G_{1,\gamma}{\REV{\Phi}}_{1,\gamma}, \!&\! \mbox{for option rLM} \\[1ex]
G_{1,\gamma}{\REV{\Phi}}_{1,\gamma}, \!&\! \mbox{for option fLM}
\end{array}\, 
\right. ;\;
\widehat{G}_{2,\gamma}
=
\left\{\!\!
\begin{array}{rl}
{\REV{\Phi}}^T_{1,\gamma}G_{2,\gamma}, \!&\! \mbox{for option rLM} \\[1ex]
G_{2,\gamma}, \!&\! \mbox{for option fLM}
\end{array}\, 
\right. .
\end{array}
\end{equation}
In the next section, 
we present the IVR scheme for the partitioned solution of  \eqref{AdC:eq:SeparatedRR-ROM-FOM}.

\subsection{An IVR scheme for the coupled ROM-FOM based on a composite reduced basis}\label{sec:split-basis-algorithm-ROM-FOM}
Since the coupled ROM-FOM problem has a ROM component, the IVR scheme for \eqref{AdC:eq:SeparatedRR-ROM-FOM} necessarily involves an offline stage where one computes the reduced basis and projects the FOM on $\Omega_1$ and an online stage where one uses the coupled problem to simulate the model transmission problem. Although these stages are very similar to those outlined in Section \ref{AdC:sec:split-basis-algorithm}, we include an abridged version for the convenience of the reader. To reduce notational clutter in some cases we switch to the more compact notation $M_{2,D}$, $\bm{s}_{2,D}$, and $\bm{f}_{2,D}$ for the matrix and vector blocks of the FOM problem.

\paragraph{Offline: Computation of the composite basis ROM for $\Omega_1$}

\begin{enumerate}
\item 
\textbf{Snapshot collection.} Solve the transmission problem \eqref{AdC:strongForm} using a suitable full order model to obtain the snapshot matrix $X_1$, form $X_{1,D}$ and extract $X_{1,\gamma}$ and $X_{1,0}$.

\item 
\textbf{Reduced basis calculation.}
Choose accuracy thresholds $\delta_{1,0}, \delta_{1,\gamma} >0$, determine the reduced bases dimensions $d_{1,0}$ and $d_{1\gamma}$, and calculate the reduced bases ${\REV{\Phi}}_{1,0}$, ${\REV{\Phi}}_{1,\gamma}$. Select an option (rLM or fLM) for the Lagrange multiplier.
\item 
\textbf{Galerkin projection.}
For $\{p,q\} \in \{0,\gamma \}$, precompute the ROM matrices 
$\widetilde{M}_{1,pq}$, $\widetilde{F}_{1,pq}$, $\widetilde{Q}_{1,p \Gamma}$.
Precompute $\widehat{G}_{i,\gamma}$ as defined in \eqref{eq:hat} and 
$$
\widehat{Q}_{\gamma,\Gamma} 
:=
\left\{
\begin{array}{rl}
 \{{\REV{\Phi}}^T_{1,\gamma} G_{1,\Gamma}, -{\REV{\Phi}}^T_{1,\gamma} G_{2,\Gamma}\}, 
 & \mbox{for option rLM} \\[1ex]
 \{G_{1,\Gamma}, -G_{2,\Gamma}\}, 
 & \mbox{for option fLM}
 \end{array}
 \right..
$$ 
\end{enumerate}

\paragraph{Online: Partitioned solution of the composite basis coupled ROM-FOM system}
\begin{enumerate}
\vspace{1mm}
\item Given a simulation time interval $[0,T]$ 
choose explicit time integration schemes $D^n_{i,t}(\widetilde{\bm{u}})$ on $\Omega_i$, $i=1,2$.

\item For $p \in \{0,\gamma \}$ and and $n=0,1,\ldots$, 
compute the ROM vectors $\widetilde{\bm{f}}^n_{1,p} := {\REV{\Phi}}_{1,p}^T \bm{f}^n_{1,p}$, $\widetilde{\bm{s}}^n_{1,\gamma}$, $\widetilde{\bm{s}}^n_{1,0}$ as in Section \ref{AdC:sec:split-basis-algorithm}. Compute the FOM vectors $\bm{f}^n_{2,D}$, ${\bm{s}}^n_{2,D}$, and the right-hand side for the constraint equation:
$$
\widehat{\bm{s}}^n_{\gamma}  =  - \widehat{Q}_{\gamma,\Gamma}(\dot{\bm{g}}^n_1,\dot{\bm{g}}^n_2).
$$
\item Solve the Schur complement system for $\widehat{\bm{\lambda}}^n$:
\begin{equation}\label{eq:split-schur-ROM-FOM}
\Big(\widehat{G}_1 \widetilde{M}_1^{-1} \widehat{G}_1^T + \widehat{G}_2 {M}_{2,D}^{-1} \widehat{G}_2^T\Big) \widehat{\bm{\lambda}}^n  =
\widehat{G}_1 \widetilde{M}_1^{-1} \widetilde{\bm{s}}_{1}^n - 
\widehat{G}_2 {M}_{2,D}^{-1}  {\bm{s}}_{2}^n -
\widehat{\bm{s}}_{\gamma}^n \,.
\end{equation}
\item Solve the subdomain ROM problem
 $$
 \begin{bmatrix} 
 \widetilde{M}_{1,\gamma \gamma} & \widetilde{M}_{1,\gamma 0} \\ 
 \widetilde{M}_{1,0\gamma} & \widetilde{M}_{1,00} 
 \end{bmatrix} 
 D^n_{1,t}\left(\begin{bmatrix} \widetilde{\bm{u}}_{1,\gamma}^{n+1} \\ 
 \widetilde{\bm{u}}_{1,0}^{n+1} 
 \end{bmatrix} \right) 
 = 
 \begin{bmatrix} \widetilde{\bm{s}}_{1,\gamma}^n - \widehat{G}_1^T  \widehat{\bm{\lambda}}^n \\ 
 \widetilde{\bm{s}}_{i,0}^n
 \end{bmatrix}, 
 $$
 to obtain the ROM solution $\widetilde{\bm{u}}_1^{n+1}$ at the new time step. 
 Solve the subdomain FOM problem
$$
 \begin{bmatrix} 
{M}_{2,\gamma \gamma} &{M}_{2,\gamma 0} \\ 
{M}_{2,0\gamma} & {M}_{2,00} 
 \end{bmatrix} 
 D^n_{2,t}\left(\begin{bmatrix} {\bm{u}}_{2,\gamma}^{n+1} \\ 
{\bm{u}}_{2,0}^{n+1}
 \end{bmatrix} \right) 
= 
\begin{bmatrix} {\bm{s}}_{2,\gamma}^n + \widehat{G}_2^T \widehat{\bm{\lambda}}^n \\[1ex]
{\bm{s}}_{2,0}^n
 \end{bmatrix},
 $$
 to obtain the FOM solution ${\bm{u}}_2^{n+1}$ at the new time step. 
 \item Project the ROM solution to the state space of the FOM on $\Omega_1$:
 $$
 \bm{u}^{n+1}_1 = \left(
{\REV{\Phi}}_{1,\gamma} \widetilde{\bm{u}}^{n+1}_{1,\gamma},
{\REV{\Phi}}_{1,0} \widetilde{\bm{u}}^{n+1}_{1,0} ,\bm{g}_1^{n+1} \right) \,.
$$
\end{enumerate}

In the next section, we show that both the coupled ROM-ROM and ROM-FOM have provably non-singular Schur complements, thereby providing appropriate settings for an application of the IVR scheme.

%% file: sec_Analysis.tex
\section{Analysis}\label{sec:analysis}
We will first consider the coupled ROM-ROM formulation \eqref{AdC:eq:SeparatedRR} and use variational techniques to prove that it has a symmetric and positive definite Schur complement. We will then specialize this analysis to the case of the coupled ROM-FOM problem.

\subsection{Composite reduced basis coupled ROM-ROM}\label{sec:analysis-ROM-ROM}
Successful completion of the online stage of the IVR scheme for \eqref{AdC:eq:SeparatedRR} hinges on the unique solvability of the Schur complement system \eqref{eq:split-schur} in Step 3. We will show that this system  is uniquely solvable by proving that the Schur complement matrix 
\begin{equation}\label{eq:Schur-split}
\widetilde{S} :=
 \widetilde{G}_1 \widetilde{M}_1^{-1} \widetilde{G}_1^T + \widetilde{G}_2 \widetilde{M}_2^{-1} \widetilde{G}_2^T
 \end{equation}
 is symmetric and positive definite (SPD). It is well-known that this property requires the projected mass matrices $\widetilde{M}_i$, for $i=1,2$, to be symmetric and positive definite and the matrix $\widetilde{G}^T = (\widetilde{G}_1,-\widetilde{G}_2)^T$ to have full column rank. 
 
 While it is possible to develop purely algebraic proofs of the required properties, here we adopt a variational approach that exploits connections between the matrix defining the left hand side of the coupled ROM-ROM problem and mixed variational forms. In so doing, we illuminate how properties of the variational formulation underlying \eqref{AdC:eq:SeparatedRR} translate into properties of its algebraic equivalent. This approach can also expose potential dependencies of these properties on the dimension of the reduced basis and/or the mesh parameter.  
The latter is not easily achievable through strictly algebraic means.
 
Let $V^h_{D} = S^h_{1,D}\times S^h_{2,D}$ and $W^h=G^h_k$. We introduce the \emph{auxiliary} mixed 
variational form $B:(V^h_D\times W^h)\times (V^h_D\times W^h) \mapsto \mathbb{R}$ as
 \begin{equation}\label{eq:mixed-aux}
B({u}^h_{1,D},{u}^h_{2,D},\mu^h; {v}^h_{1,D},{v}^h_{2,D};\lambda^h):=
a({u}^h_{1,D},{u}^h_{2,D};{v}^h_{1,D},{v}^h_{2,D}) +
b({v}^h_{1,D},{v}^h_{2,D};\lambda^h) + b({u}^h_{1,D},{u}^h_{2,D};\mu^h) 
\end{equation}
where $a: V^h_D\times V^h_D\mapsto\mathbb{R}$ and $b:V^h_D\times W^h\mapsto \mathbb{R}$ are defined as 
$$
a({u}^h_{1,D},{u}^h_{2,D};{v}^h_{1,D},{v}^h_{2,D}) = 
\left(u^h_{1,D},v^h_{1,D}\right)_{0,\Omega_1} + \left(u^h_{2,D},v^h_{2,D}\right)_{0,\Omega_2} 
\ \ \mbox{and}\ \ 
b({v}^h_{1,D},{v}^h_{2,D};\lambda^h) = \left({v}^h_{1,D} - {v}^h_{2,D}, \lambda^h \right)_{0,\gamma},
$$
respectively. We call the bilinear form \eqref{eq:mixed-aux} ``auxiliary" because it is not the form that corresponds to the weak coupled problem \eqref{eq:weaF_form_LM};
 rather, it is the form that generates the matrix operators on the left-hand side of the coupled FOM-FOM problem \eqref{AdC:eq:Idx1System}. 
 
To prove that the Schur complement \eqref{eq:Schur-split} is SPD, we will apply Brezzi's mixed variational theory \cite{Brezzi_74_RAIRO} to \eqref{eq:mixed-aux} to show that the projected mass matrices are SPD and that the transpose constraint  matrix has full column rank. Application of this mixed theory requires a proper functional setting for \eqref{eq:mixed-aux}, specifically the endowment of $V^h_D$ and $W^h$ with suitable norms. We define these norms as 
\begin{equation}\label{eq:norms}
\|\{{v}^h_{1,D},{v}^h_{2,D}\}\|^2_{V} :=
\|{v}^h_{1,D}\|^2_{0,\Omega_1} + \|{v}^h_{2,D}\|^2_{0,\Omega_2}
\quad\mbox{and}\quad
\|\lambda^h\|_{W} := \|\lambda_h\|_{0,\gamma},
\end{equation}
respectively.

 
\begin{remark}\label{rem:DD}
Although the IVR scheme bears resemblance with DD methods based on Lagrange multipliers such as FETI \cite{Farhat_91_IJNME,Farhat_etal_1995_IJNME} and mortar methods \cite{Bernardi_94_INPROC}, the variational setting for its analysis provided by \eqref{eq:mixed-aux} and \eqref{eq:norms} is different from that required for the analysis of these  DD schemes. This difference stems from the fact that analysis of IVR relies on the auxiliary form \eqref{eq:mixed-aux}, which does not include any contributions from the flux terms, whereas analysis of DD methods involves the ``true'' mixed form corresponding to \eqref{eq:weaF_form_LM}, which includes such terms. A proper functional setting for the latter requires a broken $H^1$ norm on $V^h_D$, instead of the broken $L^2$ norm \eqref{eq:norms} used here, and a discrete approximation of the trace norm on $H^{-1/2}(\gamma)$. The use of the auxiliary form \eqref{eq:mixed-aux} relaxes the requirements on the functional setting for the application of the mixed theory and allows us to use a weaker norm on $V^h_D$ and a very ``crude'' approximation of the trace norm by an $L^2$ norm on $\gamma$. We refer to \cite{Toselli_Widlund_2005_Book} and  \cite{Braess_01_INPROC} for further details about the analysis of DD methods.
 \end{remark}

 To apply the mixed theory to \eqref{AdC:eq:SeparatedRR},
 we need to further adjust the variational setting so that the auxiliary form \eqref{eq:mixed-aux} produces the left-hand side of the coupled ROM-ROM problem. For $i=1,2$, let $\widetilde{S}^h_{i,\gamma}$, $\widetilde{S}^h_{i,0}$, and $\widetilde{G}^h_{k}$ be the reduced subspaces of ${S}^h_{i,\gamma}$, ${S}^h_{i,0}$, and ${G}^h_{k}$, induced by the columns of the composite RBs $\REV{\Phi}_{i,\gamma}$, $\REV{\Phi}_{i,0}$, and $\REV{\Phi}_{k,\gamma}$, respectively; see Section \ref{sec:POD}. We define the reduced subspaces $\widetilde{V}^h_C\subset V^h_D$ and $\widetilde{W}^h\subset W^h$ as
 \begin{equation}\label{eq:RS}
 \widetilde{V}^h_{C} = \widetilde{S}^h_{1,C}\times \widetilde{S}^h_{2,C}
 \quad\mbox{and}\quad
 \widetilde{W}^h=\widetilde{G}^h_{k}\,,
 \end{equation}
 respectively, where\footnote{Note that we have deviated from our usual naming convention and have labeled the subspace of $S^h_{i,D}$ engendered by the composite basis as $\widetilde{S}^h_{i,C}$. This is done in order to avoid confusion with the reduced subspace $\widetilde{S}^h_{i,D}$, which corresponds to the full subdomain RB matrix $\REV{\Phi}_{i,D}$.} $\widetilde{S}^h_{i,C}:= \widetilde{S}^h_{i,\gamma}\oplus\widetilde{S}^h_{i,0}$.
 
 It is straightforward to check that restriction of the auxiliary form \eqref{eq:mixed-aux} to the reduced subspaces \eqref{eq:RS} generates the matrix on the left-hand side of the coupled ROM-ROM problem \eqref{AdC:eq:SeparatedRR}.  For example, consider a finite element function
 ${u}^h_{i,C}\in \widetilde{S}^h_{i,C}$ with coefficient vector 
 $\bm{u}_{i,C} = (\bm{u}_{i,\gamma},\bm{u}_{i,0})$, where
 \begin{equation}\label{eq:example-a}
 \bm{u}_{i,\gamma} = \REV{\Phi}_{i,\gamma}\widetilde{\bm{u}}_{i,\gamma}
 \quad\mbox{and}\quad
  \bm{u}_{i,0} = \REV{\Phi}_{i,0}\widetilde{\bm{u}}_{i,0},
 \end{equation}
 for some modal amplitudes $\widetilde{\bm{u}}_{i,\gamma}$ and $\widetilde{\bm{u}}_{i,0}$. 
 Let now $\{{u}^h_{1,C},{u}^h_{2,C}\},\{{v}^h_{1,C},{v}^h_{2,C}\}\in \widetilde{V}^h_{C}$. Using  \eqref{eq:example-a} it easily follows that 
 \begin{equation}\label{eq;mass-SPD}
 a({u}^h_{1,C},{u}^h_{2,C};{v}^h_{1,C},{v}^h_{2,C}) 
 =
 \sum_{i=1}^2
\begin{bmatrix} 
\widetilde{\bm{u}}^T_{i,\gamma} & \widetilde{\bm{u}}^T_{i,0}
 \end{bmatrix} 
 \begin{bmatrix} 
\widetilde{M}_{i,\gamma\gamma} & \widetilde{M}_{i,\gamma 0} \\ 
\widetilde{M}_{i,0\gamma} & \widetilde{M}_{i,00} 
 \end{bmatrix} 
 \begin{bmatrix}
 \widetilde{\bm{u}}_{i,\gamma} \\ \widetilde{\bm{u}}_{i,0}\,
 \end{bmatrix}.
 \end{equation}

Application of the Brezzi theory requires verification of two separate  conditions on $a(\cdot,\cdot)$ and $b(\cdot,\cdot)$. Specialized to the functional setting for \eqref{AdC:eq:SeparatedRR} constructed earlier, these conditions are as follows.

\paragraph{Coercivity on the kernel} The form  $a(\cdot,\cdot)$ is coercive on the nullspace $\widetilde{Z}^h_{C}\subset \widetilde{V}^h_{C}$ of $b(\cdot,\cdot)$ defined as
$$
\widetilde{Z}^h_{C}
=
\left\{\{{v}^h_{1,C},{v}^h_{2,C}\}\in \widetilde{V}^h_{C} \,\big |
b({v}^h_{1,C},{v}^h_{2,C};\mu^h) = 0\quad\forall\mu^h\in\widetilde{W}^h
\right\} \,.
$$
 
 \paragraph{Inf-sup condition} For any $\mu^h\in\widetilde{W}^h$ the form $b(\cdot,\cdot)$ satisfies the inequality
\begin{equation}\label{eq:infsup}
\sup_{\{{v}^h_{1,C},{v}^h_{2,C}\}\in \widetilde{V}^h_{C}\times \widetilde{V}^h_{C}} 
\frac{b({v}^h_{1,C},{v}^h_{2,C};{\mu}^h)}{\|\{{v}^h_{1,C},{v}^h_{2,C}\}\|_V}
\ge \beta \| \mu^h\|_W,
\end{equation}
with a mesh-independent constant $\beta$. 
\medskip

The proof of the first Brezzi condition is trivial as it is easy to see that $a(\cdot,\cdot)$ is coercive on all of $V^h_{D}\times V^h_{D}$. Since strong coercivity is inherited on subspaces, it follows that $a(\cdot,\cdot)$ is coercive on $\widetilde{V}^h_{C}\times \widetilde{V}^h_{C}$ and its subspace $\widetilde{Z}^h_{C}\times \widetilde{Z}^h_{C}$. 
Using \eqref{eq;mass-SPD}, it easily follows that the algebraic translation of this property amounts to the statement that the projected mass matrices are SPD, which verifies the first condition necessary to establish that the Schur complement \eqref{eq:Schur-split} is SPD.

It is well-known that the second condition, i.e., the requirement that $\widetilde{G}^T$ has full column rank, is  a consequence of $b(\cdot,\cdot)$ satisfying the \textit{inf-sup} condition; see, e.g., \cite[p.38, Proposition 3.1]{Brezzi_90_CMAME} for a discussion of the relationships between algebraic and variational properties of discrete mixed problems. To prove the \textit{inf-sup} condition, we will need the following auxiliary result.

\begin{lemma}\label{lem:Q}
Let $h_{k}$ denote the characteristic element size of the interface mesh defining the Lagrange multiplier space $G^h_k$.  
There exists an operator $\mathcal{Q}: \widetilde{W}^h\mapsto \widetilde{V}^h_C$ such that for every $\mu^h\in\widetilde{W}^h$ there holds
\begin{equation}\label{eq:Q}
(a)\ \ \|\mu^h\|^2_{W} \le C_1 b(\mathcal{Q}(\mu^h);\mu^h)
\quad
\mbox{and}
\quad
(b) \ \ \|\mathcal{Q}(\mu^h)\|_{V} \le C_2 h^\alpha_{k} \|\mu^h\|_W \,,
\end{equation}
where $\alpha \ge 0$ and $C_1$, $C_2$ are positive constants independent of this element size. 
\end{lemma}
\begin{proof}
Let $k=1$ or $k=2$ be the index used to define the reduced basis for the Lagrange multiplier space.
According to \eqref{eq:example-a}, the coefficient vector of a function 
 ${u}^h_{k,C}\in \widetilde{S}^h_{k,C}$  is given by 
 $\bm{u}_{k,C} = (\REV{\Phi}_{k,\gamma}\widetilde{\bm{u}}_{k,\gamma}, \REV{\Phi}_{k,0}\widetilde{\bm{u}}_{k,0})$, where $\widetilde{\bm{u}}_{k,\gamma}\in\mathbb{R}^{d_{k,\gamma}}$ and  $\widetilde{\bm{u}}_{k,0}\in\mathbb{R}^{d_{k,0}}$ are the associated interface and interior modal amplitudes. 
Let $\mu^h\in\widetilde{W}^h$ be an arbitrary function in the reduced Lagrange multiplier space. 
The coefficient vector of this function is given by the ansatz in \eqref{AdC:eq:ansatz-separate}, i.e., $\bm{\mu} = \REV{\Phi}_{k,\gamma} \widetilde{\bm{\mu}}$, with a modal amplitude $\widetilde{\bm{\mu}} \in\mathbb{R}^{d_{k,\gamma}}$. It follows that the coefficient vector 
\begin{equation}\label{eq:lift}
\bm{u}^{\mu}_{k,C} = (\bm{\mu},\bm{0}) = (\REV{\Phi}_{k,\gamma} \widetilde{\bm{\mu}},\bm{0});
\quad
\bm{0}\in\mathbb{R}^{n_{k,0}}
\end{equation}
defines a lifting ${u}^{h,\mu}_{k,C}\in \widetilde{S}^h_{k,C}$ of $\mu^h$ such that 
$$
{u}^{h,\mu}_{k,C}\big |_{\gamma} = \mu^h\,.
$$
We define the operator $\mathcal{Q}=\{\mathcal{Q}_1,\mathcal{Q}_2\}$ using this lifting as 
\begin{equation}\label{eq:op-Q}
\mathcal{Q}(\mu^h) = 
\left\{
\begin{array}{rl}
\{{u}^{h,\mu}_{1,C},0 \} \in \widetilde{V}^h_C & \mbox{if $k=1$}\\[1.5ex]
-\{0, {u}^{h,\mu}_{2,C} \} \in \widetilde{V}^h_C & \mbox{if $k=2$}
\end{array}
\right. \,.
\end{equation}
With this definition the first assertion in \eqref{eq:Q} holds trivially with $C_1=1$:
$$
 b(\mathcal{Q}(\mu^h);\mu^h) 
 =
 \left(\mathcal{Q}_1(\mu^h)-\mathcal{Q}_2(\mu^h),\mu^h\right)_{0,\gamma}
 =
 \left(\mu^h,\mu^h\right)_{0,\gamma}
 =
 \|\mu^h\|^2_W\,.
$$
To prove the second assertion in \eqref{eq:Q}, we start by noting that
\begin{equation}\label{eq:Qnorm}
\|\mathcal{Q}(\mu^h)\|^2_V 
=
\|\mathcal{Q}_1(\mu^h)\|^2_{0,\Omega_1} + \|\mathcal{Q}_2(\mu^h)\|^2_{0,\Omega_2}
=
\| {u}^{h,\mu}_{k,C}\|^2_{0,\Omega_k}\,.
\end{equation}
Next, we recall the equivalence relation  \cite[p.386, Lemma 9.7]{Ern_04_BOOK}
\begin{equation}\label{eq:equiv}
\underline{C}_{\nu} h^{\nu} |\bm{v}|^2 \le \|{v}^h\|^2_{0,\omega} \le \overline{C}_{\nu} h^{\nu} |\bm{v} |^2
\end{equation}
that holds for any nodal finite element function ${v}^h$ defined on a quasi-uniform finite element partition of a bounded region $\omega\subset \mathbb{R}^{\nu}$ and its coefficient vector  $\bm{v} \in \mathbb{R}^n$.  
Application of the upper bound in \eqref{eq:equiv} to the lifting ${u}^{h,\mu}_{k,C}$, together with the 
definition \eqref{eq:lift} of its coefficient vector, yields
 $$
 \| {u}^{h,\mu}_{k,C}\|^2_{0,\Omega_k}
\le
\overline{C}_{\nu} h^{\nu}_{k} |\bm{u}^{\mu}_{k,C}|^2 =  \overline{C}_{\nu} h^{\nu}_{k} |\bm{\mu} |^2 \,.
$$
Since $\bm{\mu}$ is also the coefficient of the Lagrange multiplier $\mu^h$,
application of the lower bound in \eqref{eq:equiv} with $\nu-1$ gives the inequality
 $$
 |\bm{\mu} |^2 \le \frac{1}{\underline{C}_{\nu-1} h^{(\nu-1)}_{k}} \|\mu^h\|^2_W\,.
 $$
 In conjunction with \eqref{eq:Qnorm}, these two inequalities combine to produce the bound 
 $$
  \|\mathcal{Q}(\mu^h)\|^2_V
  \le 
  \frac{\overline{C}_{\nu}}{\underline{C}_{\nu-1}}h_{k} \|\mu^h\|^2_W.
 $$
 Therefore, the second assertion of the lemma holds with $C_2 = \overline{C}_{\nu}/\underline{C}_{\nu-1}$ and $\alpha =1/2$.
 \end{proof}
 
 Having established the existence of the operator $\mathcal{Q}$, the proof of the \textit{inf-sup}
 condition is straightforward. The following lemma provides the formal argument.
 
 \begin{lemma}\label{lem:infsup}
Assume that $h_{k} < 1$. Then,  the bilinear form $b(\cdot,\cdot)$ satisfies the \textit{inf-sup}
 condition \eqref{eq:infsup}.
 \end{lemma}
 \begin{proof}
 Let $\mu^h\in \widetilde{W}^h$ be an arbitrary reduced space function. Using the assumption on the mesh size and the properties \eqref{eq:Q} of the operator $\mathcal{Q}$, we find that 
 $$
\sup_{\{{v}^h_{1,C},{v}^h_{2,C}\}\in \widetilde{V}^h_{C}\times \widetilde{V}^h_{C}} 
\frac{b({v}^h_{1,C},{v}^h_{2,C};{\mu}^h)}{\|\{{v}^h_{1,C},{v}^h_{2,C}\}\|_V}
\ge
\frac{b(\mathcal{Q}(\mu^h),\mu^h)}{\|\mathcal{Q}(\mu^h)\|_V}
\overset{(\ref{eq:Q}a)}{=} \frac{\|\mu^h\|^2_W}{\|\mathcal{Q}(\mu^h)\|_V}
\overset{(\ref{eq:Q}b)}\ge
 \frac{\underline{C}_{\nu-1}}{\overline{C}_{\nu}}h^{-1/2}_{k} \|\mu^h\|_W
 \ge \beta \|\mu^h\|_W,
 $$
 with $\beta = \underline{C}_{\nu-1}/\overline{C}_{\nu}$. 
 \end{proof}
 
\begin{remark}\label{lem:trace}
The proofs of Lemma \ref{lem:Q} and Lemma \ref{lem:infsup} highlight the key role played by the  \emph{trace-compatibility} condition for our analysis. Specifically, this condition guarantees the existence of a lifting \eqref{eq:lift} of the Lagrange multiplier into the composite RB space, which is needed for the construction of the operator $\mathcal{Q}$. This operator is essential for showing that \eqref{eq:infsup} holds.
\end{remark}

This completes the verification of the assumptions necessary to assert that the Schur complement \eqref{eq:Schur-split} is SPD. Therefore, we can conclude that the IVR formulation for the composite coupled ROM-ROM problem is well-posed. 
 
\subsection{Composite reduced basis coupled ROM-FOM}\label{sec:analysis-ROM-FOM}
Let us now specialize the results of Section \ref{sec:analysis-ROM-FOM} to the case of the coupled 
ROM-FOM formulation \eqref{AdC:eq:SeparatedRR-ROM-FOM} (note that the FOM-ROM case, where a FOM is used in $\Omega_1$ and a ROM is used in 
$\Omega_2$, is analogous).  
To prove that the Schur complement
\begin{equation}\label{eq:Schur-split-ROM-FOM}
\widehat{S}: =
 \widehat{G}_1 \widetilde{M}_1^{-1} \widehat{G}_1^T + \widehat{G}_2 {M}_{2,D}^{-1} \widehat{G}_2^T
 \end{equation}
is symmetric and positive definite, we will use the same variational approach based on showing that the auxiliary mixed variational form \eqref{eq:mixed-aux} satisfies the conditions of Brezzi's theory. To that end,
 we specialize the functional setting from Section  \ref{sec:analysis-ROM-ROM} to the present case as follows.

First, we shall retain the norms \eqref{eq:norms} for the full order state space $V^h_D$ and the Lagrange multiplier space $W^h$. Next, we define the \emph{hybrid} state space $\widetilde{V}^h_H\subset V^h_D$ as 
 $$
 \widetilde{V}^h_{H} = \widetilde{S}^h_{1,C}\times {S}^h_{2,D},
$$
 where $\widetilde{S}^h_{1,C}= \widetilde{S}^h_{1,\gamma}\oplus\widetilde{S}^h_{1,0}$ is the composite RB space on $\Omega_1$. Finally, we set
 $$
 \widehat{W}^h
 =
 \left\{
 \begin{array}{ll}
 \widetilde{G}^h_{1,\gamma}, & \mbox{for option rLM} \\[1ex]
 G^h_2, &  \mbox{for option fLM}
 \end{array} \,.
 \right.
 $$
 As in the case of \eqref{AdC:eq:SeparatedRR}, it is straightforward to show that restriction of the auxiliary form \eqref{eq:mixed-aux} to $\widetilde{V}^h_H$ and $\widehat{W}^h$ produces the matrix on the left-hand side of \eqref{AdC:eq:SeparatedRR-ROM-FOM}. 
 Likewise, since $\widetilde{V}^h_{H}$ is a subspace of $V^h_D$, the first Brezzi condition is trivially satisfied. 
 
Specialized to  \eqref{AdC:eq:SeparatedRR-ROM-FOM}, the second (\textit{inf-sup}) condition now reads: 
for any $\widehat{\mu}^h\in\widehat{W}^h$ the form $b(\cdot,\cdot)$ satisfies the inequality
 \begin{equation}\label{eq:infsup-RF}
\sup_{\{{v}^h_{1,H},{v}^h_{2,H}\}\in\widetilde{V}^h_{H}\times\widetilde{V}^h_{H}} 
\frac{b({v}^h_{1,H},{v}^h_{2,H};{\mu}^h)}{\|\{{v}^h_{1,H},{v}^h_{2,H}\}\|_V}
\ge \beta \| \widehat{\mu}^h\|_W,
\end{equation}
with a mesh-independent constant $\beta$. Again, as in the case of \eqref{AdC:eq:SeparatedRR}, the proof of \eqref{eq:infsup-RF} requires an operator  $\mathcal{Q}: \widehat{W}^h\mapsto\widetilde{V}^h_H$ such that  
for every $\widehat{\mu}^h\in\widehat{W}^h$ there holds
\begin{equation}\label{eq:Q-RF}
(a)\ \ \|\widehat{\mu}^h\|^2_{W} \le C_1 b(\mathcal{Q}(\widehat{\mu}^h);\widehat{\mu}^h)
\quad
\mbox{and}
\quad
(b) \ \ \|\mathcal{Q}(\widehat{\mu}^h)\|_{V} \le C_2 \widehat{h}^\alpha_{\gamma} \|\widehat{\mu}^h\|_W \,,
\end{equation}
where $\widehat{h}_{\gamma} =h_{1,\gamma}$ for option rLM, $\widehat{h}_{\gamma} =h_{2,\gamma}$  for option fLM,
$\alpha \ge 0$, and $C_1$, $C_2$ are positive constants independent of $\widehat{h}_{\gamma}$. 

Since both options for the Lagrange multiplier space are trace-compatible, the operator $\mathcal{Q}$ for the coupled ROM-FOM case can be easily defined by a minor modification of \eqref{eq:op-Q} to account for the particular option used. 
Specifically, we set 
\begin{equation}\label{eq:op-Q-RF}
\mathcal{Q}(\mu^h) = 
\left\{
\begin{array}{rl}
\{{u}^{h,\mu}_{1,C},0 \} \in \widetilde{V}^h_C & \mbox{for option rLM}\\[1.5ex]
-\{0, {u}^{h,\mu}_{2,D} \} \in \widetilde{V}^h_C & \mbox{for option fLM}
\end{array}
\right. \,.
\end{equation}
where ${u}^{h,\mu}_{1,C}$ is the lifting of $\widehat{\mu}^h$ defined in Lemma \ref{lem:Q} and ${u}^{h,\mu}_{2,D}$ is the lifting of $\widehat{\mu}^h$ into the finite element space $S^h_{2,D}$ defined by the coefficient vector 
$\bm{u}^{\mu}_{2,D} = (\widehat{\bm{\mu}},\bm{0})$.
It is straightforward to verify that the operator \eqref{eq:op-Q-RF} satisfies the inequalities in \eqref{eq:Q-RF}. Then, using the same arguments as in the proof of Lemma \ref{lem:infsup}, one can show that \eqref{eq:infsup-RF} holds with the same constant $\beta$ as in that lemma. This establishes all conditions necessary for the Schur complement \eqref{eq:Schur-split-ROM-FOM} to be SPD.

A few comments about these results are now in order. As we have mentioned earlier, it is possible to prove the full column rank property of the transpose constraint matrices in the coupled ROM-ROM and ROM-FOM problems directly using purely algebraic tools. However, this approach fails to account for the fact that we are dealing with matrices obtained through a discretization process followed by a Galerkin projection. Such matrices carry an implicit dependence on the discretization mesh parameter and the size of the RBs employed in the projection. As a result, the condition numbers and the ranks of the matrices in the coupled FOM-FOM, ROM-ROM, and ROM-FOM problems also depend on these parameters. An algebraic approach treats these matrices as having a given fixed dimension and generally cannot reveal the dependence of condition numbers and ranks on the mesh size and the RB dimension.

In contrast, the \textit{inf-sup} condition not only establishes that these transpose constraint matrices have full column ranks, but it also provides  a lower bound on their smallest singular values; see, e.g., \cite{Bochev_06_ETNA}. 
In particular, by showing that  the \textit{inf-sup} conditions \eqref{eq:infsup} and \eqref{eq:infsup-RF} hold with mesh and RB-independent lower bounds, we effectively prove that the smallest singular values of the associated constraint matrices are bounded away from zero \emph{independently} of the mesh size and/or the dimensions of the reduced bases. 
To put it differently, by adopting a variational approach we are able to show that the transpose constraint matrices cannot become computationally rank-deficient both when one varies the mesh size of the coupled FOM-FOM and when one varies the dimension of the composite RB. This property is highly non-trivial to establish using algebraic approaches.

%% file: sec_NumResults.tex
\section{Numerical results}\label{sec:num}
The objectives of this section are two-fold. First, we aim to confirm numerically the theoretical analysis in Section \ref{sec:analysis}, specifically the fact that projection of the coupled FOM-FOM \eqref{AdC:eq:Idx1System} onto the composite RB spaces, using a trace-compatible Lagrange multiplier space, leads to coupled ROM-ROM \eqref{AdC:eq:SeparatedRR} and ROM-FOM \eqref{AdC:eq:SeparatedRR-ROM-FOM} problems with non-singular Schur complements, independently of the \REV{RB} size. 
\REV{Our second goal is to demonstrate numerically the accuracy of the partitioned schemes in the two distinct simulation settings outlined in Section \ref{AdC:sec:intro}. We recall that the first one is characterized by a continuous diffusion coefficient, i.e., we consider \eqref{AdC:strongForm} with $\kappa_1=\kappa_2$. Keeping in mind the distinctions with the use of this term in the ROM literature highlighted in Section \ref{sec:organization}, we shall refer to this case as the ``domain decomposition'' (DD) setting. 
We also recall that the second setting represents a bona fide transmission problem (TP) characterized by a discontinuous diffusion coefficient. Section \ref{sec:single} presents reproductive and predictive results for the DD setting, while Section \ref{sec:multi} provides predictive tests in the TP setting. Since reproductive tests for the latter largely mirror the results for former, Section \ref{sec:multi} includes only \REVpk{predictive} TP tests.}

Following our previous work  \cite{DeCastro_23_INPROC}, we use 
the solid body rotation test from \cite{AdC:Leveque_96_SINUM}, specialized to \eqref{AdC:strongForm}. The computational domain for this test is the unit square $\Omega := (0,1) \times (0,1)$, the initial condition  comprises a cone, a cylinder, and a smooth hump (Figure \ref{AdC:fig:ICS}), and the advection
 field is defined as $\bm{a} := (0.5 - y, x - 0.5)$.
We split $\Omega$ into subdomains $\Omega_1 := (0,0.5) \times (0,1)$ and $\Omega_2 := (0.5,1) \times (0,1)$, impose homogeneous Dirichlet boundaries on all non-interface boundaries $\Gamma_i$, for $i=1,2$, and set the final time to be \REV{$T_f:=2 \pi$}, representing one full rotation.  

In all examples we use a uniform partition of  $\Omega$ into $64\times 64$ square elements yielding 4225 nodes in $\Omega$ and 2145 nodes in $\Omega_i$ for $i=1,2$, as seen in Figure \ref{AdC:fig:unifMesh}. 
\REV{It is easy to see that $\gamma^h_1 = \gamma^h_2$, i.e., the interface finite element partitions induced by the subdomain meshes are identical.} This setting eliminates error pollution due to non-matching interface grids from the numerical results and allows us to examine the ``pure'' properties of the partitioned schemes. In particular, in this setting, the IVR solution of the coupled FOM-FOM problem \eqref{AdC:eq:Idx1System} obtained by solving the subdomain equations in \eqref{AdC:eq:Idx1ODESystem} coincides, to machine precision, with a \emph{single domain solution} obtained by treating \eqref{AdC:strongForm} as a single PDE with a discontinuous coefficient; see \cite{AdC:CAMWA}.  
\REV{Finally, we note that} all results in this section were obtained by using the forward Euler method as the time discretization scheme.

\begin{figure}[!ht]
\begin{center}
	\subfigure[Initial conditions]{\includegraphics[width=0.475\linewidth]{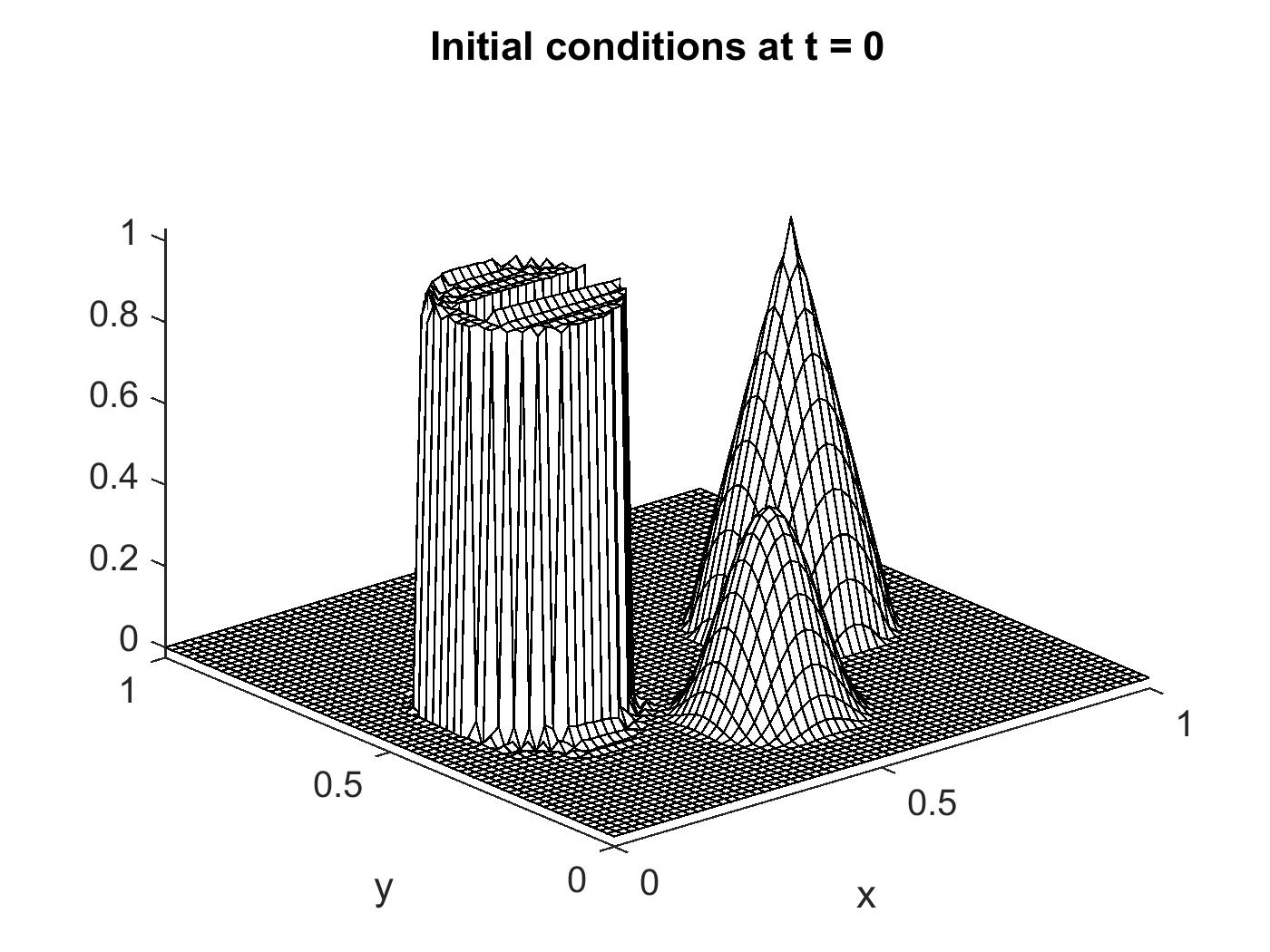}\label{AdC:fig:ICS}} 
	\subfigure[The finite element partitions $\Omega^h_1$ (blue) and $\Omega^h_2$ (red).]{\includegraphics[width=0.475\linewidth]{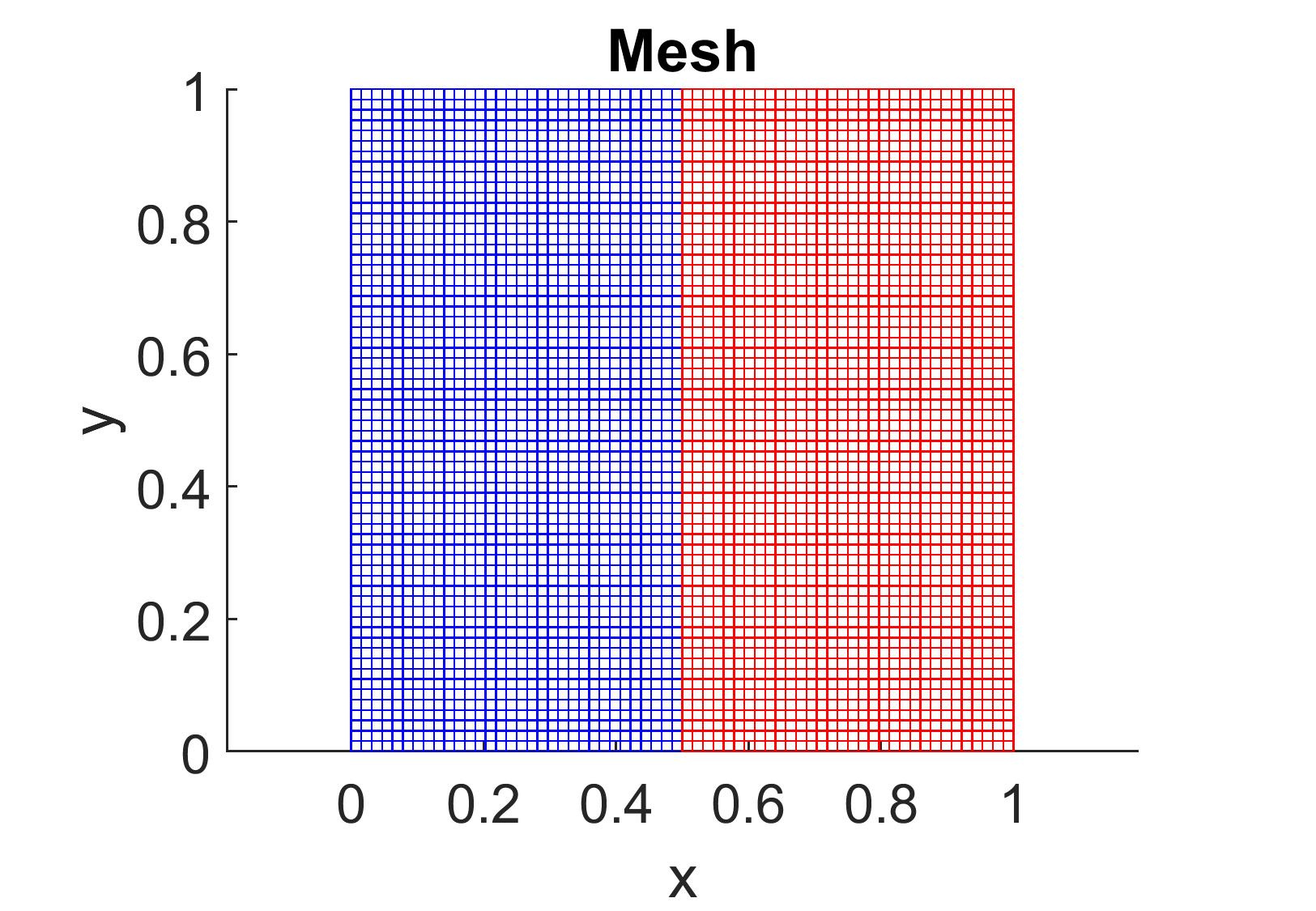}\label{AdC:fig:unifMesh}} \caption{Initial conditions, domain partitioning, and mesh for the  model 2D transmission problem.}
\end{center}
\end{figure}

\subsection{\REV{Domain decomposition setting}}\label{sec:single}
The model problem is parameterized with respect to the diffusion coefficient $\kappa_i$, \REV{which for the domain decomposition case is the same on both subdomains, i.e., $\kappa_1 = \kappa_2:=\kappa$.}
We perform the reproductive tests using a \REV{RB} obtained from solution snapshots corresponding to \REV{$\kappa= 10^{-5}$.} For the predictive tests, we define the reduced bases from snapshots computed with \REV{$\kappa = 10^{-2}$ and $\kappa= 10^{-8}$}, and then simulate the model problem  with \REV{$\kappa= 10^{-5}$}.

\REV{To obtain the subdomain solution snapshots\REVpk{,} we restrict a single domain finite element solution of \eqref{AdC:strongForm} to $\Omega_1$ and $\Omega_2$, respectively.} 
\REV{The single domain solution} is computed using the time step $\Delta t_s = 9.156\times 10^{-4}$ for $\kappa = 10^{-2}$, and $\Delta t_s = 1.684\times 10^{-3}$ for $\kappa = 10^{-5}$ and $10^{-8}$.
\REV{These time steps are determined from  the Courant-Friedrichs-Lewy (CFL) condition.}
\REV{Since both the reproductive and the predictive tests are performed for  $\kappa= 10^{-5}$, the 
partitioned solution scheme employed for all coupled formulations uses the time step $\Delta t =  1.684\times 10^{-3}$.}
To demonstrate the importance of the composite \REV{RB} and trace-compatible Lagrange multipliers for the properties of the Schur complement, we present results for the partitioned solution of the coupled ROM-ROM and FOM-ROM problems, implemented with the composite RB and with alternative choices for the Lagrange multiplier (LM) space. We use as a benchmark the single domain solution introduced earlier. For the coupling of a ROM to a FOM we choose to implement the FOM on $\Omega_1$ and the ROM on $\Omega_2$, but note that similar performance is achieved if this choice were reversed. Still, in what follows, for consistency, we label this formulation as FOM-ROM.
 %
%
To summarize, we perform tests using the following schemes:
\begin{itemize}
\item RR-rLM: partitioned solution of the coupled ROM-ROM problem \eqref{AdC:eq:SeparatedRR}. 
\item FR-fLM: partitioned solution of the coupled FOM-ROM \eqref{AdC:eq:SeparatedRR-ROM-FOM} with the (full) LM space  $G^h_1$.
\item FR-rLM: partitioned solution of the coupled FOM-ROM \eqref{AdC:eq:SeparatedRR-ROM-FOM} with the (reduced) LM space  $\REV{\Phi}_{2,\gamma}$.
\item FF-fLM partitioned solution of the coupled FOM-FOM \eqref{AdC:eq:Idx1System}.
\end{itemize}
The partitioned schemes above are supported by rigorous theory that asserts the existence of \REV{well-posed} Schur complements for the associated coupled problems, i.e., Schur complements that are provably non-singular and have bounded condition numbers. As an example of a formulation that is not supported by such a theory, we consider
\begin{itemize}
\item RR-fLM: partitioned solution of the coupled ROM-ROM problem \eqref{AdC:eq:SeparatedRR}  with the (full) LM space  $G^h_1$.
\end{itemize}

\subsubsection{Reproductive results}
First, we present the results for the reproductive case.
With the snapshot time step set to $\Delta t_s =1.684\times 10^{-3}$, 3732 snapshots are collected. A prerequisite for an effective ROM is the rapid decay of the singular values. We first confirm that this is indeed the case and that most of the energy, defined in \eqref{AdC:eq:snapEnergy}, is contained within a much smaller subset of the snapshots. 
\begin{figure}[t]
  \begin{center}
   \includegraphics[width=0.5\textwidth]{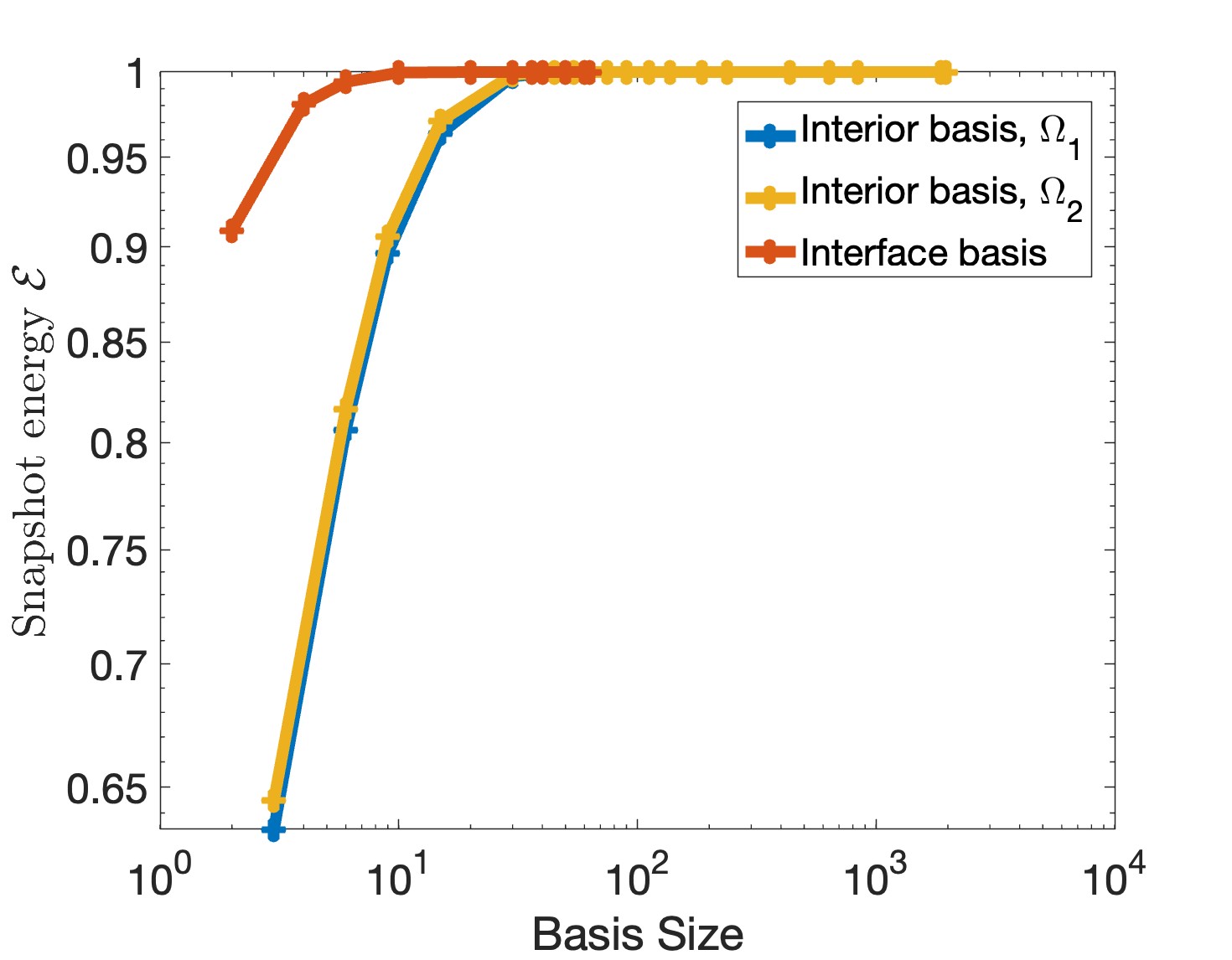}
  \end{center}
  \vspace{-2ex}
\caption{Snapshot energy \eqref{AdC:eq:snapEnergy} as a function \REV{of interior $\REV{\Phi}_{i,0}$ and interface $\REV{\Phi}_{i,\gamma}$ basis sizes} in the reproductive regime.} \label{AdC:fig:snapEnergy}
\end{figure}
The plots in Figure \ref{AdC:fig:snapEnergy} show the energy retained in the interior $\REV{\Phi}_{i,0}$ and interface $\REV{\Phi}_{i,\gamma}$ \REV{RB} sets as a function of their 
respective sizes, $d_{i,0}$ and $d_{i,\gamma}$. 
The plot reveals that just $d_{1,0}=24$, $d_{2,0} = 21$ and $d_{i,\gamma}=6$ interior and interface modes, 
respectively, are sufficient to capture 99\% of the energy in $X_{i,0}$ and $X_{i,\gamma}$. 
Setting  $d_{1,0}=57$, $d_{2,0} = 50$, and $d_{i,\gamma}=18$ captures 99.999\% of the snapshot energies. 
In what follows, we denote the size of the composite \REV{RB} 
$\REV{\Phi}_{i,C}=\{\REV{\Phi}_{i,\gamma},\REV{\Phi}_{i,0}\}_{i=1,2}$ 
as $d_{i,C} = d_{i,0}+d_{i,\gamma}$.

To assess the accuracy of the partitioned solutions of the coupled ROM-ROM and FOM-ROM problems, 
we report their relative errors with respect to the single domain solution of the model problem \REV{with the same diffusion coefficient as used for the reproductive tests, i.e., $\kappa= 10^{-5}$.}
We define these errors as 
\begin{align}\label{AdC:eq:relError}
\REV{\epsilon(t) := \frac{||\{\bm{u}^{t}_{1,P},\bm{u}^{t}_{2,P}\} - \{\bm{u}^{t}_{1,S},\bm{u}^{t}_{2,S}\}||_V}{||\{\bm{u}^{t}_{1,S},\bm{u}^{t}_{2,S}\}||_V},}
\end{align}
where $\|\cdot\|_V$ is the norm defined in \eqref{eq:norms},
$\REV{\{\bm{u}^{t}_{1,S},\bm{u}^{t}_{2,S}\}}$ 
is the single domain solution of the model problem, and $\REV{\{\bm{u}^{t}_{1,P},\bm{u}^{t}_{2,P}\}}$ denotes a partitioned solution of the coupled ROM-ROM, FOM-ROM or FOM-FOM problems, \REV{at a chosen time $t \in [0, 2\pi]$}.

When using the composite \REV{RB}, one can set the 
dimensions $d_{i,0}$ and $d_{i,\gamma}$ for the interior and interface bases independently. 
Here, we choose $d_{i,\gamma}$ to be two-fifths of the total composite 
basis size $d_{i,C}$, i.e., we set 
\begin{equation}\label{eq:RBdim}
d_{i,\gamma} = \frac{2}{5}  \Big(d_{i,\gamma} + d_{i,0}  \Big) \implies d_{i,\gamma} = \frac{2}{3} d_{i,0}.
\end{equation}
Note that the dimension of the reduced Lagrange multiplier space in the coupled ROM-ROM \eqref{AdC:eq:SeparatedRR} is given by either $d_{1,\gamma}$ or $d_{2,\gamma}$.

When setting $d_{i,\gamma}$, one also has to account for the fact that the total number of modes  $d^{\max}_{i,\gamma}$ available for the construction of the interface \REV{RB} $\REV{\Phi}_{i,\gamma}$ is, in general, smaller than the number  $d^{\max}_{i,0}$ available for the construction of $\REV{\Phi}_{i,0}$. As a result, direct application of \eqref{eq:RBdim} may yield values for $d_{i,\gamma}$ that exceed the number $d^{\max}_{i,\gamma}$ of interface modes available. To avoid this, we further refine the choice of the interface dimension according to
$$d_{i,\gamma} = \min\Big\{\frac{2}{3} d_{i,0},d^{\max}_{i,\gamma} \Big\}.$$
In all our simulations $d^{\max}_{i,\gamma} = 63$.

\begin{figure}[t!]
  \begin{center}
    \includegraphics[width=0.6\textwidth]{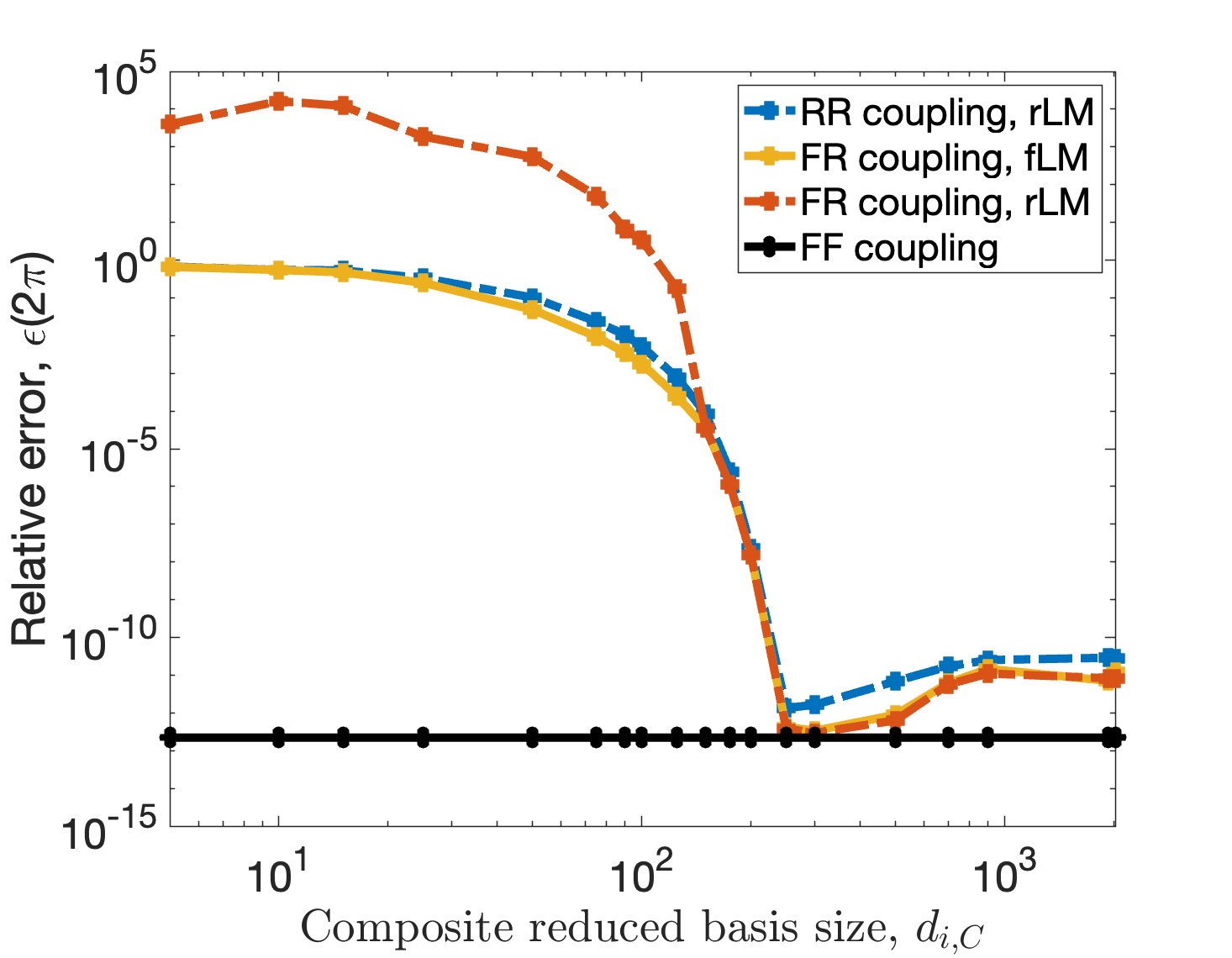}
  \end{center}
  \vspace{-2ex}
\caption{Relative error \eqref{AdC:eq:relError} at \REV{the final time} $T_f=2\pi$ of the partitioned solution for each coupled formulation as a function of the composite reduced basis size $d_{i,C}=d_{i,0}+d_{i,\gamma}$  in the reproductive regime.} \label{AdC:fig:L2err}
\end{figure}
\begin{figure}[h!]
\centering
\subfigure[$d_{i,0} = 15, d_{i,\gamma} = 10$]{\includegraphics[scale=.10]{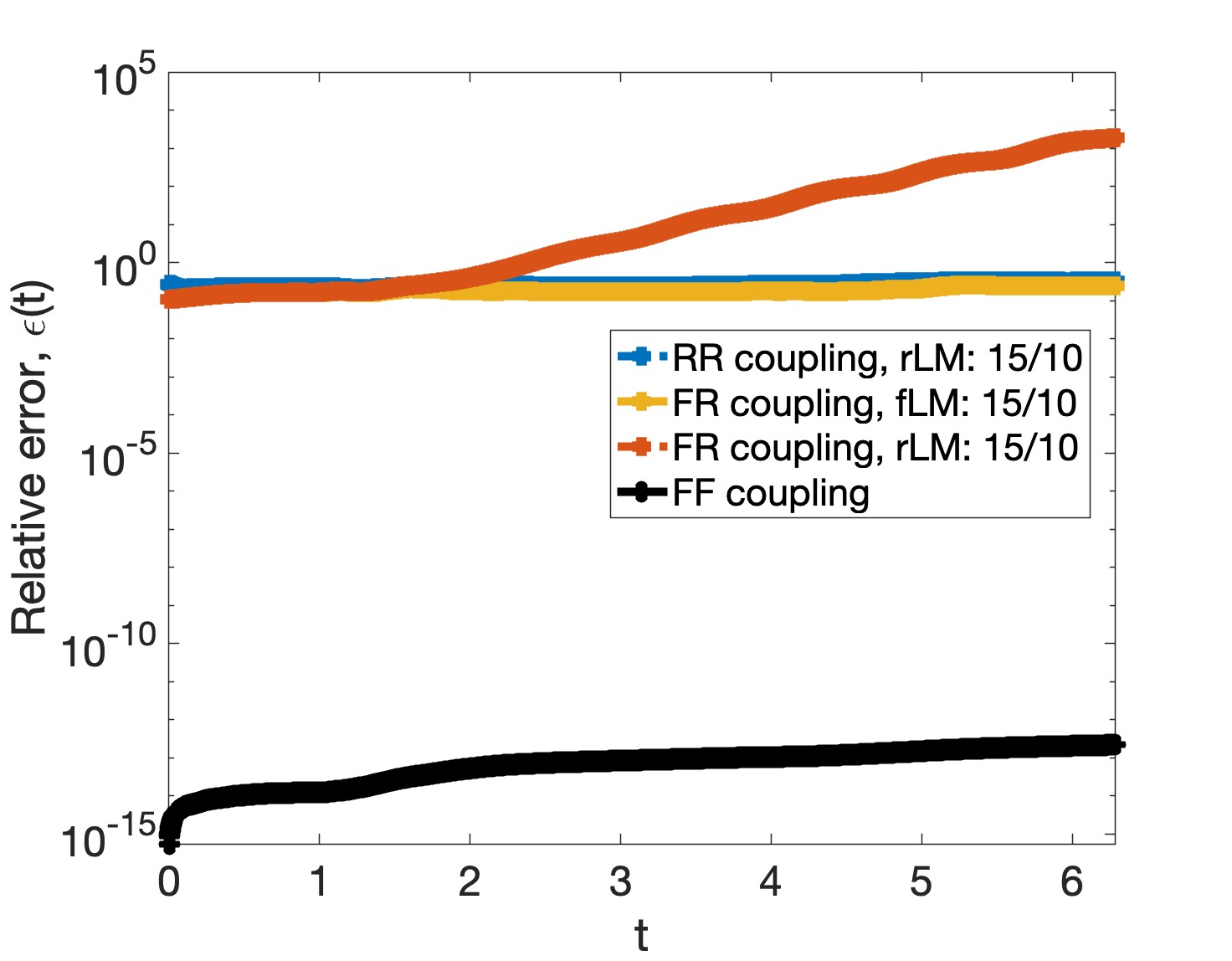}\label{AdC:fig:errVTime1510}} 
\subfigure[$d_{i,0} = 60, d_{i,\gamma} = 40$]{\includegraphics[scale=.10]{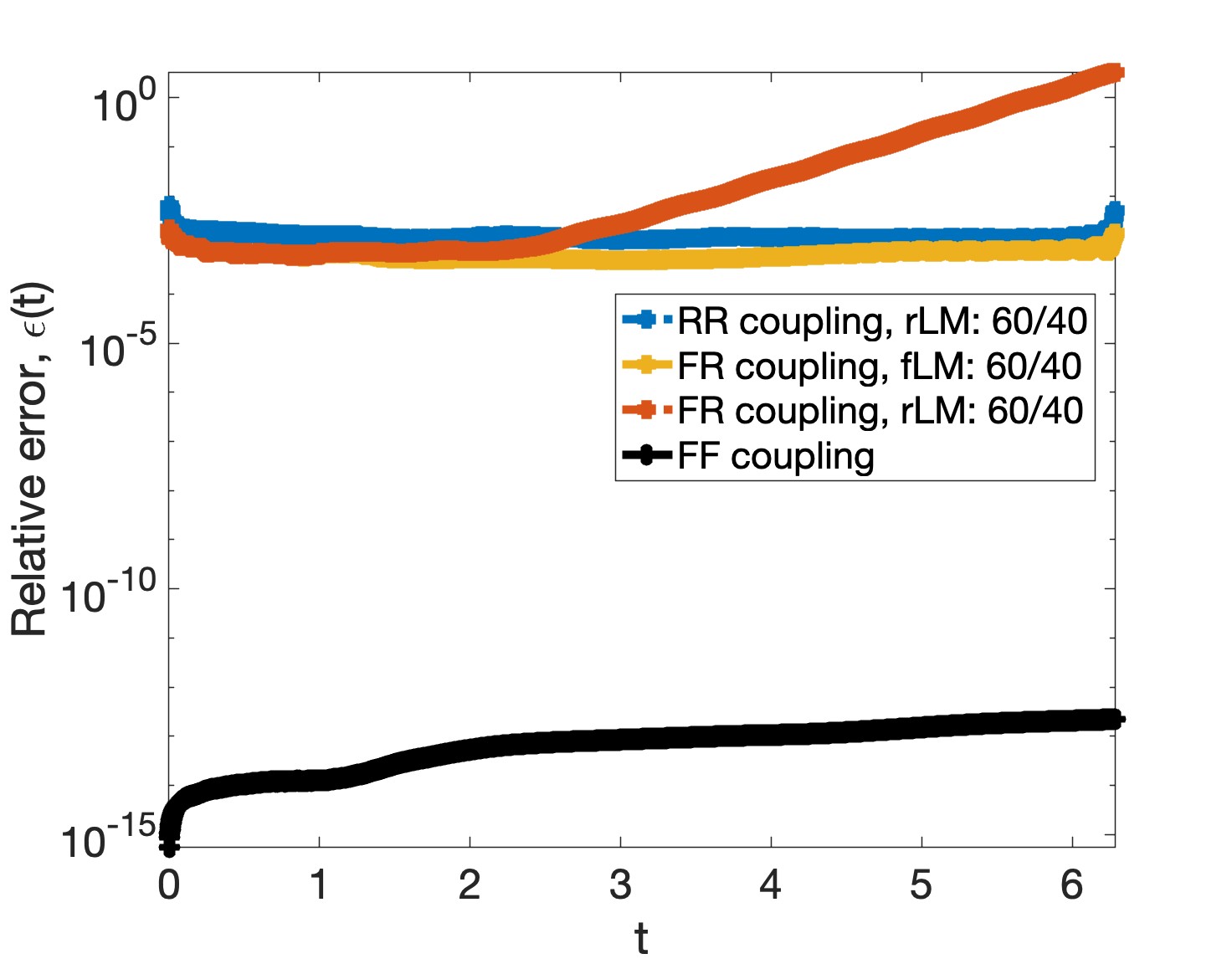}\label{AdC:fig:errVTime6040}} 
\subfigure[$d_{i,0} = 90, d_{i,\gamma} = 60$]{\includegraphics[scale=.10]{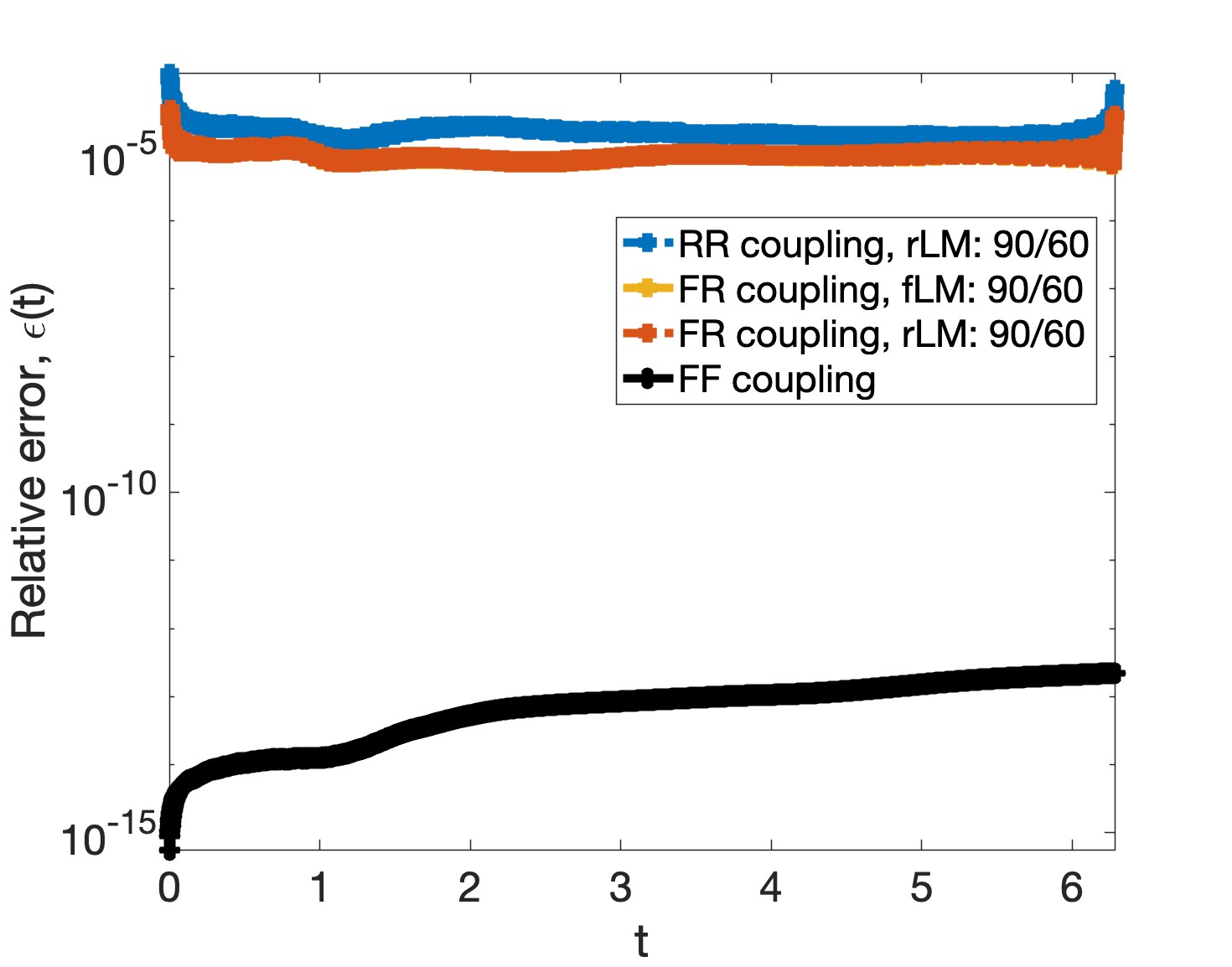}\label{AdC:fig:errVTime9060}} 
\subfigure[$d_{i,0} = 237, d_{i,\gamma} = 63$]{\includegraphics[scale=.10]{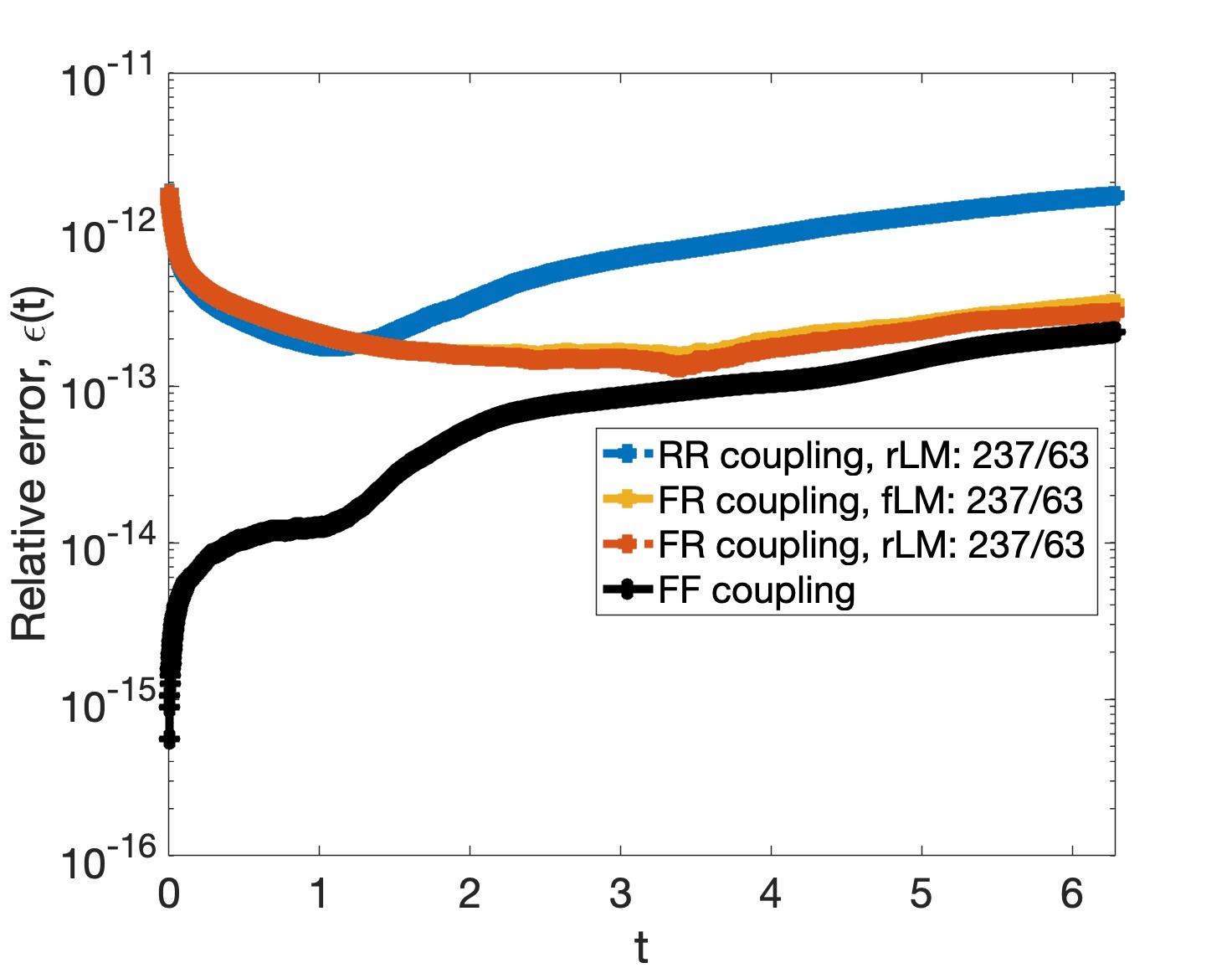}\label{AdC:fig:errVTime23763}} 
\subfigure[$d_{i,0} = 1953, d_{i,\gamma} = 63$]{\includegraphics[scale=.10]{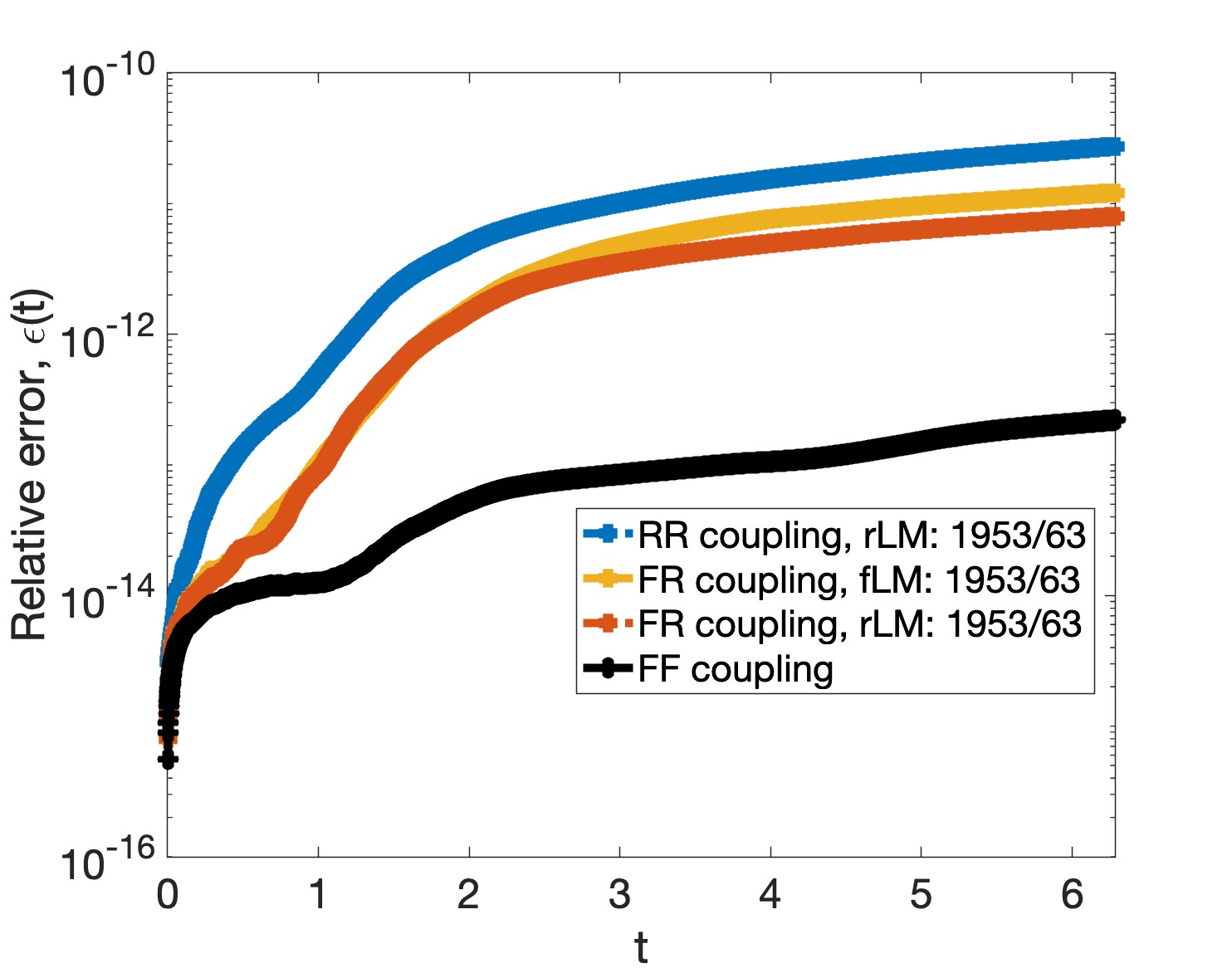}\label{AdC:fig:errVTime195363}} 
\subfigure[Single domain ROM]{\includegraphics[scale=.08]{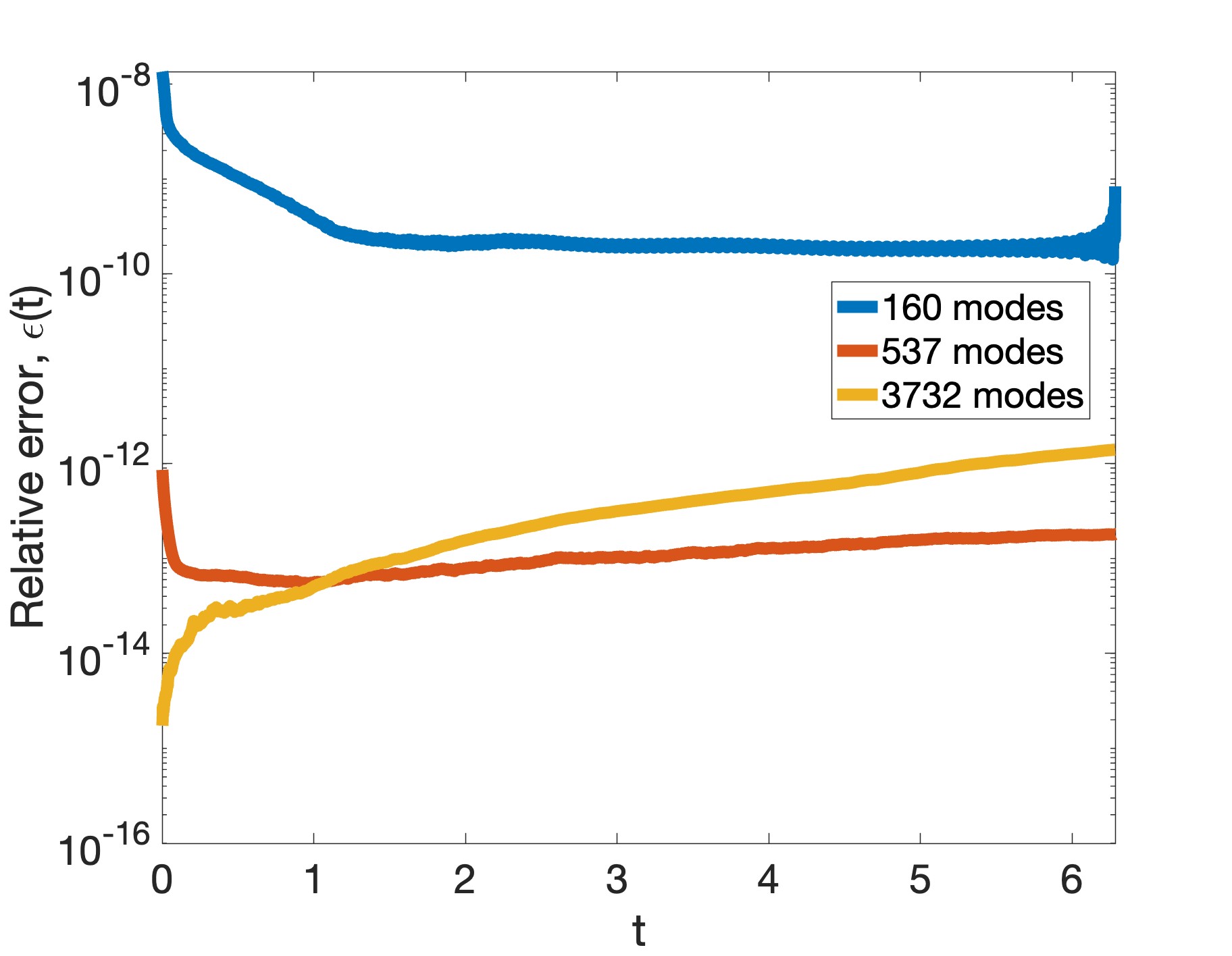}\label{AdC:fig:L2Time_SingleROM_ReprodSinglePhysics}} 
\caption{\REV{Relative error \eqref{AdC:eq:relError} of the partitioned solution for select RB sizes as a function of time in the reproductive regime.}}\label{AdC:fig:L2errVsTime_BasisSizes}
\end{figure}

We first examine the \REV{behavior of the} relative error \eqref{AdC:eq:relError} \REV{when our partitioned schemes are applied to} \REV{coupled} formulations \REV{with provably well-posed Schur complements.} 
\REV{Figure \ref{AdC:fig:L2err} plots $\epsilon(T_f)$, i.e., the relative error at the final time, as a function of the composite \REV{RB} size $d_{i,C}$.
The plots in this figure show that} as the total size of the composite \REV{RB} $d_{i,C}$ is increased, the error \REV{at the final time} in the partitioned solutions of the coupled ROM-ROM and FOM-ROM problems approaches that of the partitioned solution of the coupled FOM-FOM problem, as expected for a reproductive test. 
\REV{At the same time, we observe that, while $\epsilon(T_f)$ for the coupled FOM-ROM problem with the reduced space Lagrange multiplier is essentially the same as for the other formulations when $d_{i,C}$ is large enough, it is significantly higher for small RB sizes.}

\REV{To examine this issue further we plot the relative error \emph{as a function of time} for ``small'' ($d_{i,C}\le 100$, Fig.~\ref{AdC:fig:L2errVsTime_BasisSizes}(a,b)), ``medium'' ($15\REVpk{0}\le d_{i,C}\le 300$, Fig.~\ref{AdC:fig:L2errVsTime_BasisSizes}(c,d)), and ``large'' ($d_{i,C} = 2016$, Fig.~\ref{AdC:fig:L2errVsTime_BasisSizes}(e)) composite RB sizes. Note that the ``large'' RB has the same number of modes as the FOM. 
For comparison, in Fig.~\ref{AdC:fig:L2errVsTime_BasisSizes}(f) we plot $\epsilon(t)$ for three instances of a single domain ROM. \REVpk{The three instances have approximately the same total number of modes,  i.e., $d_{1,0}+d_{2,0}+d_{i,\gamma}$, as the ``small'', ``medium'' and ``large'' RBs in Fig.~\ref{AdC:fig:L2errVsTime_BasisSizes}(b), (d), and (e), respectively.}}%

\REV{The plots in Figure \ref{AdC:fig:L2errVsTime_BasisSizes} clearly show that for ``medium'' and ``large'' RB sizes $\epsilon(t)$ for all coupled formulations is stable and roughly comparable to that of the single domain ROM. 
However, the behavior of $\epsilon(t)$ for the FR-rLM formulation deviates significantly from that of the other formulations when the RB size is small. The plots in Fig.~\ref{AdC:fig:L2errVsTime_BasisSizes}(a,b) show that, up to $t\approx 2$, the relative errors of all formulations have roughly the same magnitude. However, as time integration continues past this time, $\epsilon(t)$ for the FR-rLM formulation begins to grow, while the relative error of all other formulations remains about the same. Therefore, the large size of $\epsilon(T_f)$ for this formulation is caused by the accumulation of errors during the explicit time stepping.}

\REV{Although a rigorous analysis of the source of these errors is beyond the scope of this paper, below we offer some insights into the probable cause for the growth of $\epsilon(t)$ for the FR-rLM formulation with a ``small'' RB size. Before we provide the details, let us remark that the behavior of $\epsilon(t)$ in this case does not contradict the analysis in Section \ref{sec:analysis} because our theory asserts well-posedness of the Schur complement \eqref{eq:Schur-split}, which is \emph{independent} of time. In fact, the plots in Figure \ref{AdC:fig:condS}, that will be discussed in more detail shortly, reveal that for ``small'' RB sizes condition number of the Schur complement for the FR-rLM formulation is actually lower than that for the benchmark coupled FOM-FOM problem.}

\REV{Since FR-rLM and FR-fLM only differ in the choice of the Lagrange multiplier space\REVpk{,} let us compare and contrast the enforcement of the coupling condition \eqref{AdC:interfaceConditions} in these formulations. This task is greatly simplified by the fact that in all our examples $\gamma_{1}^h = \REVpk{\gamma_2^{h}}$. As a result, $S^h_{1,\gamma} = S^h_{2,\gamma}$, $n_{1,\gamma}=n_{2,\gamma}=n_{\gamma}$, and $G_{1,\gamma} = G_{2,\gamma} = G_{\gamma}$, where $G_{\gamma}$ is a symmetric and positive definite interface mass matrix.}
\REV{In the FR-fLM formulation $\lambda^h\in G^h_1$. Taking into account that the FOM is defined on $\Omega_1$ and that $G_{1,\gamma} = G_{2,\gamma} = G_{\gamma}$, the last equation in \eqref{AdC:eq:SeparatedRR-ROM-FOM}  specializes to
\begin{equation}\label{eq:FR-fLM-interface}
G_{1,\gamma} \dot{\bm{u}}_{1,\gamma} - 
G_{2,\gamma}\Phi_{2,\gamma} \dot{\widetilde{\bm{u}}}_{2,\gamma} 
=
G_\gamma\left(\dot{\bm{u}}_{1,\gamma} - 
\Phi_{2,\gamma} \dot{\widetilde{\bm{u}}}_{2,\gamma} \right)
= 
0 \,.
\end{equation}
Since $G_{\gamma}$ is non-singular, it follows that 
\begin{equation}\label{eq:equality}
 \dot{\bm{u}}_{1,\gamma}  - \Phi_{2,\gamma} \dot{\widetilde{\bm{u}}}_{2,\gamma} = 0\,.
\end{equation}
Thus, in the FR-fLM formulation on \emph{matching interface grids} the coupling condition is enforced \emph{pointwise}.} 

\REV{In contrast, in the FR-rLM formulation $\bm{\lambda} = \Phi_{2,\gamma}\widetilde{\bm{\lambda}}$ and the  last equation in \eqref{AdC:eq:SeparatedRR-ROM-FOM} now assumes the form
\begin{equation}\label{eq:FR-rLM-interface}
\Phi_{2,\gamma}^TG_{1,\gamma} \dot{\bm{u}}_{1,\gamma} - 
\Phi_{2,\gamma}^TG_{2,\gamma}\Phi_{2,\gamma} \dot{\widetilde{\bm{u}}}_{2,\gamma}
=
\Phi_{2,\gamma}^T G_{\gamma}\left( \dot{\bm{u}}_{1,\gamma}  - \Phi_{2,\gamma} \dot{\widetilde{\bm{u}}}_{2,\gamma}\right)
=0 \,.
\end{equation}
Let $\Phi_{2,\gamma}^{\prime}$ be the $n_\gamma\times r_2-d_{2,\gamma}$ matrix of discarded left singular vectors from the SVD decomposition of the snapshot set $X_{2,\gamma}$. Then, \eqref{eq:FR-rLM-interface} \REVpk{permits} the existence \REVpk{of} a \REVpk{nonzero} vector $\bm{\delta}\in \mathbb{R}^{(r_2-d_{2,\gamma})}$ such that
$$
G_{\gamma}\left( \dot{\bm{u}}_{1,\gamma}  - \Phi_{2,\gamma} \dot{\widetilde{\bm{u}}}_{2,\gamma}\right) 
=
\Phi_{2,\gamma}^{\prime}\bm{\delta} \,.
$$
It follows that
\begin{equation}\label{eq:inequality}
 \dot{\bm{u}}_{1,\gamma}  - \Phi_{2,\gamma} \dot{\widetilde{\bm{u}}}_{2,\gamma} 
 = G_{\gamma}^{-1}\Phi_{2,\gamma}^{\prime}\bm{\delta}\,.
\end{equation} 
In other words, in the FR-rLM formulation\REVpk{,} the coupling condition is satisfied approximately whereas in the FR-fLM case this condition holds pointwise.}

\REV{The plots in \REVpk{Figure \ref{AdC:fig:snapEnergy}} reveal that just 6 interface modes are sufficient to capture $99\%$ of the snapshot energy \eqref{AdC:eq:snapEnergy} of $X_{i,\gamma}$. However, for small values of $d_{2,\gamma}$\REVpk{,} the snapshot energy contained in the discarded modes  $\Phi_{2,\gamma}^{\prime}$ may still be large enough so that accumulation of errors at each time step due to the right hand side in \eqref{eq:inequality} eventually destroys the accuracy of the numerical solution.}

\REV{We note that this phenomenon is not limited to the FOM-ROM formulation\REVpk{,} but is rather a consequence of coupling \REVpk{subdomain formulations that are imbalanced with respect to their accuracy.} Indeed, we observed similar error growth when a ROM with a ``large'' RB size was coupled to a ROM with a ``small'' RB size. Thus,  when coupling subdomain models that differ significantly in their resolution, a useful rule of a thumb would be to define the Lagrange multiplier space by always using the side with a higher resolution.}

\begin{figure}[t!]
\centering
	\subfigure[All coupled problems]{\includegraphics[scale=.15]{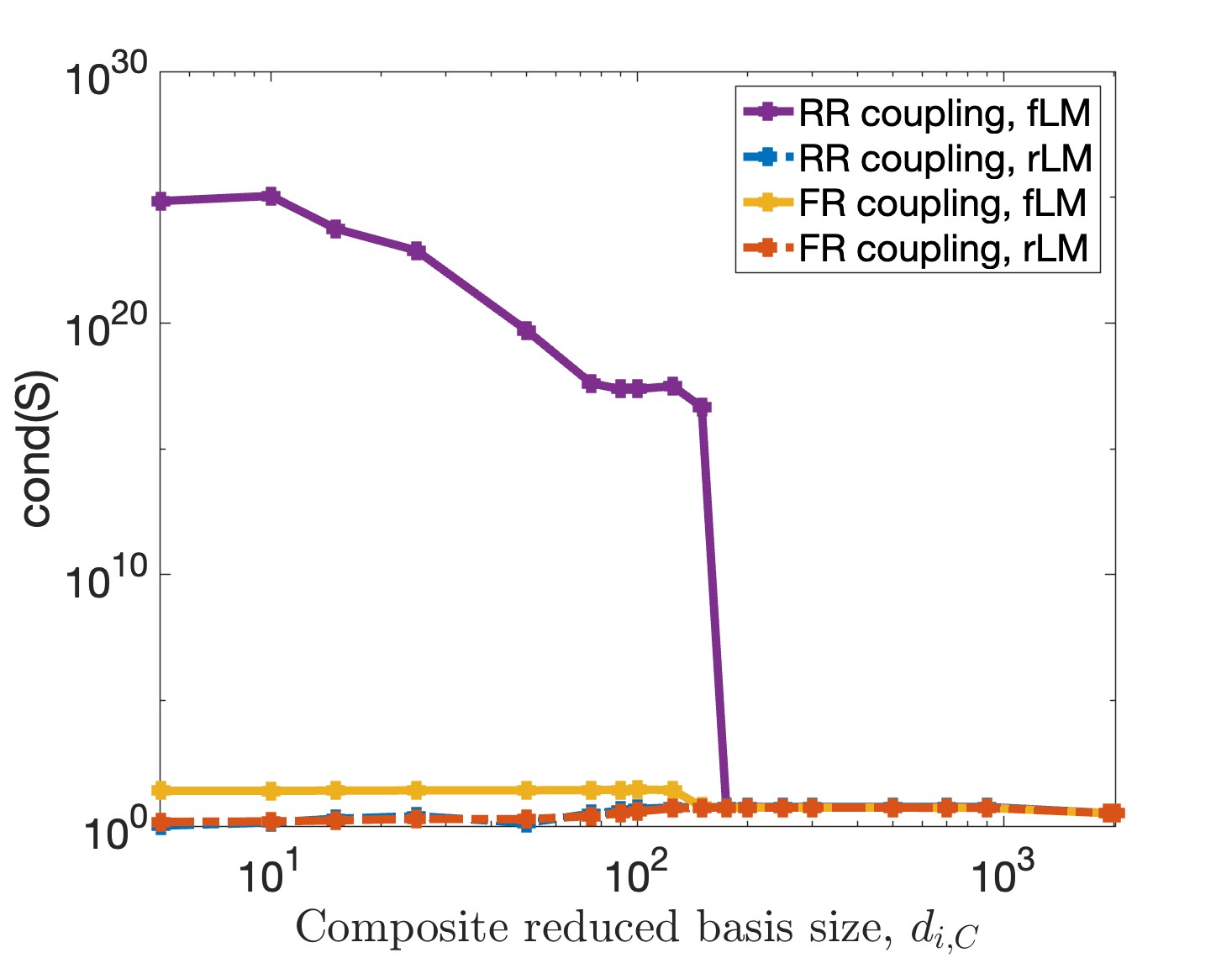}\label{AdC:fig:cond_RRfLM_reprod}} 
\subfigure[Provably well-posed Schur complement]{\includegraphics[scale=.15]{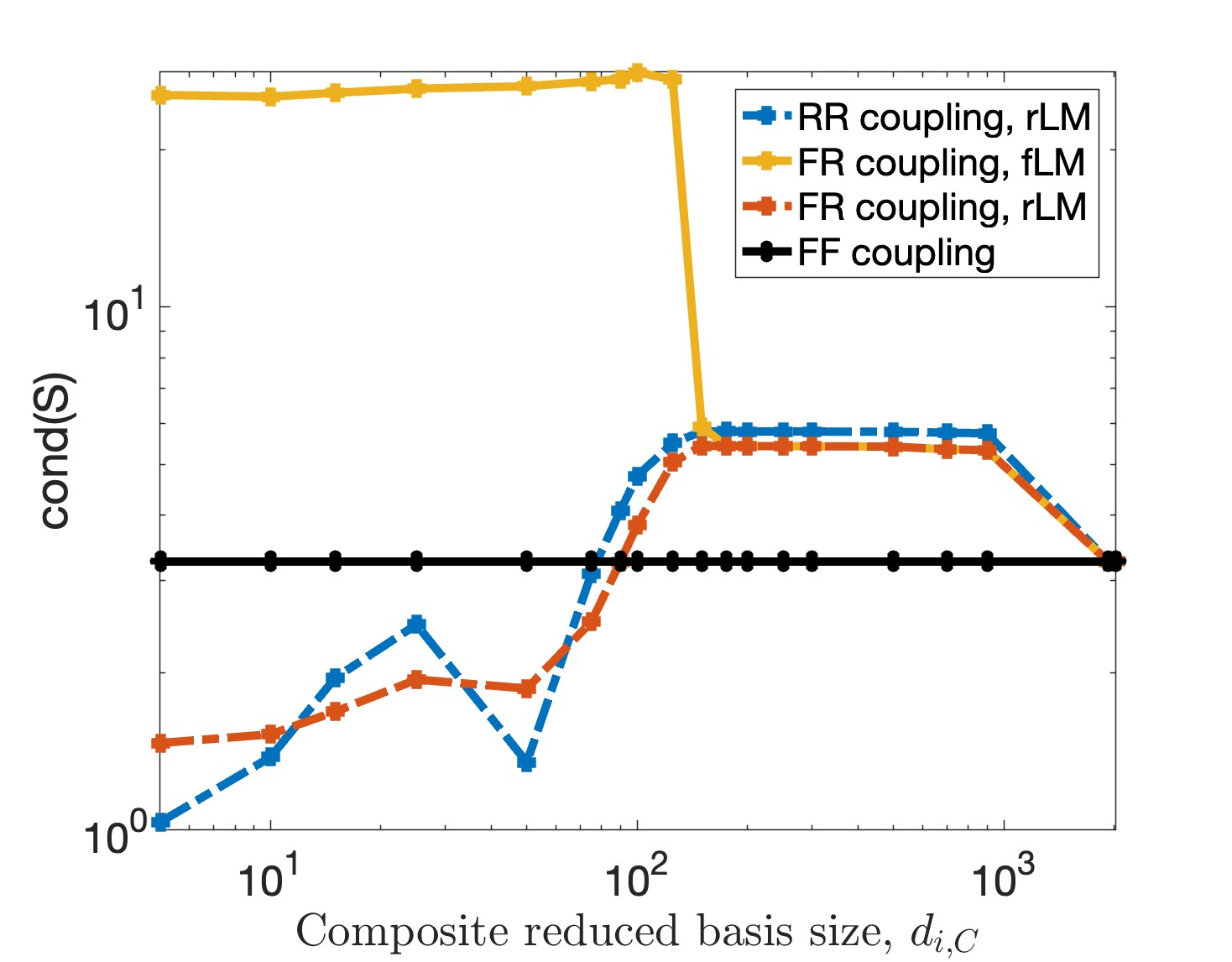}\label{AdC:fig:cond_FFfLM_reprod}}  
\caption{Condition number of Schur complement matrix for each coupled formulation as a function of the composite reduced basis size $d_{i,C}=d_{i,0}+d_{i,\gamma}$ size  in the reproductive regime.  Subfigure (a) reports results for all methods evaluated, whereas subfigure (b) focuses only 
on methods with provably well-posed Schur complements.}\label{AdC:fig:condS}
\end{figure}

\REV{Next, we highlight the importance of the Lagrange multiplier basis for the well-posedness of the Schur complement} in the coupled ROM-ROM and FOM-ROM problems. \REV{To that end,}  in Figure \ref{AdC:fig:condS}, we compare and contrast the condition number \REV{of this matrix for the couplings that satisfy the trace compatibility condition  
with the RR-fLM scheme that does not satisfy this condition.} Conditioning of the Schur complement is a measure of its ``well-posedness'' and can be used to confirm the conclusions from the  analysis in Section \ref{sec:analysis}.

The most important takeaway from Figure \ref{AdC:fig:condS} is that the Schur complements of the coupled ROM-ROM \eqref{AdC:eq:SeparatedRR} and FOM-ROM \eqref{AdC:eq:SeparatedRR-ROM-FOM} problems, which employ \emph{trace-compatible} Lagrange multiplier spaces conforming with the theory in Section \ref{sec:analysis}, have essentially constant condition numbers\footnote{\REV{Although, in Figure \ref{AdC:fig:condS}(b), the condition number of the Schur complement for the FR-fLM problem appears significantly larger than that for the other couplings,  the range for the $y$-axis in Figure \ref{AdC:fig:condS}(b) is $[0,28.1]$ with the upper limit representing $\max \rm{cond}(S)$ for the FR-fLM problem. Thus, in all cases $\rm{cond}(S)$ is of order at most $O(10)$}.}
 with respect to the reduced basis dimension. 
This corroborates numerically the theoretical conclusions asserting that the condition number of the Schur complement should be independent of the size of the reduced basis.  
Moreover, we see that using the trace-compatible Lagrange multiplier spaces required by the theory produces coupled ROM-ROM and FOM-ROM problems whose Schur complements are of the same order as those of the coupled FOM-FOM problem.
We recall that the latter also uses trace-compatible Lagrange multiplier spaces and is provably well-posed \cite{AdC:CAMWA}.

At the same time, using Lagrange multiplier spaces that are \emph{not trace-compatible} clearly leads to Schur complements whose condition number \emph{depends} on the reduced basis size. Specifically, by inspecting Figure \ref{AdC:fig:condS}(a), we see that when  the full order interface finite element space $G^h_2$ is used as a Lagrange multiplier space to couple two ROMs (RR-fLM), the Schur complement of the resulting coupled problem has very high condition numbers for smaller dimensions of the reduced basis. While the condition number does decrease as the reduced basis size increases, it is still high compared to that of the coupled FOM-FOM problem, and it only reaches a reasonable scale when the reduced basis size is larger than one would wish to consider. 
\begin{figure}[!t]
\centering
\subfigure[RR-rLM, 15/10 modes]{\includegraphics[scale=.15]{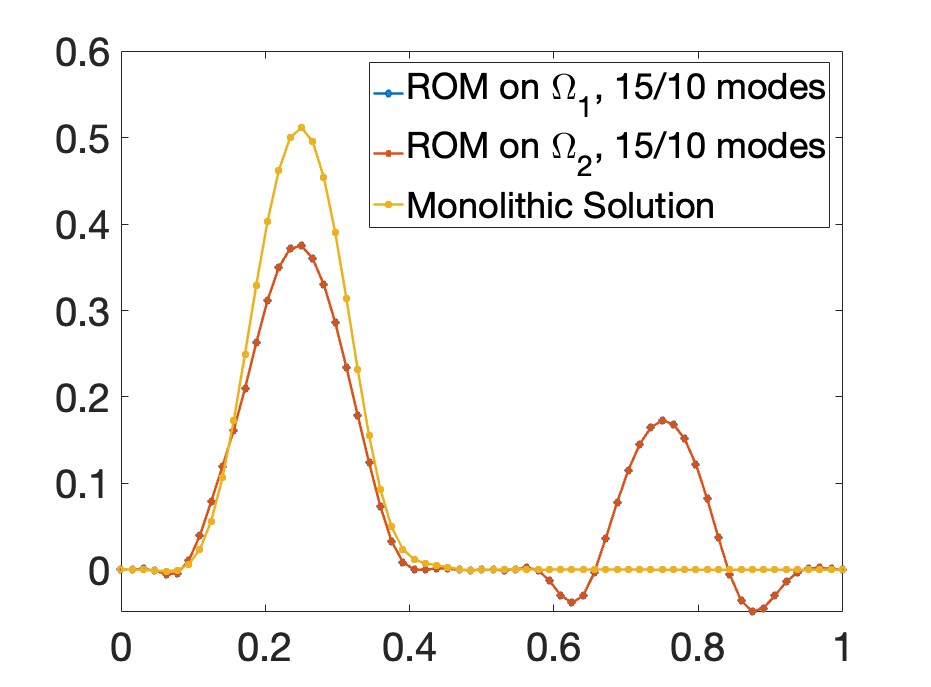}\label{AdC:fig:RR_Inter_15_10}} 
\hfill
\subfigure[\REV{FR-fLM, 15/10 modes}]{\includegraphics[scale=.078]{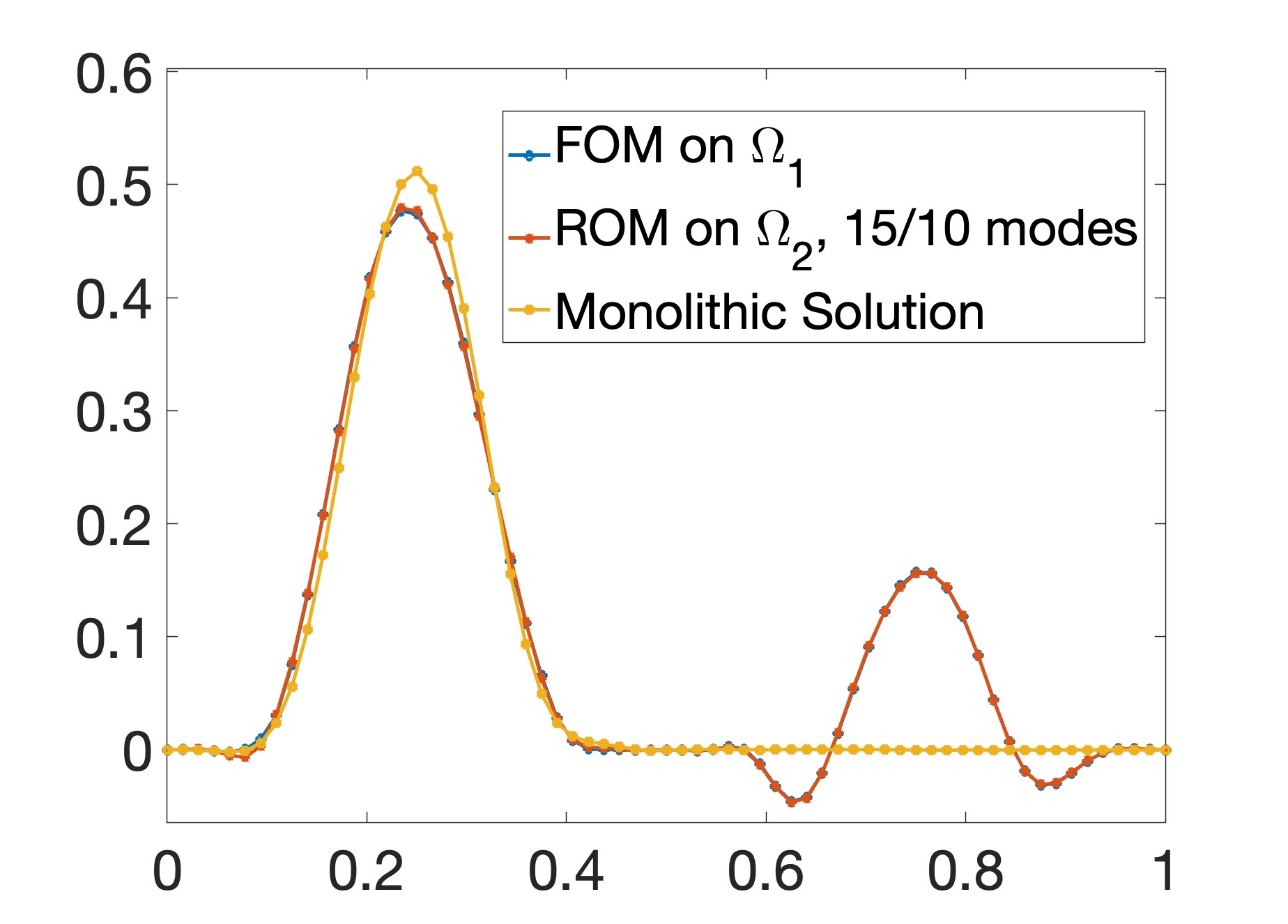}\label{AdC:fig:FRfLM_Inter_15_10}} 
\hfill
\subfigure[\REV{FR-rLM, 15/10 modes}]{\includegraphics[scale=.078]{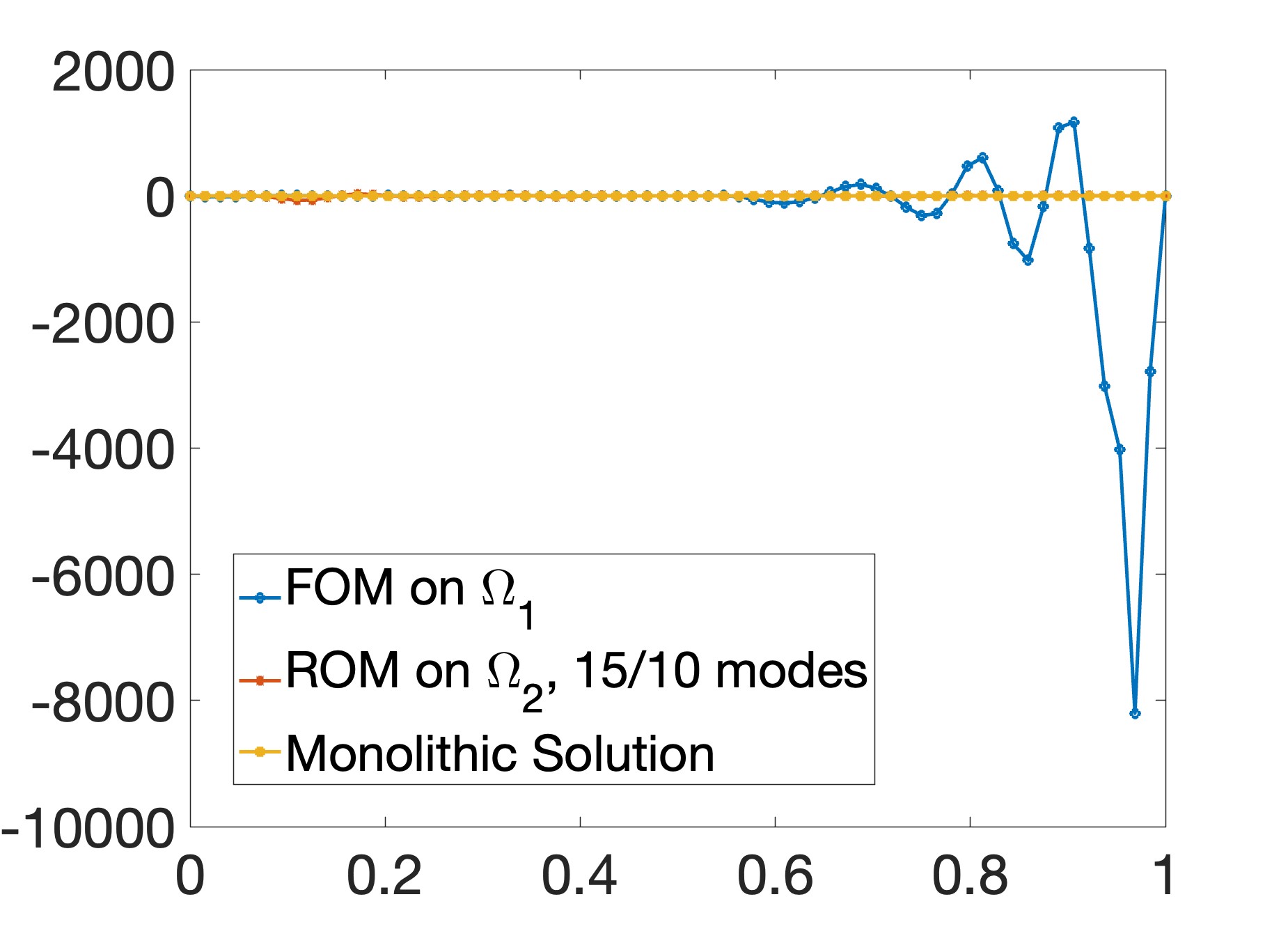}\label{AdC:fig:FRrLM_Inter_15_10}} 
\subfigure[RR-rLM 60/40 modes]{\includegraphics[scale=.15]{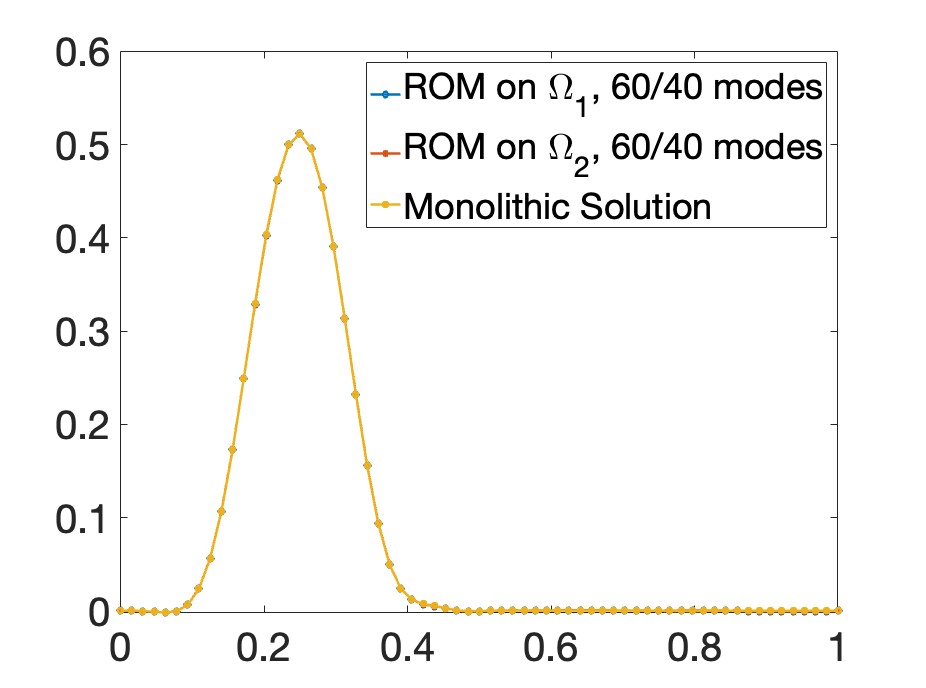}\label{AdC:fig:RR_Inter_60_40}} 
\hfill
\subfigure[FR-fLM 60/40 modes]{\includegraphics[scale=.15]{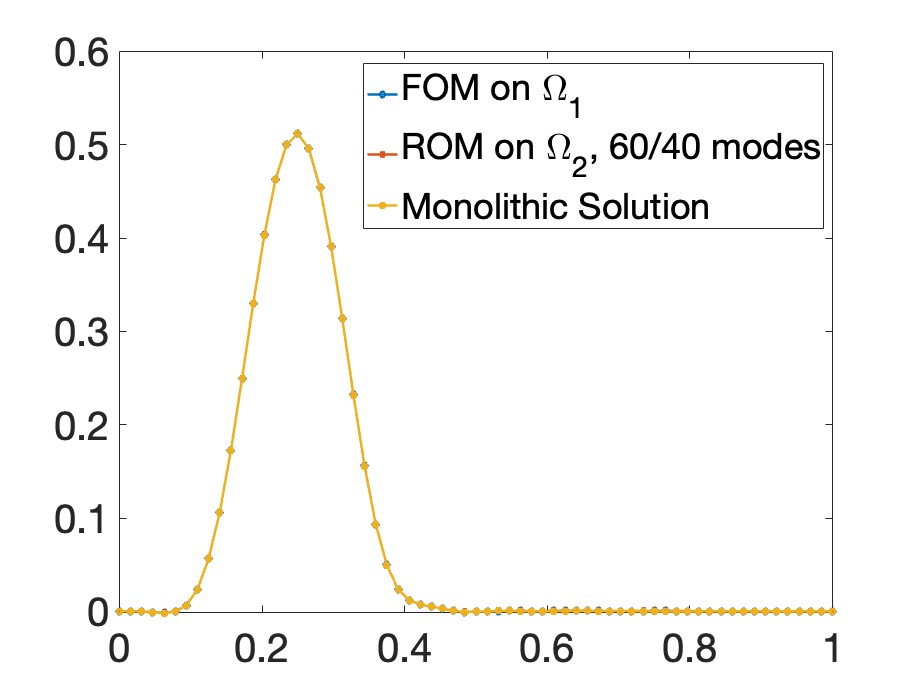}\label{AdC:fig:FRfLM_Inter_60_40}}
\hfill
 \subfigure[\REV{FR-rLM 60/40 modes}]{\includegraphics[scale=.078]{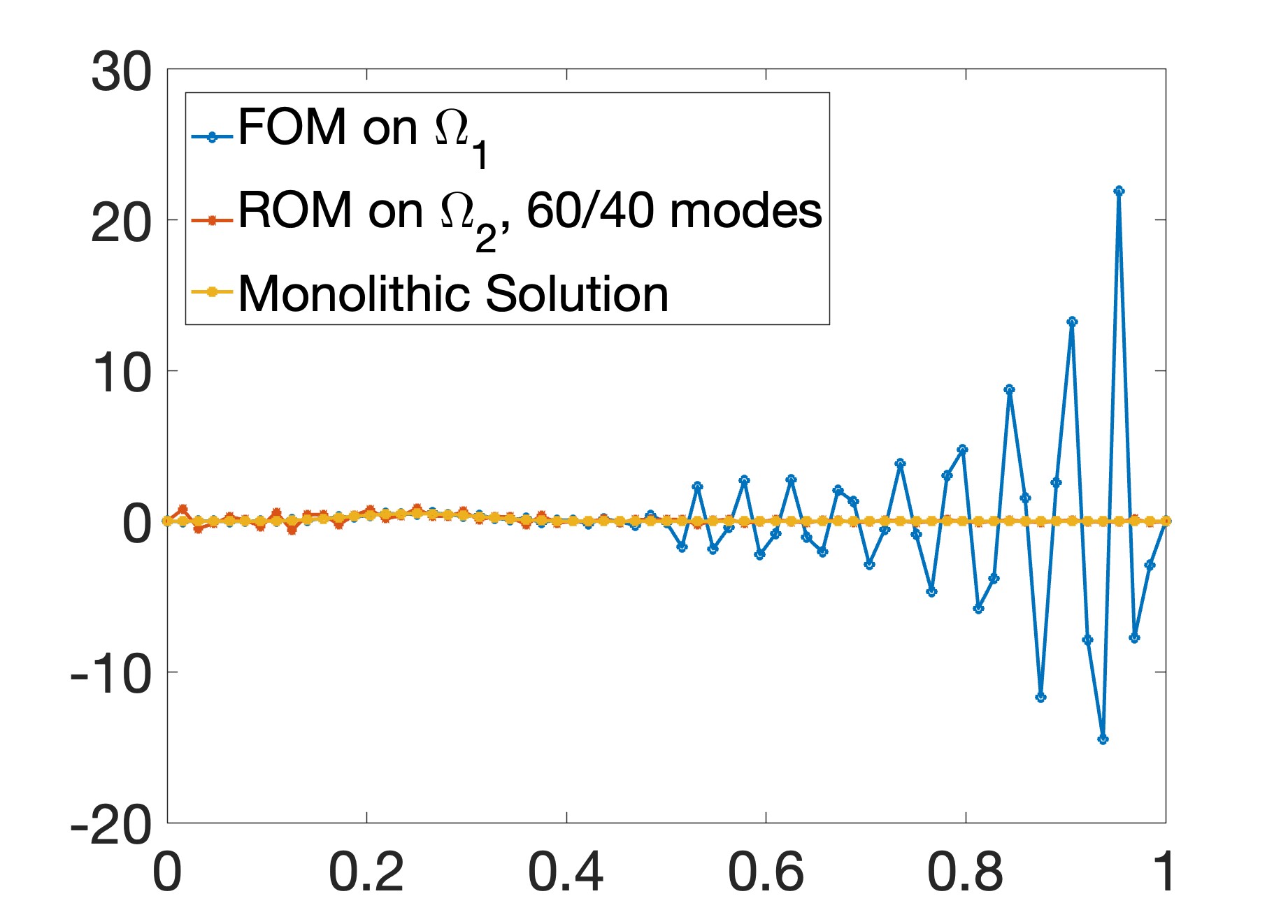}\label{AdC:fig:FRrLM_Inter_60_40}} 
\subfigure[\REV{RR-rLM 90/60 modes}]{\includegraphics[scale=.078]{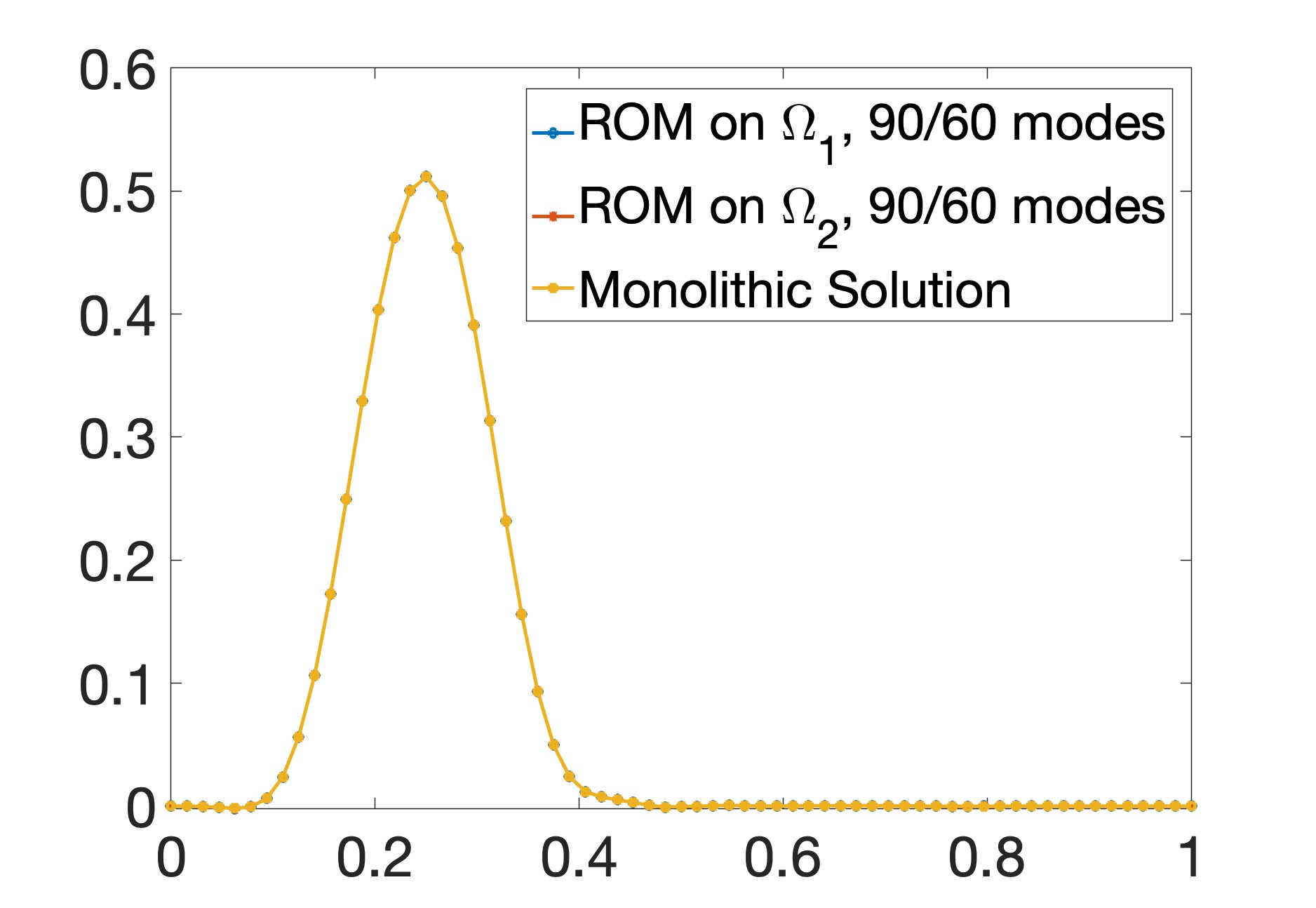}\label{AdC:fig:RRrLM_Inter_90_60}} 
\hfill
\subfigure[\REV{FR-fLM 90/60 modes}]{\includegraphics[scale=.078]{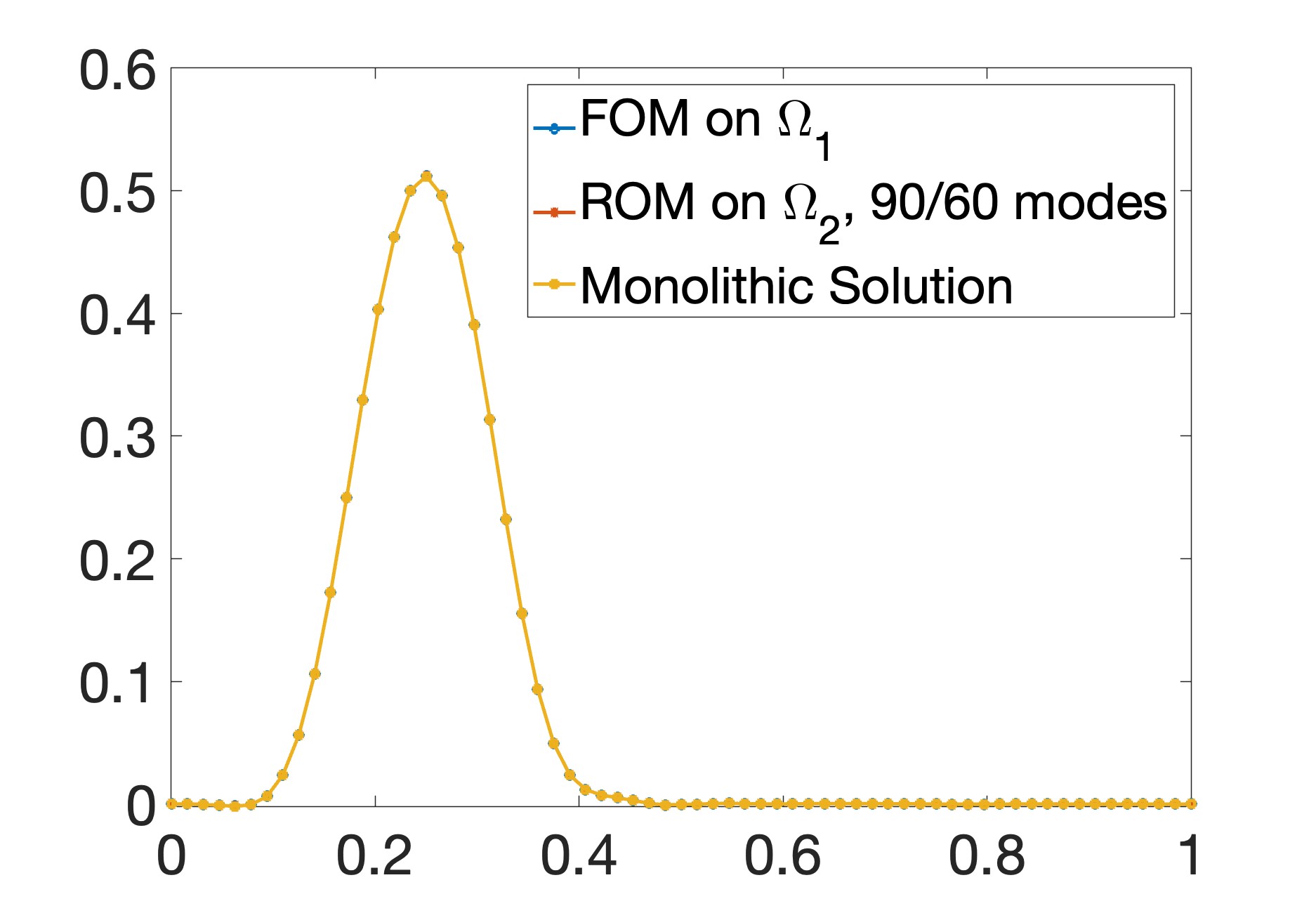}\label{AdC:fig:FRfLM_Inter_90_60}} 
\hfill
 \subfigure[FR-rLM 90/60 modes]{\includegraphics[scale=.15]{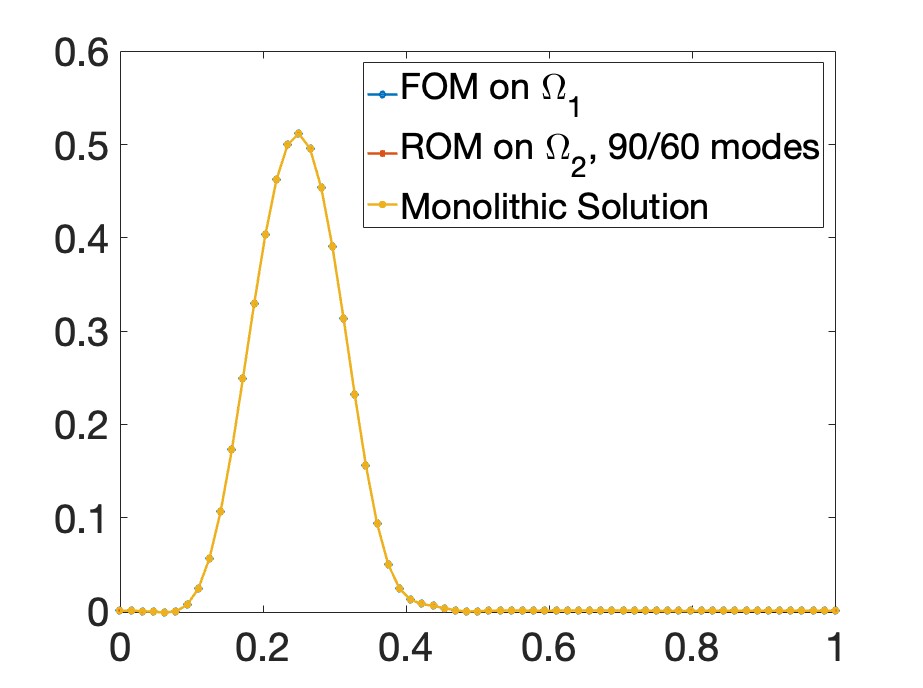}\label{AdC:fig:FRrLM_Inter_90_60}} 
\caption{\REV{Comparison of the interface states at $T_f$ for the partitioned schemes  with provably well-posed Schur complements vs. the single domain (monolithic) solution of the model problem. Reproductive test in the DD setting. The oscillations in the FR-rLM formulations with ``small'' RB sizes are due to accumulation of interface errors during the time integration caused by the approximate enforcement \eqref{eq:inequality} of the coupling condition. The legend ``$m/n$ modes'' corresponds to $m$ interior and $n$ interface modes.}}\label{AdC:fig:InterPlots}
\end{figure}

To understand the root cause for this behavior, recall that trace-compatibility requires every element of the Lagrange multiplier space to have a corresponding subdomain state whose trace on the interface matches the multiplier. This property is essential for the construction of the operator $\mathcal{Q}$ that plays a key role in showing that the Schur complement is non-singular; see Remark \ref{lem:trace}. 
At the same time, it is clear that when the subdomain states are represented by a reduced basis, their traces will not be able to reproduce every possible element of the  full interface space $G^h_2$, i.e., the latter is not trace-compatible.  As the size of the reduced basis for the states increases, trace compatibility is restored and the condition number of the Schur complement is reduced. This inflection point is clearly visible in Figure \ref{AdC:fig:cond_RRfLM_reprod} and corresponds to the instance when the traces of the reduced basis states contain the Lagrange multiplier space. 
An algebraic explanation of this behavior is that when the subdomain states are represented by a reduced basis, the full size Lagrange multiplier space will over-constrain the states leading to nearly rank deficient, or rank-deficient transpose constraint matrices. 

\begin{remark}
Numerical studies in \cite{DeCastro_23_INPROC} have shown that  the partitioned solution of the coupled ROM-ROM problem implemented with the full LM space $G^h_2$ has, to machine precision, the same errors as the partitioned solution of the this problem implemented with the reduced interface LM space $\REV{\Phi}_{i,\gamma}$. This suggests that solution errors alone do not tell the whole story about the quality of the Schur complement underpinning these partitioned solutions. Using Lagrange multiplier spaces that violate our analysis may result in seemingly reasonable errors for specific instances of discretization and ROM parameters, but is not guaranteed to work across all possible regimes. 
\end{remark}

Finally, in Figure \ref{AdC:fig:InterPlots}, we examine how well the partitioned solutions of the provably well-posed coupled problems satisfy the interface condition, which is an important characteristic of any partitioned scheme. To that end, we 
compare the partitioned solutions with the single domain solution of the model problem on the interface $\gamma$.
The plots in Figure \ref{AdC:fig:InterPlots} reveal that, for a large enough composite reduced basis, partitioned solutions of the coupled ROM-ROM and FOM-ROM problems have essentially the same accuracy on the interface as the solution of the coupled FOM-FOM problem (FF-fLM). We note that the FOM-ROM coupling with reduced LM space does require a larger basis size for the same accuracy as the FOM-ROM with full LM space, but it is capable of attaining the same level of errors. \REV{In particular, as discussed earlier, the oscillations in the FR-rLM formulations with ``small'' RB sizes are due to accumulation of interface errors during the time integration caused by the approximate enforcement \eqref{eq:inequality} of the coupling condition.}

\subsubsection{Predictive results}
We recall that for the predictive tests we collect snapshots at two diffusion coefficients, $\kappa_i = 10^{-2}$ and $10^{-8}$, and compute the partitioned solutions with $\kappa_i = 10^{-5}$ for $i=1,2$. 
\REV{We collect a total of 10,594 solution snapshots for $\kappa_i = 10^{-2}$ and $\kappa_i = 10^{-8}$ using the time steps stated at the beginning of this section.} Again, we note that the singular values decay rapidly, so that we are able to capture most of the snapshot energy within a much smaller subset of modes, 
as shown in Figure \ref{AdC:fig:snapEnergyPred}.  These plots reveal that only $d_{1,0}=26$, $d_{2,0} = 22$, and $d_{i,\gamma}=7$ interior and interface modes are sufficient to capture 99\% of the energy in $X_{i,0}$ and $X_{i,\gamma}$. Setting  $d_{1,0}=67$, $d_{2,0} = 58$, and $d_{i,\gamma}=20$ captures 99.999\% of the snapshot energies.  This is approximately the same rate of decay as shown in the reproductive case.
\begin{figure}[h]
  \begin{center}
   \includegraphics[width=0.45\textwidth]{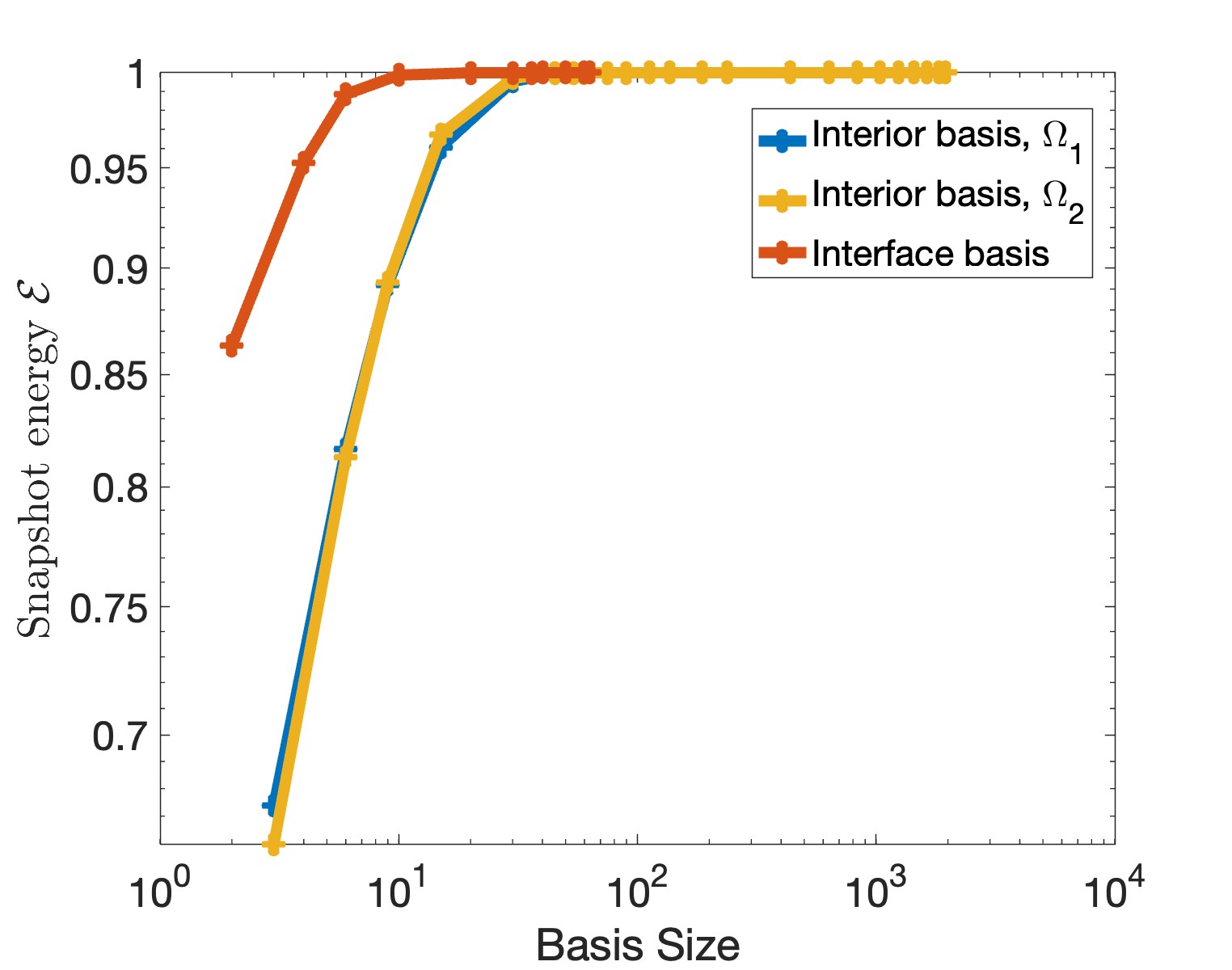}
  \end{center}
  \vspace{-2ex}
\caption{Snapshot energy \eqref{AdC:eq:snapEnergy} as a function \REV{of interior $\REV{\Phi}_{i,0}$ and interface $\REV{\Phi}_{i,\gamma}$ basis sizes} in the predictive regime.} \label{AdC:fig:snapEnergyPred}
\end{figure}

We first consider the relative errors \eqref{AdC:eq:relError} of the partitioned solutions of the coupled ROM-ROM and FOM-ROM problems. Figures  \REV{\ref{AdC:fig:L2errPred}--\ref{AdC:fig:L2errVsTime_BasisSizesPred}} summarizes our results.  In all three cases we see that, for a sufficiently large composite reduced basis size, partitioned solutions are able to achieve relative errors of roughly $10^{-2} $ or $10^{-3}$. This confirms the ability of the partitioned schemes presented in this paper to simulate the model problem for parameter values that have not been used in the construction of the reduced basis. 

\REV{\REVpk{In the predictive case, as previously in the reproductive case,} the approximate satisfaction of the coupling condition \eqref{eq:inequality} leads to accumulation of errors during the time integration that eventually destroys the accuracy of the solution for ``small'' RB sizes. As a result, as in the reproductive test, the relative error at the final time for the FR-rLM formulation  is significantly larger than that for the other formulations.  The plots in Figure \ref{AdC:fig:L2errVsTime_BasisSizesPred} reveal that, also similar to the reproductive case,  all formulations have comparable errors up to time $t \approx 2$, and that the error buildup for the FR-rLM formulation begins after that time.}

We next examine the condition numbers of the Schur complement matrices involved in the predictive tests. Again, we compare and contrast the conditioning of these matrices for \REV{the couplings that satisfy the trace compatibility condition and the} RR-fLM coupling that satisfies this condition only for large enough size of the reduced basis. These results are summarized in Figure \ref{AdC:fig:condSPred}. The plots in this figure mirror the behavior of the condition number observed in the reproductive test. Thus, one can conclude that the theoretical results in Section \ref{sec:analysis} remain in full force in the predictive regime as well, confirming the need for trace-compatible Lagrange multiplier spaces when constructing the coupled ROM-ROM and FOM-ROM problems.
\begin{figure}[t!]
  \begin{center}
    \includegraphics[width=0.5\textwidth]{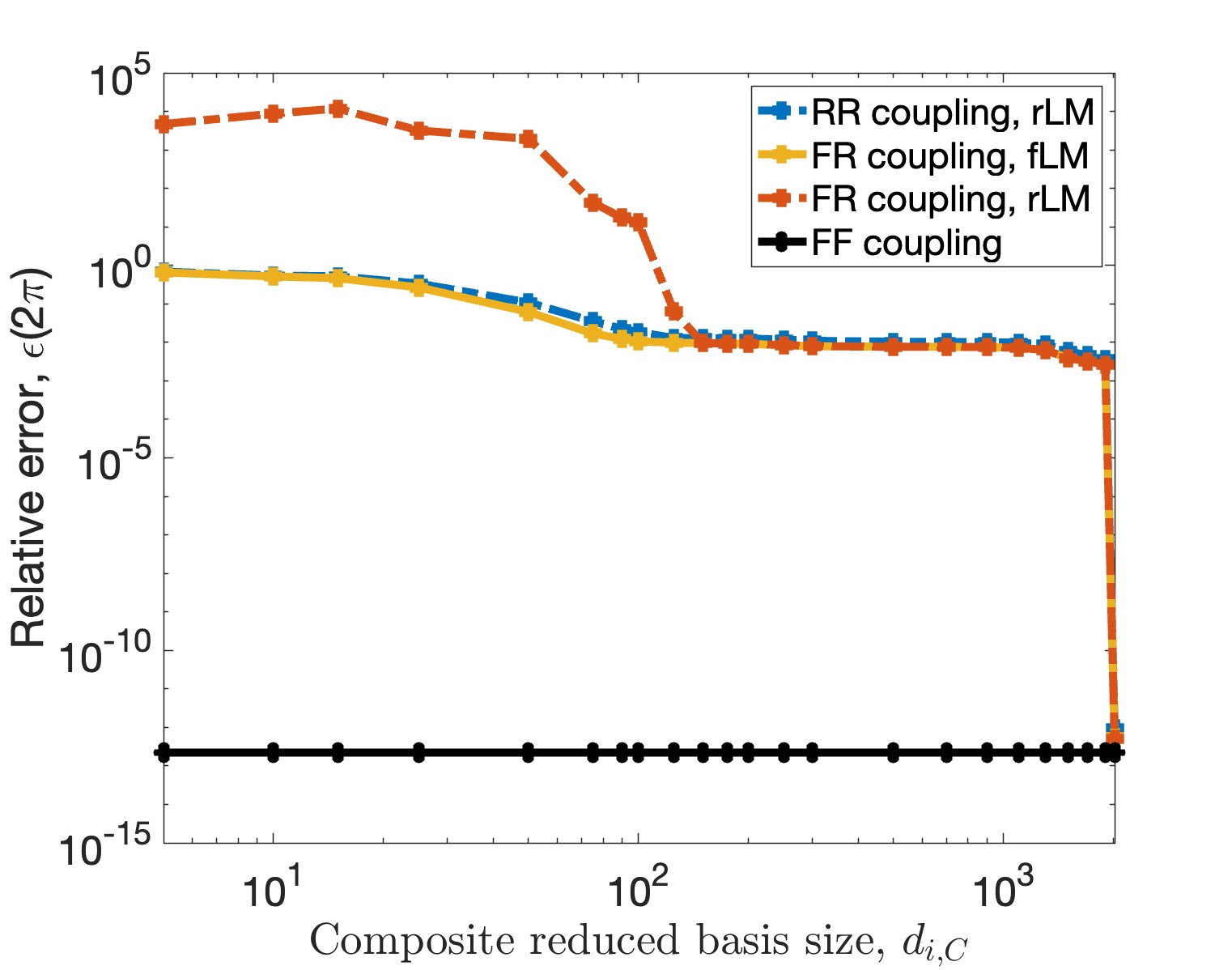}
  \end{center}
  \vspace{-2ex}
\caption{Relative error \eqref{AdC:eq:relError} at \REV{the final time $T_f=2\pi$} of the partitioned solution for each coupled formulation as a function of the composite reduced basis size $d_{i,C}=d_{i,0}+d_{i,\gamma}$  in the predictive regime.} \label{AdC:fig:L2errPred}
\end{figure}
\begin{figure}[t!]
\centering
\subfigure[$d_{i,0} = 15, d_{i,\gamma} = 10$]{\includegraphics[scale=.10]{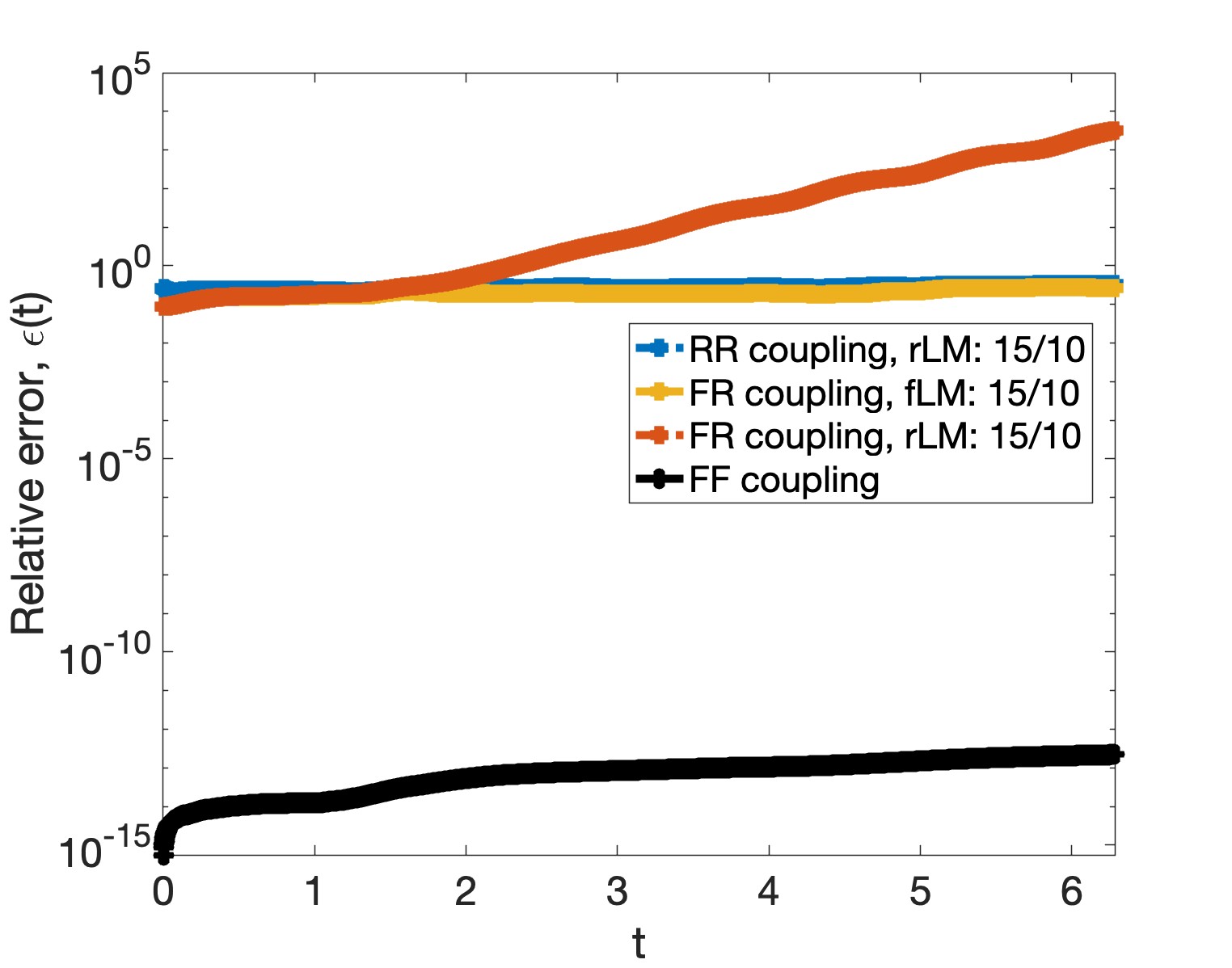}\label{AdC:fig:errVTime1510Pred}} 
\subfigure[$d_{i,0} = 60, d_{i,\gamma} = 40$]{\includegraphics[scale=.10]{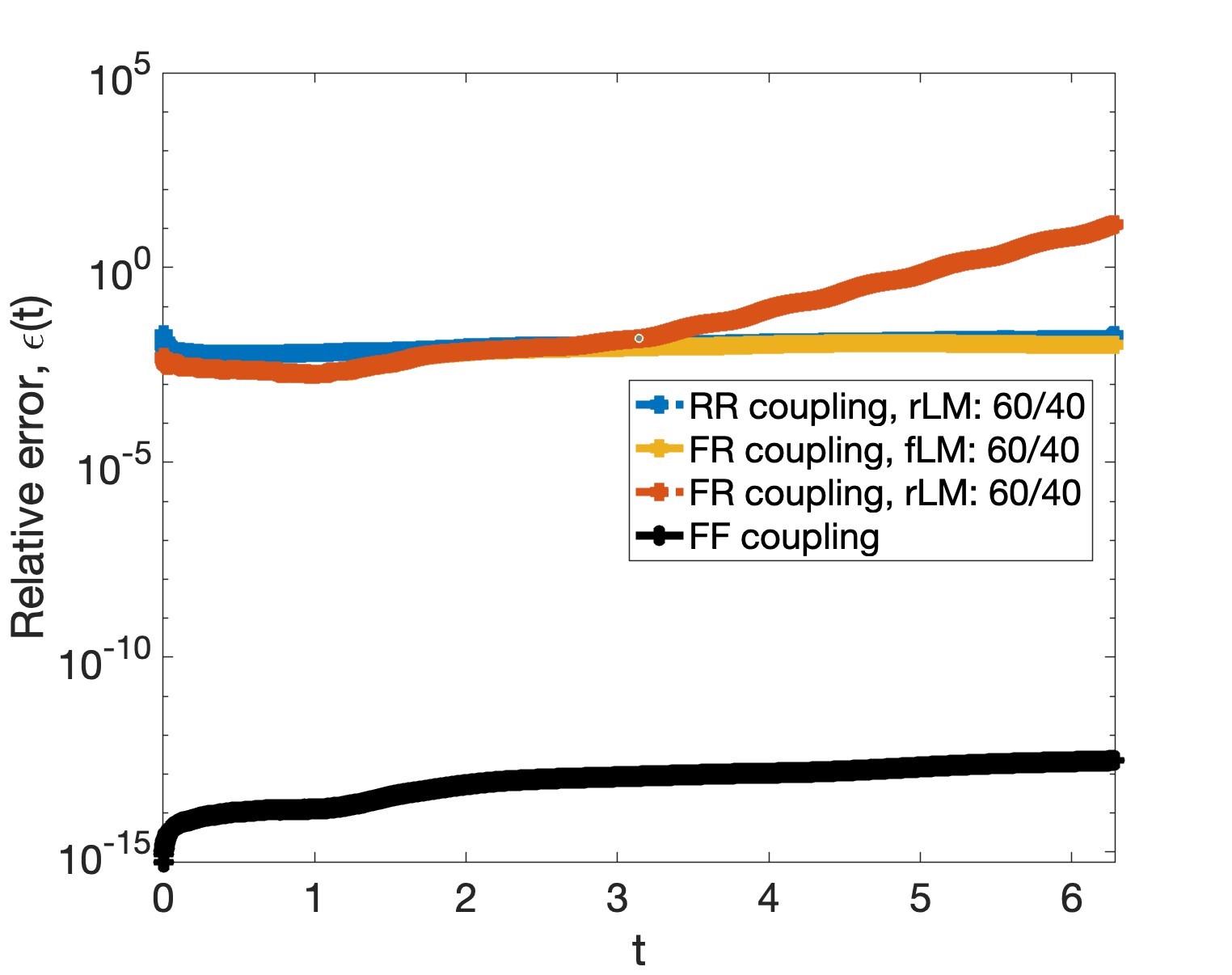}\label{AdC:fig:errVTime6040Pred}} 
\subfigure[$d_{i,0} = 90, d_{i,\gamma} = 60$]{\includegraphics[scale=.10]{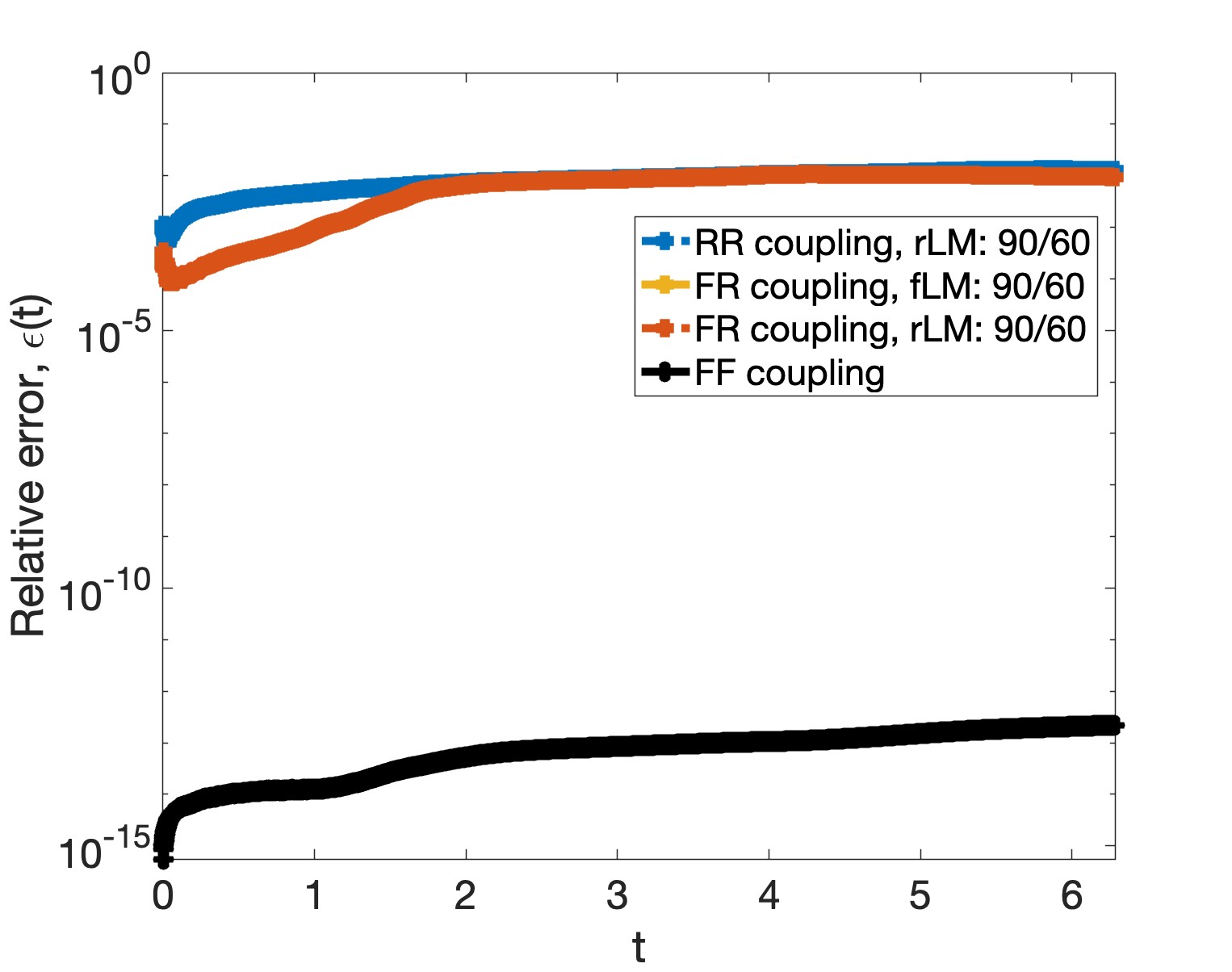}\label{AdC:fig:errVTime9060Pred}} 
\subfigure[$d_{i,0} = 237, d_{i,\gamma} = 63$]{\includegraphics[scale=.10]{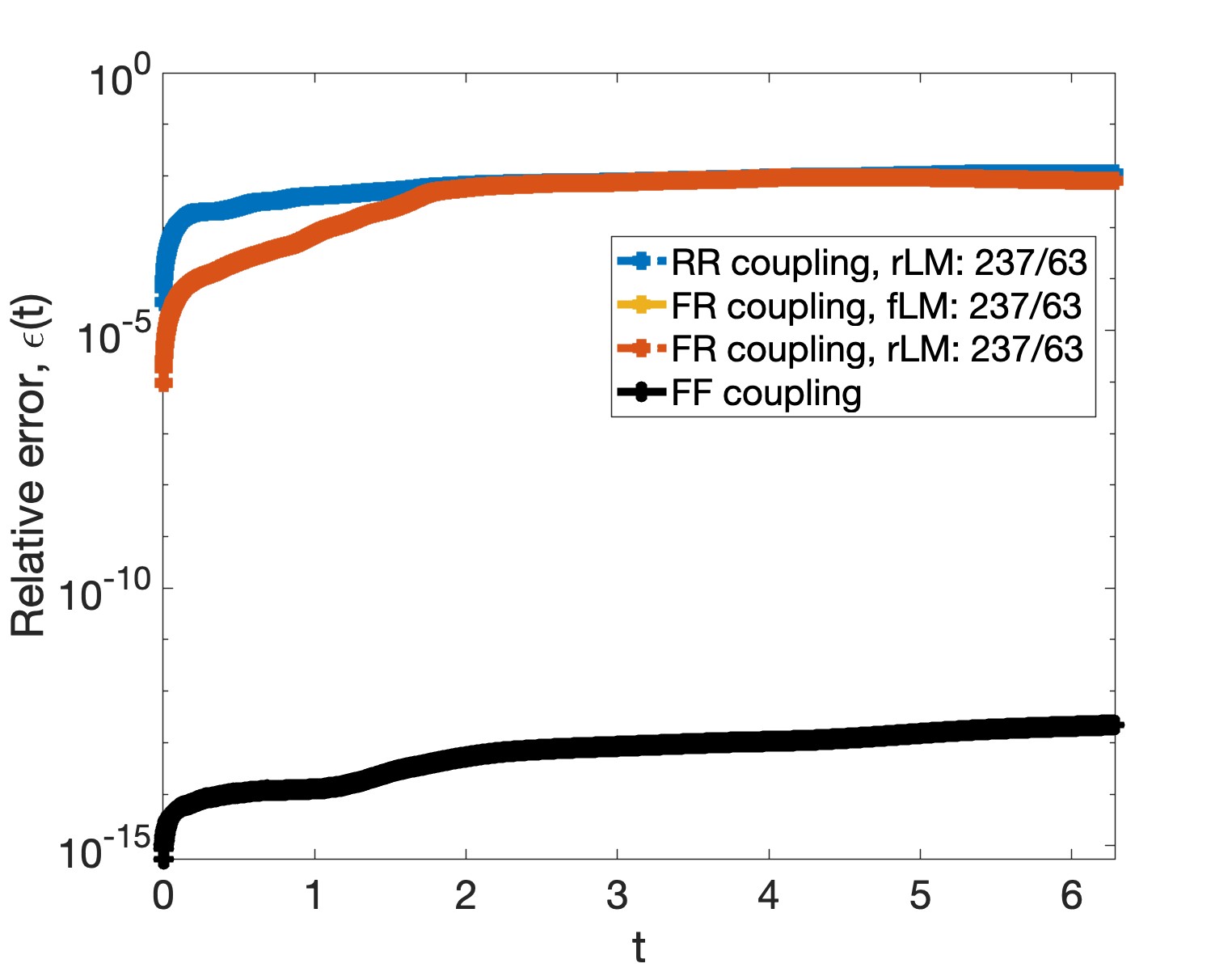}\label{AdC:fig:errVTime23763Pred}} 
\subfigure[$d_{i,0} = 1953, d_{i,\gamma} = 63$]{\includegraphics[scale=.10]{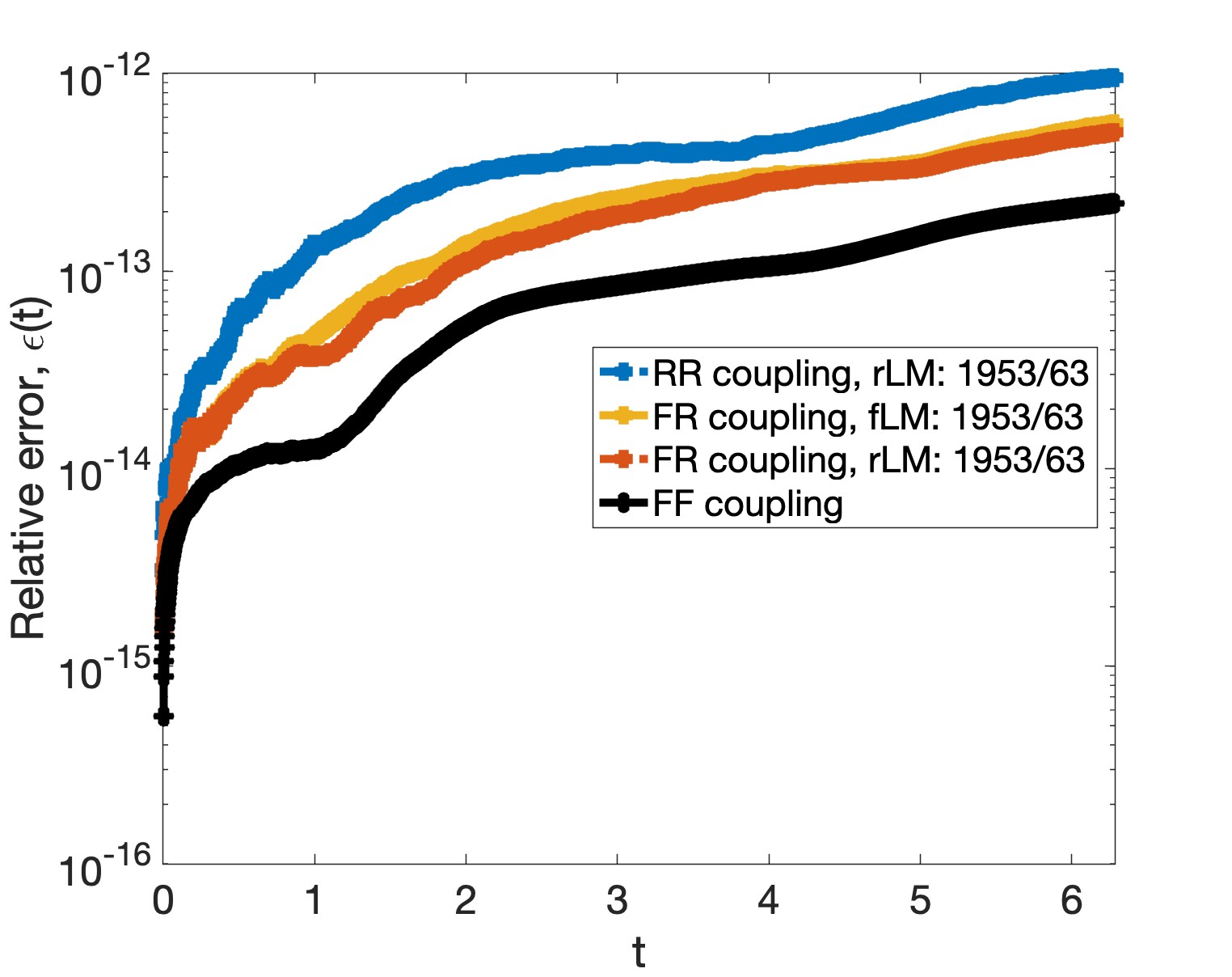}\label{AdC:fig:errVTime195363Pred}} 
\subfigure[Single domain ROM]{\includegraphics[scale=.088]{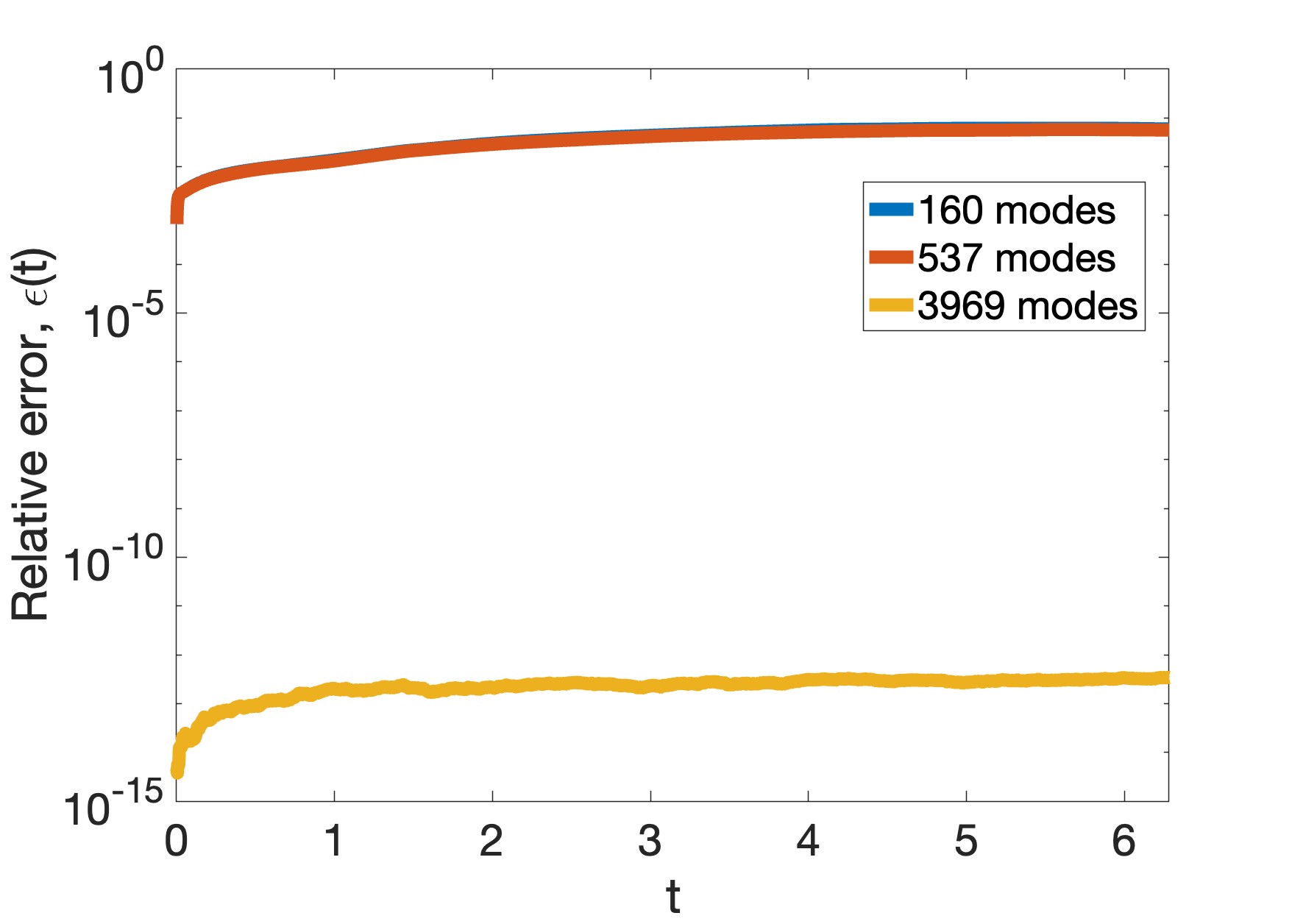}\label{AdC:fig:errVsTimeSingleROM_PredSingleDomain}} 
\caption{\REV{Relative error \eqref{AdC:eq:relError} of the partitioned solution for select basis sizes as a function of time in the predictive regime. The error plots for 160 and 537 modes in subfigure (f) are indistinguishable.}}\label{AdC:fig:L2errVsTime_BasisSizesPred}
\end{figure}
\begin{figure}[h!]
\centering
\subfigure[All coupled problems]{\includegraphics[scale=.15]{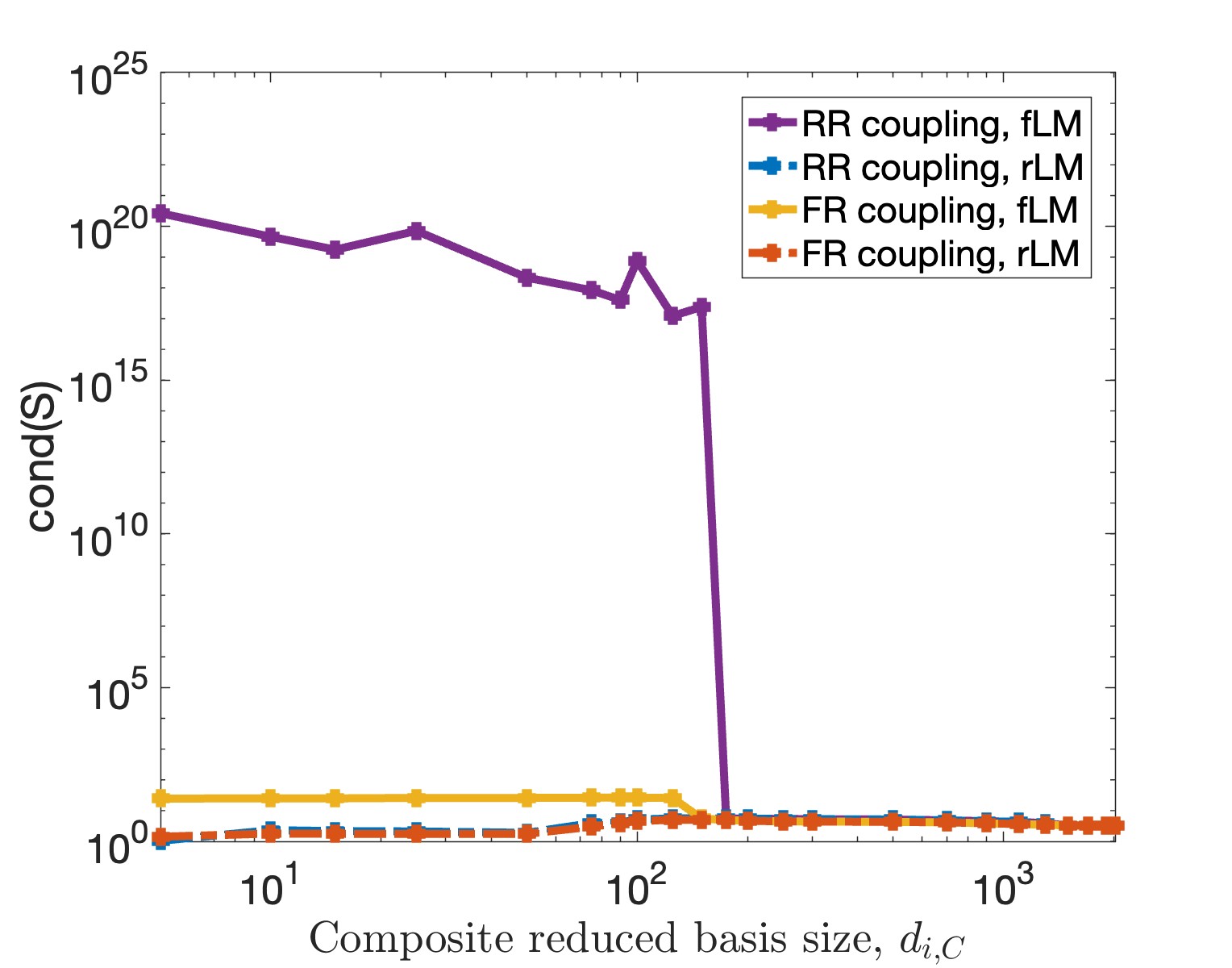}\label{AdC:fig:cond_RRfLM_pred}} 
\subfigure[Provably well-posed Schur complement]{\includegraphics[scale=.15]{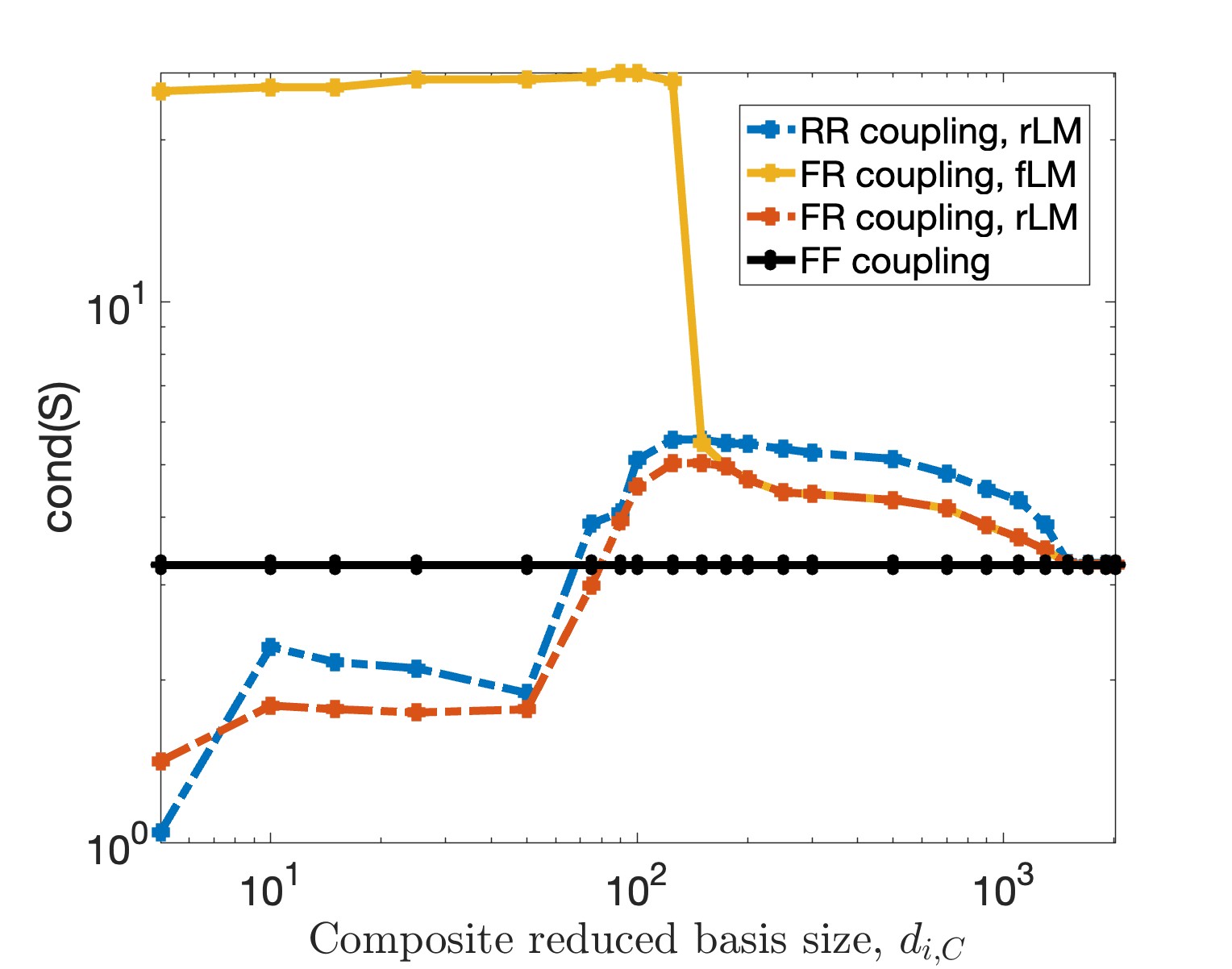}\label{AdC:fig:cond_noRRfLM_pred}}   
\caption{Condition number of Schur complement matrix for each coupled formulation as a function of the composite reduced basis size $d_{i,C}=d_{i,0}+d_{i,\gamma}$ size  in the predictive regime. Subfigure (a) reports results for all methods evaluated, whereas subfigure (b) focuses only 
on methods with provably well-posed Schur complements.}\label{AdC:fig:condSPred}
\end{figure}

We conclude with results showing how well the partitioned solutions satisfy the interface condition in the predictive tests. As in the reproductive test, we focus only on the well-posed coupled formulations, and compare the interface states of the partitioned and single domain solutions. These results are summarized  in Figure \ref{AdC:fig:InterPlots_Pred}.
We again see that partitioned solutions of the couplings with provably well-posed Schur complements are able to satisfy the interface condition. The coupled ROM-ROM and FOM-ROM problems display agreement with the single domain solution on the interface. \REV{Not surprisingly, we see that, as in the reproductive test, the FR-rLM formulation requires a larger basis set to achieve the same accuracy as the other formulations. In particular, for ``small'' RB sizes we see the same oscillatory behavior at the final time, which is caused by the approximate enforcement \eqref{eq:inequality} of the coupling condition and the subsequent accumulation of errors during the time integration.}

\begin{figure}[!ht]
\centering
\subfigure[RR-rLM, 15/10 modes]{\includegraphics[scale=.08]{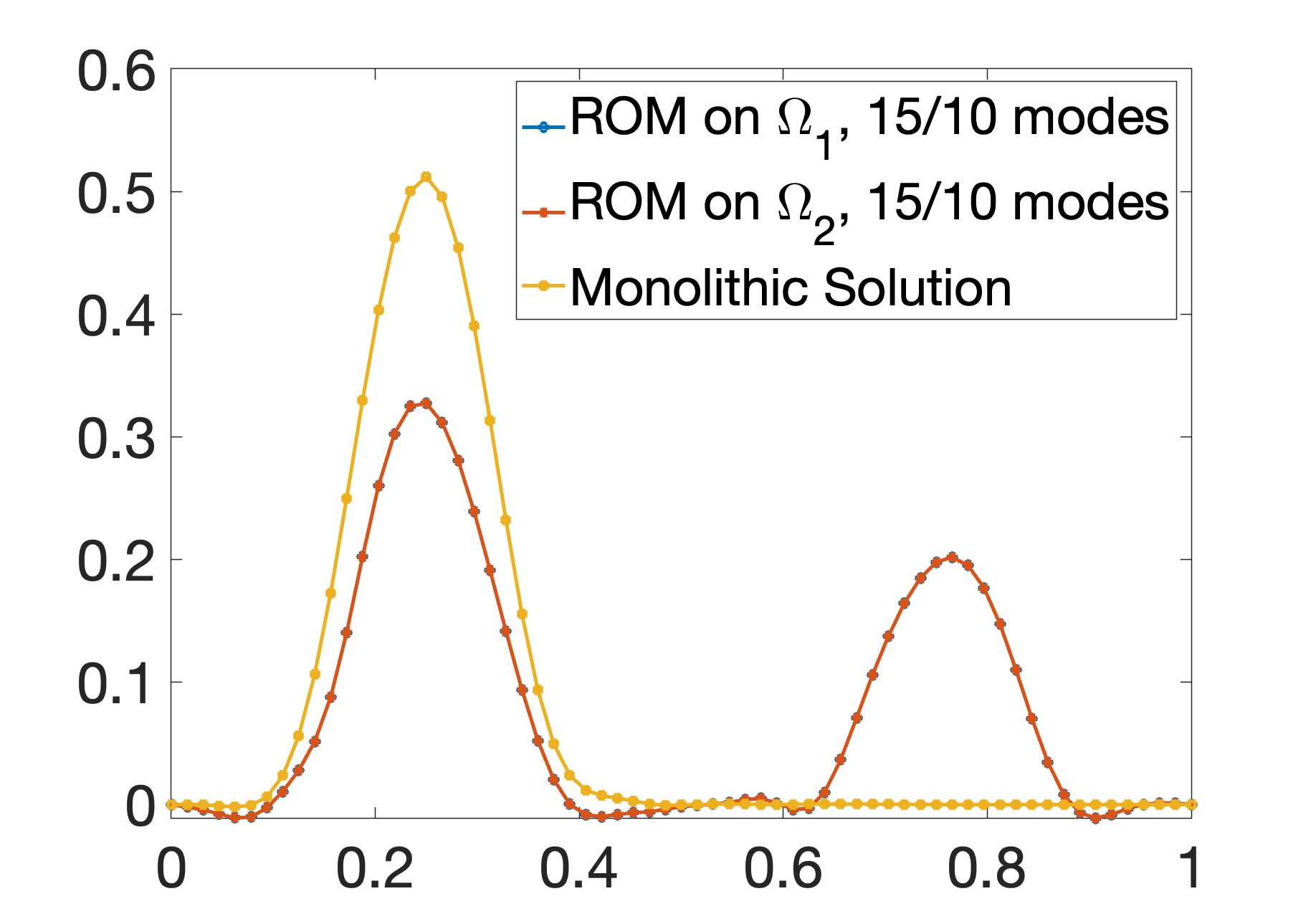}\label{AdC:fig:RR_Inter_15_10_pred}} 
\hfill
\subfigure[\REV{FR-fLM, 15/10 modes}]{\includegraphics[scale=.08]{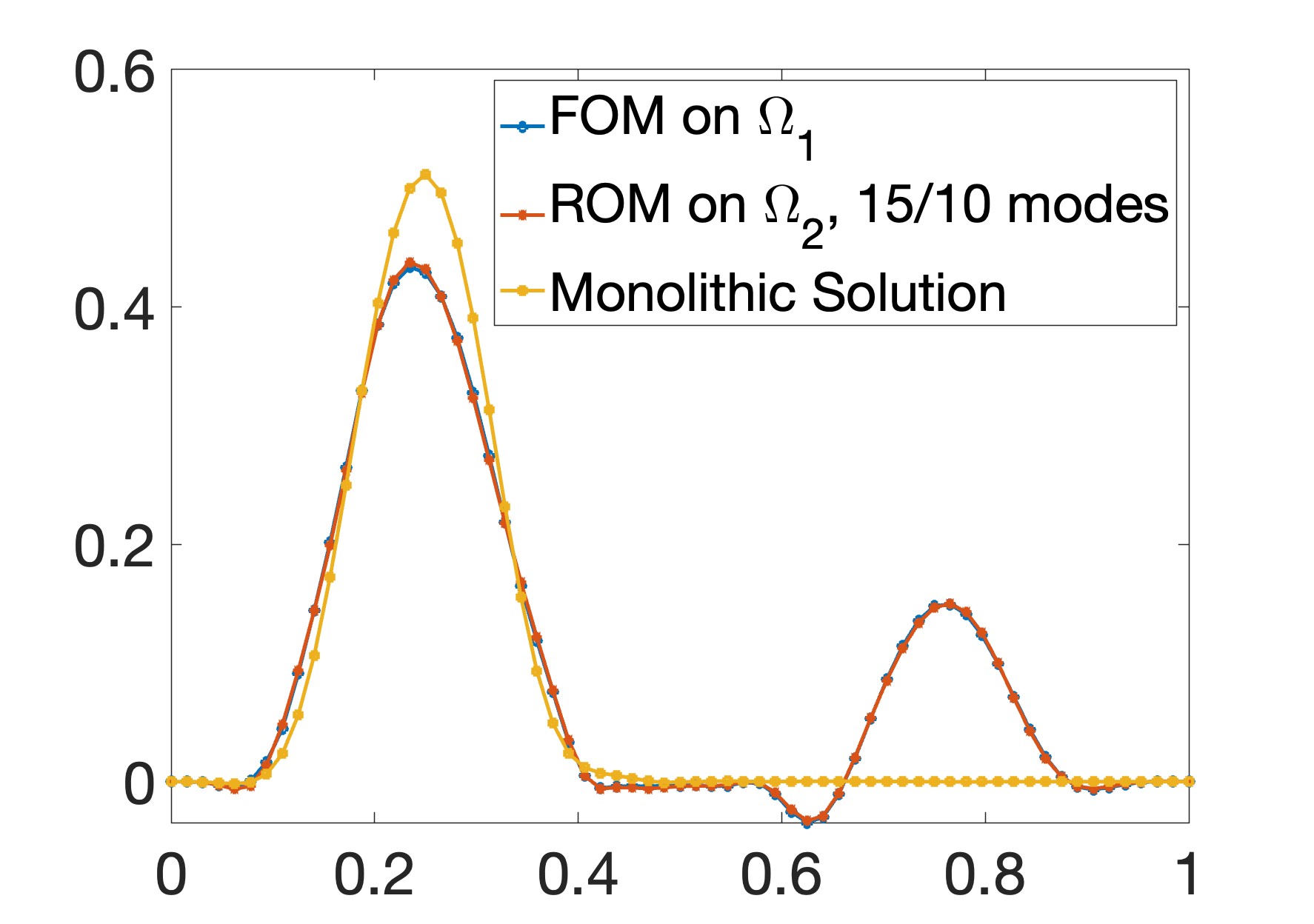}\label{AdC:fig:FRfLM_Inter_15_10_pred}} 
\hfill
\subfigure[\REV{FR-rLM, 15/10 modes}]{\includegraphics[scale=.085]{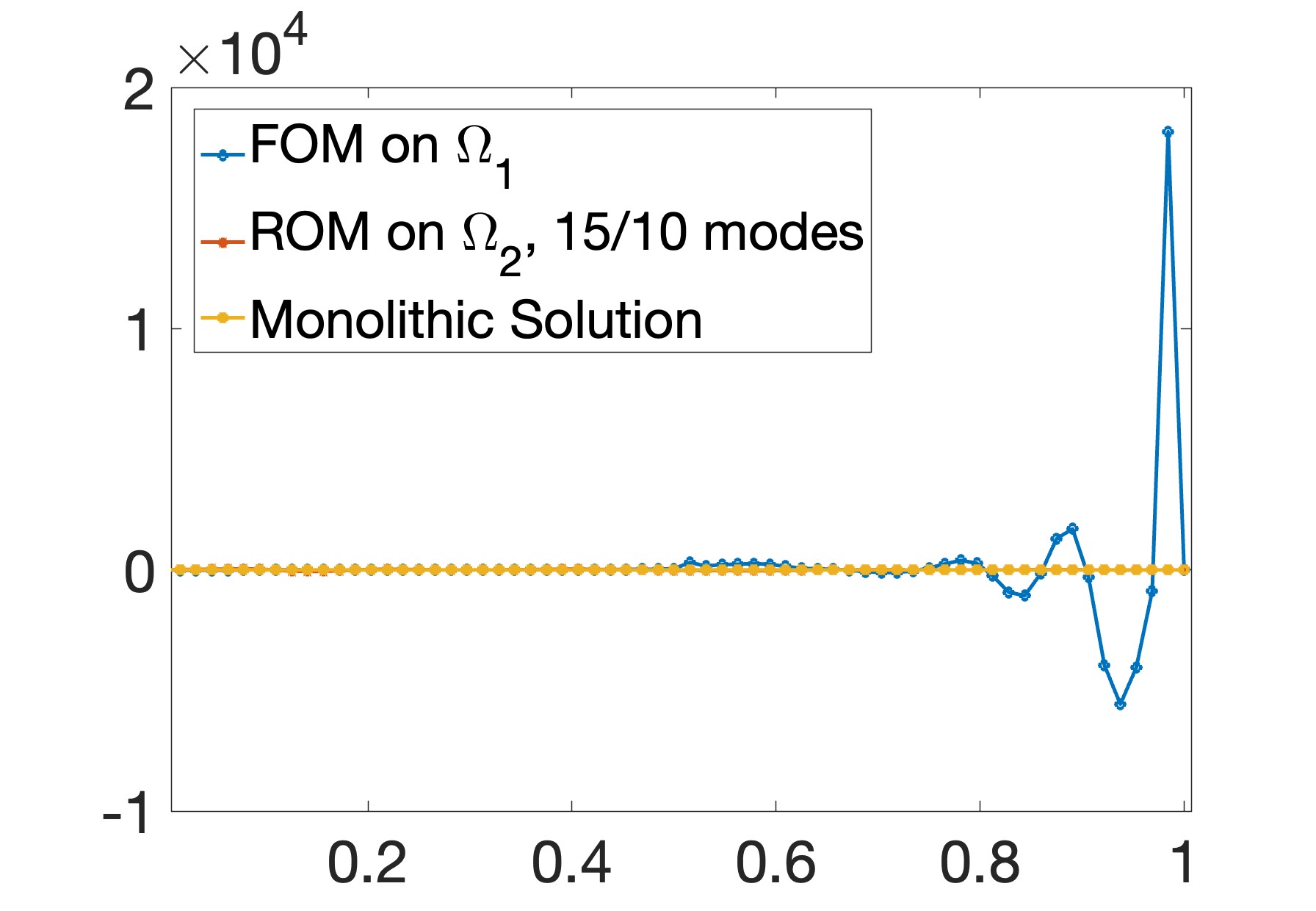}\label{AdC:fig:FRrLM_Inter_15_10_pred}} 
\subfigure[RR-rLM 60/40 modes]{\includegraphics[scale=.08]{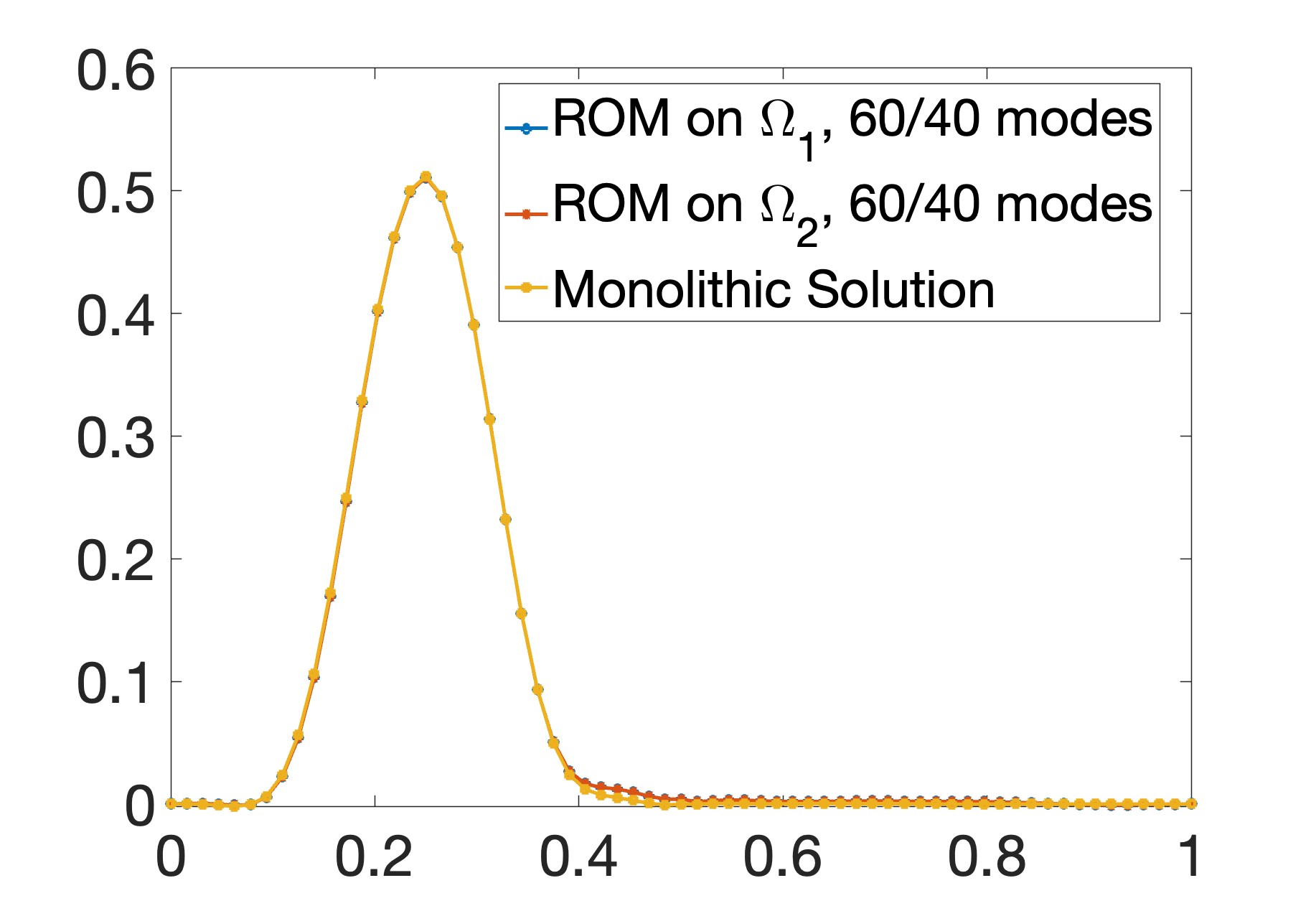}\label{AdC:fig:RR_Inter_60_40_pred}} 
\hfill
\subfigure[FR-fLM 60/40 modes]{\includegraphics[scale=.08]{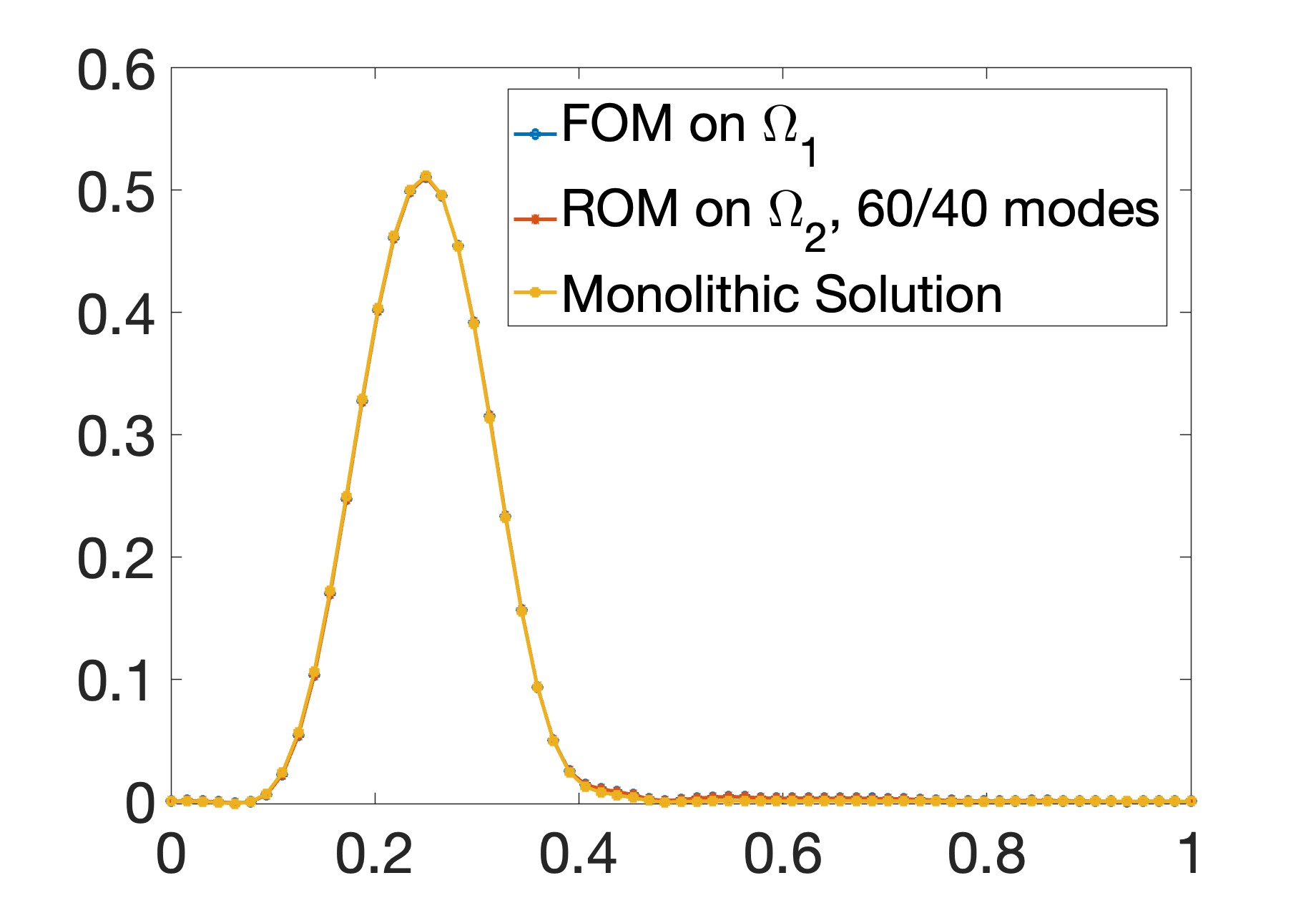}\label{AdC:fig:FRfLM_Inter_60_40_pred}} 
\hfill
\subfigure[\REV{FR-rLM, 60/40 modes}]{\includegraphics[scale=.08]{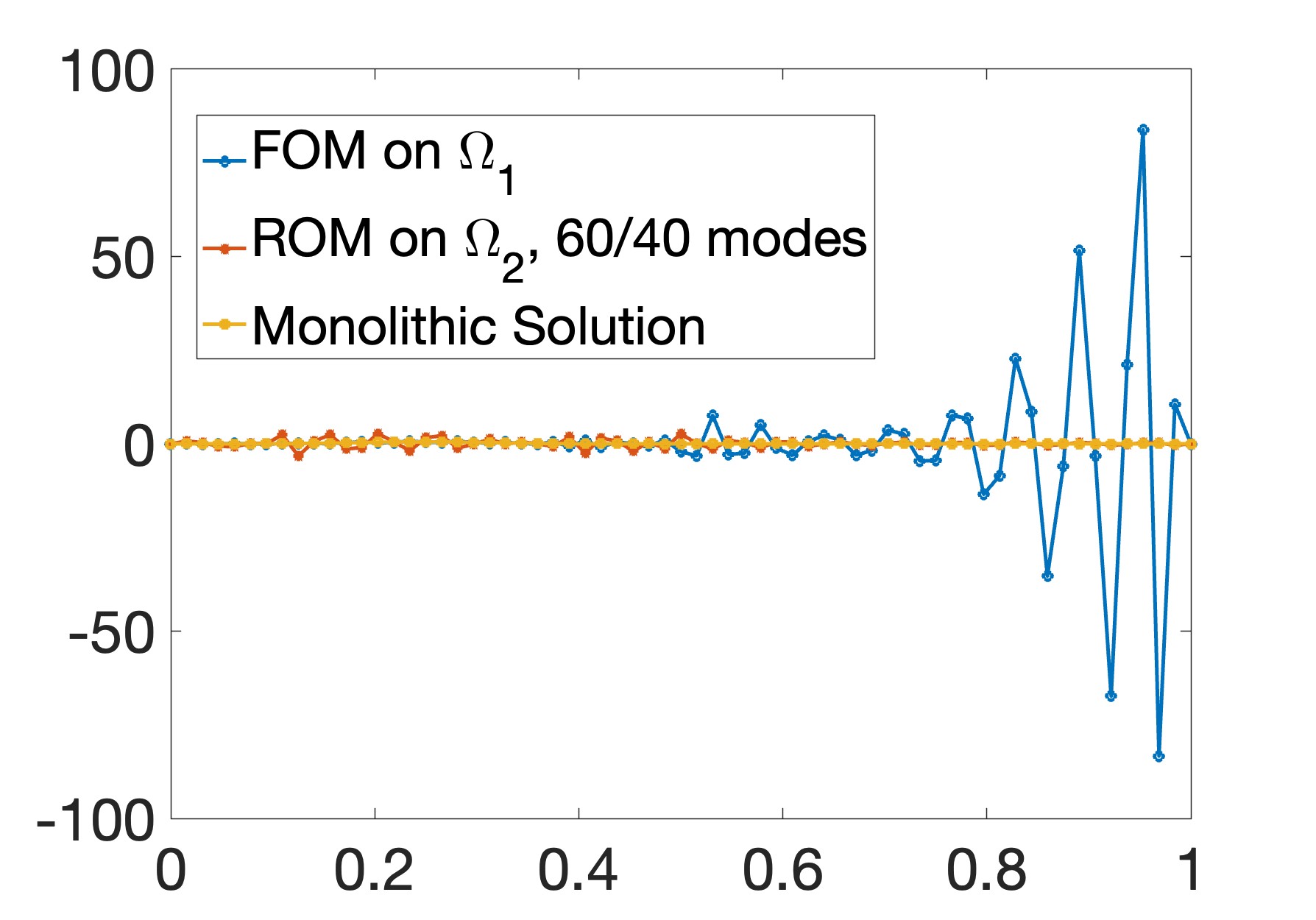}\label{AdC:fig:FRrLM_Inter_60_40_pred}} \\
$ $\hfill
\subfigure[\REV{RR-rLM, 90/60 modes}]{\includegraphics[scale=.08]{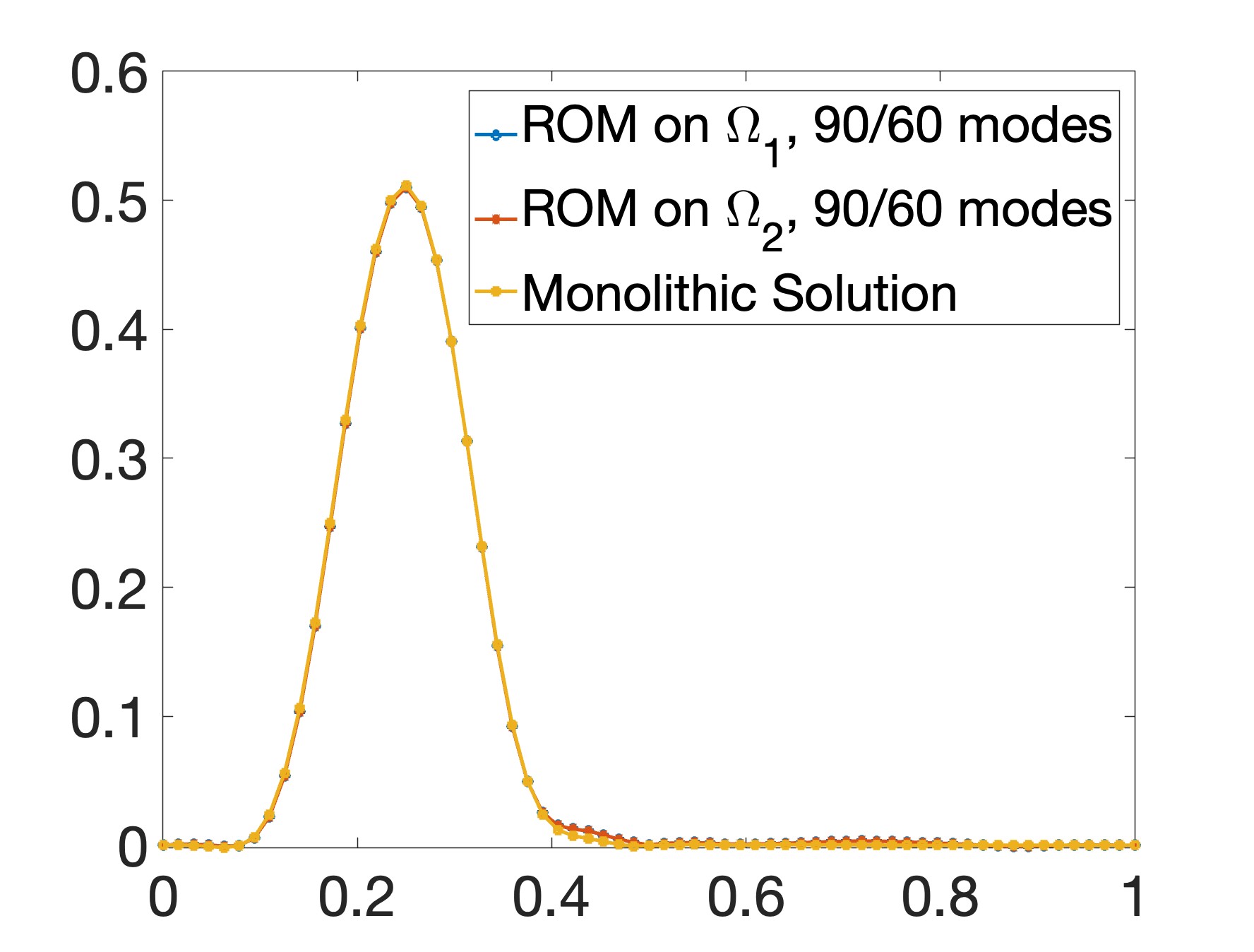}\label{AdC:fig:RRrLM_Inter_90_60_pred}} 
\hfill
\subfigure[\REV{FR-fLM, 90/60 modes}]{\includegraphics[scale=.08]{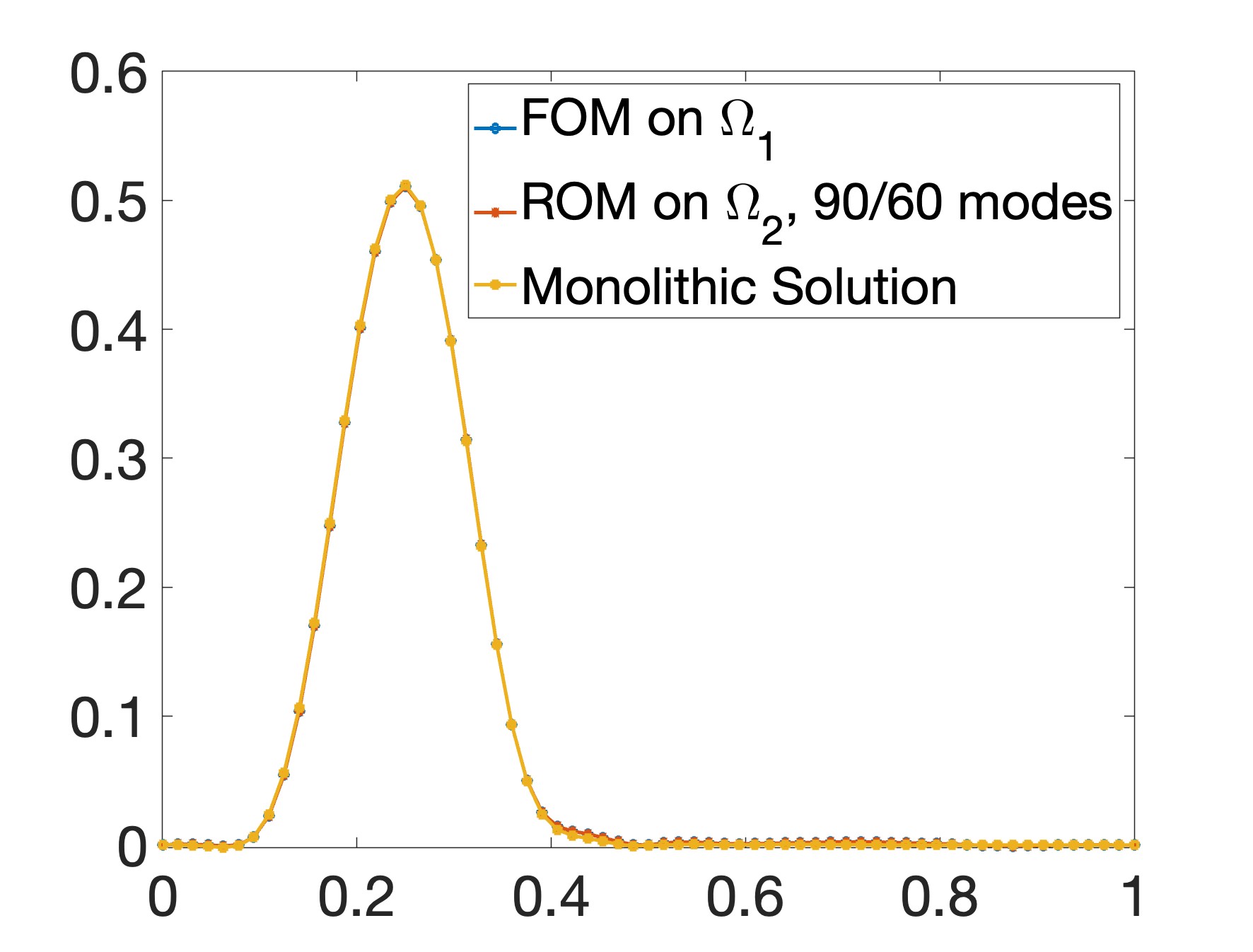}\label{AdC:fig:FRfLM_Inter_90_60_pred}} 
\hfill
\subfigure[FR-rLM 90/60 modes]{\includegraphics[scale=.08]{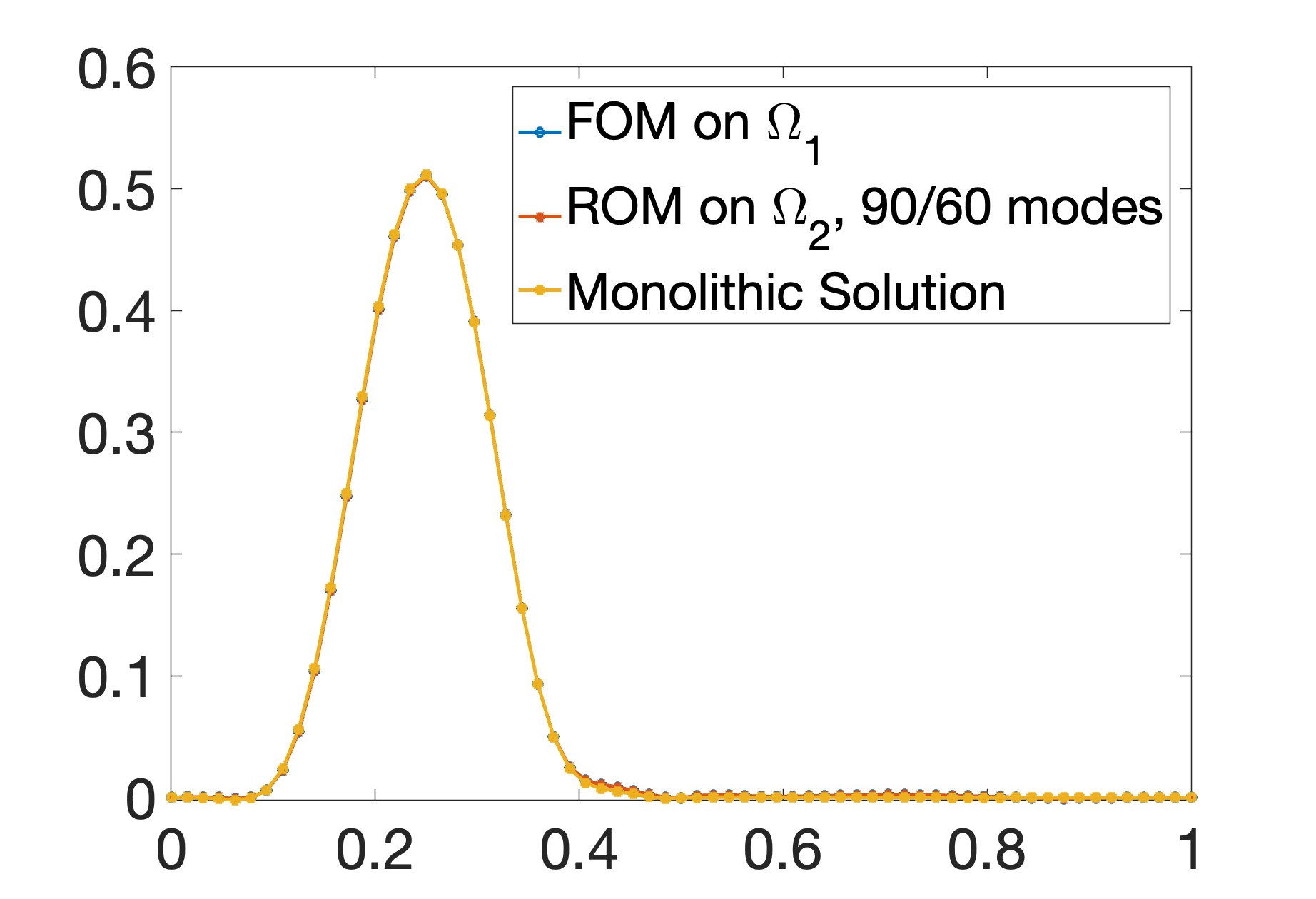}\label{AdC:fig:FRrLM_Inter_90_60_pred}} 
\caption{\REV{Comparison of the interface states at $T_f$ for the partitioned schemes  with provably well-posed Schur complements vs. the single domain (monolithic) solution of the model problem. Predictive test in the DD setting. The oscillations in the FR-rLM formulations with ``small'' RB sizes are due to accumulation of interface errors during the time integration caused by the approximate enforcement \eqref{eq:inequality} of the coupling condition. The legend ``$m/n$ modes'' corresponds to $m$ interior and $n$ interface modes.}}\label{AdC:fig:InterPlots_Pred}
\end{figure}
%

\subsection{\REV{Transmission problem case}}\label{sec:multi}
\REV{For the TP variant of our model problem, we parameterize the model problem using a discontinuous diffusion coefficient, i.e., $\kappa_1\neq \kappa_2$. Since in our experiments we observed essentially the same behavior as in the DD case, we limit ourselves to showing results for a predictive TP test. For this test, we consider a diffusion coefficient $\kappa_1 = 10^{-5}$ in $\Omega_1$, and $\kappa_2 = 10^{-4}$ in $\Omega_2$. To obtain the subdomain snapshots, we proceed as in the DD case and restrict a single domain finite element solution to $\Omega_1$ and $\Omega_2$, respectively. The single domain solution is computed using the time step} \REV{$\Delta t_s = 1.684\times 10^{-3}$}.

\REV{The snapshot energy plots in Figure \REVpk{\ref{AdC:fig:snapEnergyMulti}} confirm that the rapid decay of the singular values continues to hold in the TP case. In particular, just} 
\REV{$d_{1,0}=23$, $d_{2,0}=19$ and $d_{i,\gamma}=5$} 
\REV{interior and interface nodes, respectively are sufficient to capture $99\%$ of the energy in $X_{i,0}$ and $X_{i,\gamma}$.}

\begin{figure}[h]
  \begin{center}
   \includegraphics[width=0.5\textwidth]{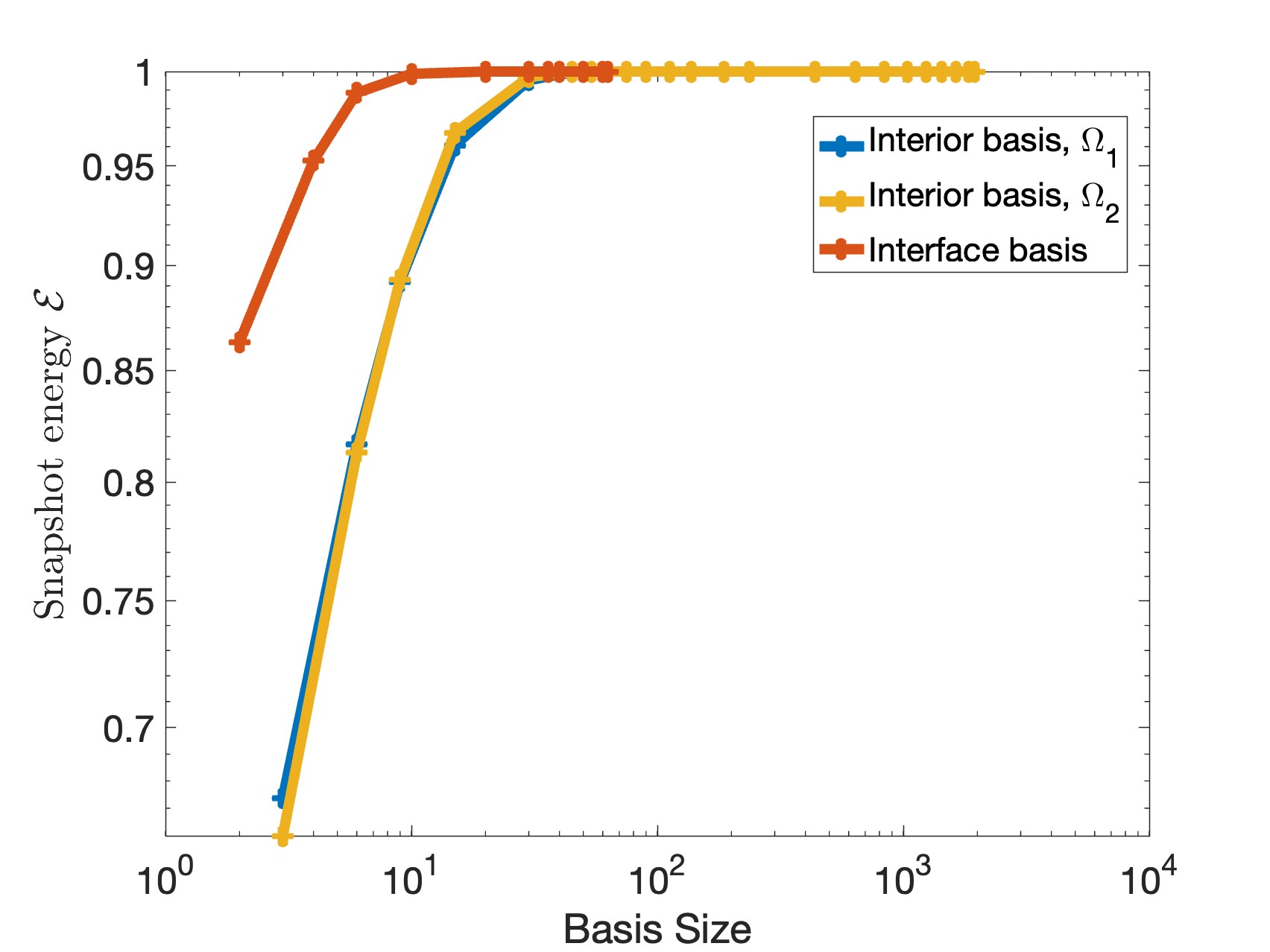}
  \end{center}
		\caption{Snapshot energy \eqref{AdC:eq:snapEnergy} as a function of the POD basis size for the interior $\REV{\Phi}_{i,0}$ and interface $\REV{\Phi}_{i,\gamma}$ bases in the predictive regime for the multiphysics example.} \label{AdC:fig:snapEnergyMulti}
\end{figure}

\REV{Next, in Figures \ref{AdC:fig:L2errMulti} and \ref{AdC:fig:L2errVsTime_BasisSizesMulti}, we show the relative errors for each formulation. Similar to the DD case, the relative error at the final time of the FR-rLM formulation for small RB sizes is much larger than that for the other formulations.  The plots in Figure \ref{AdC:fig:L2errVsTime_BasisSizesMulti} confirm that, again, this is due to the accumulation of error during the explicit time stepping. Interestingly enough, these plots also show that the growth of the relative error in the TP case occurs at roughly the same time instances as in the DD case.  Moreover, 
the reader can observe convergence with basis refinement for all proposed schemes \REVpk{(Figure \ref{AdC:fig:L2errVsTime_BasisSizesMulti})}.}

\REV{We continue with plots of the condition number of the Schur complement for all coupled formulations in Figure \ref{AdC:fig:condSMulti}. These plots mirror the behavior of this quantity from the DD case and once again underscore the importance of using Lagrange multiplier spaces that satisfy the trace compatibility condition; see Remark \ref{lem:trace}. Most notably, Figure \ref{AdC:fig:condSMulti}(b) shows no substantive difference in the behavior of the condition number in the TP case as long as the formulation satisfies the trace compatibility condition.}

\REV{Finally, in Figure \ref{AdC:fig:InterPlots_MultiPhysics}, we take a look at the interface states of all formulations with provably well-posed Schur complements at the final time $T_f$. The results in this figure affirm yet again the consistency in the behavior of the proposed schemes for the DD and TP examples.}

\begin{figure}[h!]
  \begin{center}
    \includegraphics[width=0.6\textwidth]{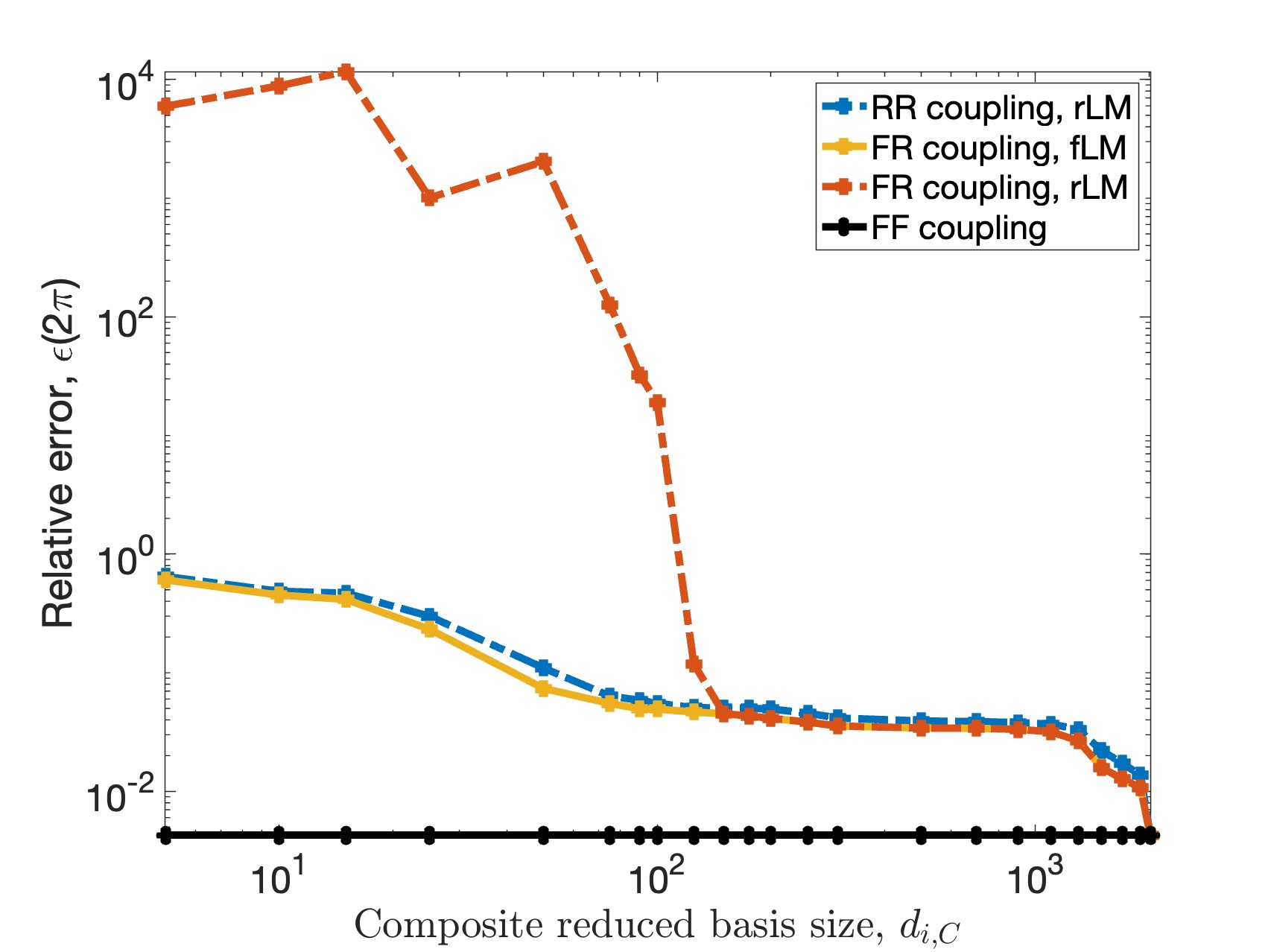}
  \end{center}
\caption{\REV{Relative error \eqref{AdC:eq:relError} at $T_f=2\pi$ of the partitioned solution for each coupled formulation as a function of the composite reduced basis size $d_{i,C}=d_{i,0}+d_{i,\gamma}$  in the predictive regime for the TP example.}} \label{AdC:fig:L2errMulti}
\end{figure}
\begin{figure}[t!]
\centering
\subfigure[$d_{i,0} = 15, d_{i,\gamma} = 10$]{\includegraphics[scale=.10]{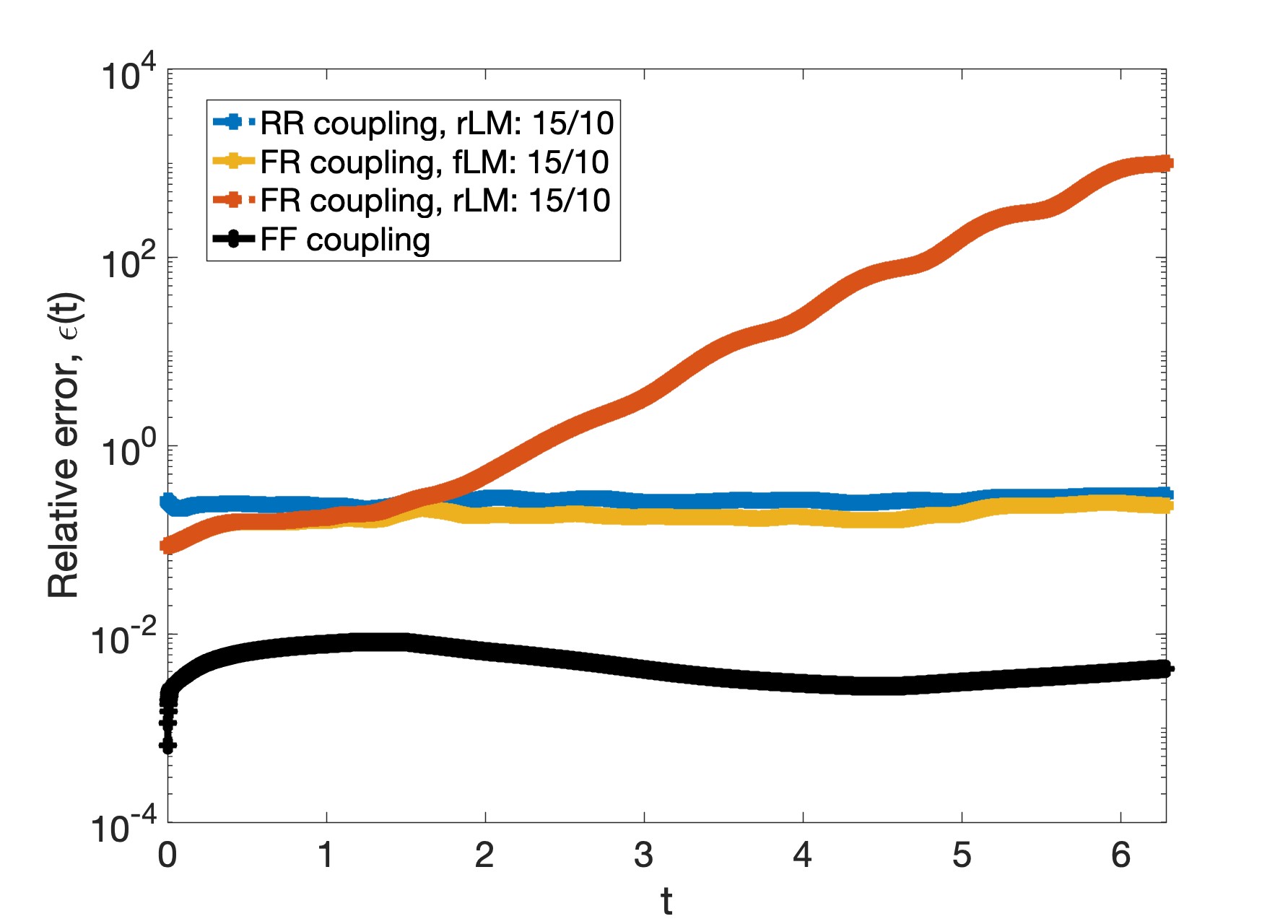}\label{AdC:fig:errVTime1510Multi}} 
\subfigure[$d_{i,0} = 60, d_{i,\gamma} = 40$]{\includegraphics[scale=.10]{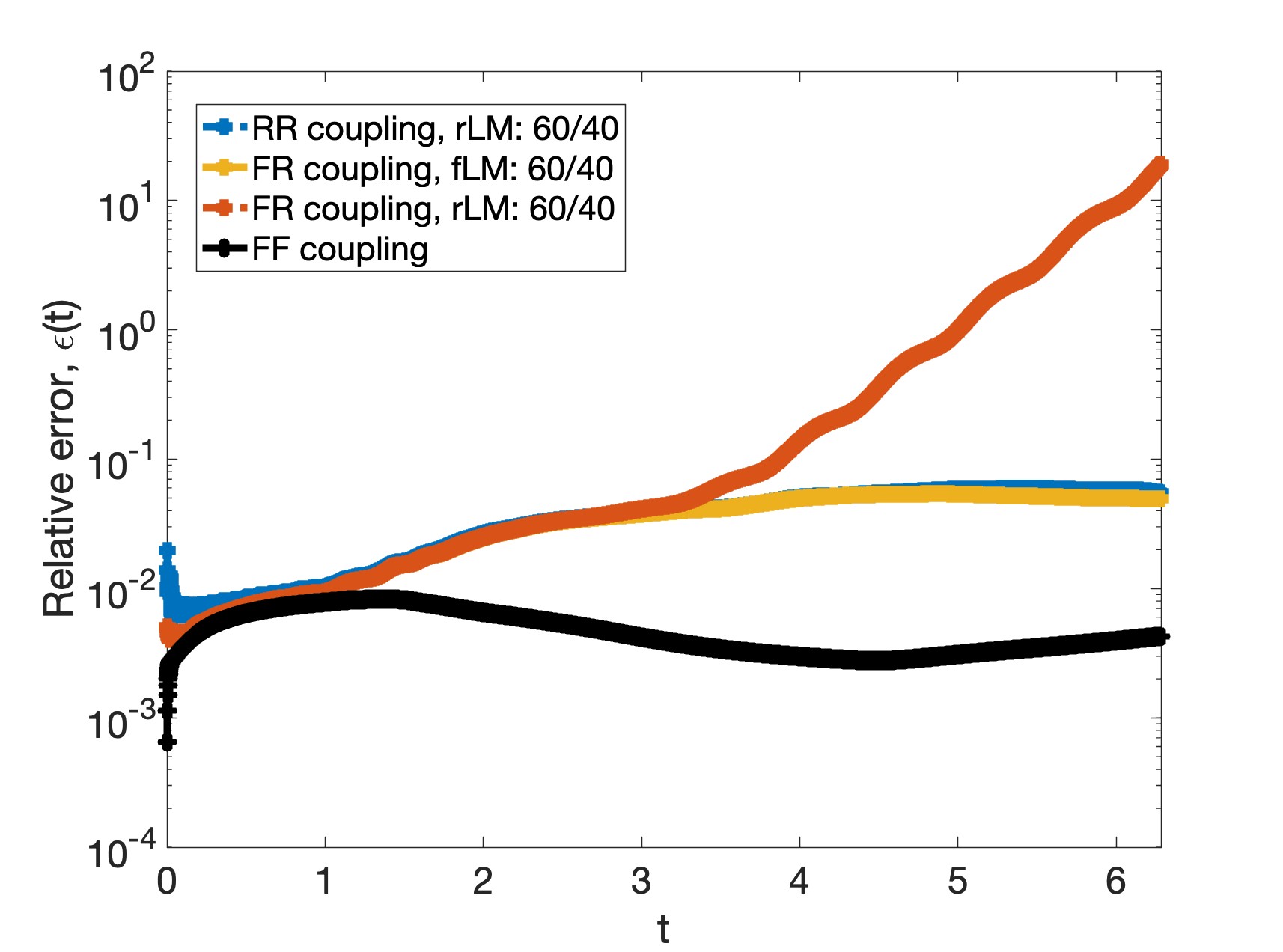}\label{AdC:fig:errVTime6040Multi}} 
\subfigure[$d_{i,0} = 90, d_{i,\gamma} = 60$]{\includegraphics[scale=.10]{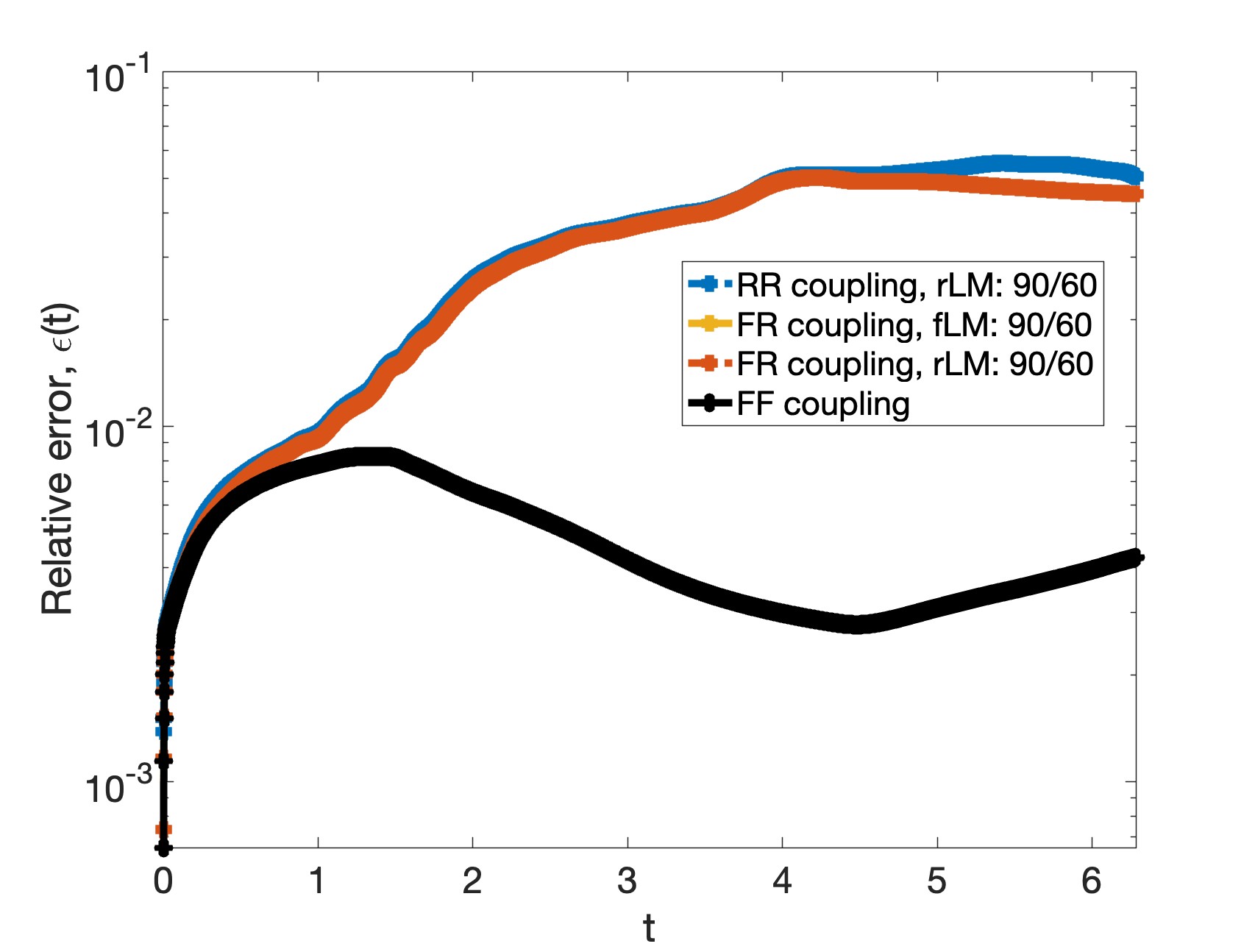}\label{AdC:fig:errVTime9060Multi}} 
\subfigure[$d_{i,0} = 237, d_{i,\gamma} = 63$]{\includegraphics[scale=.10]{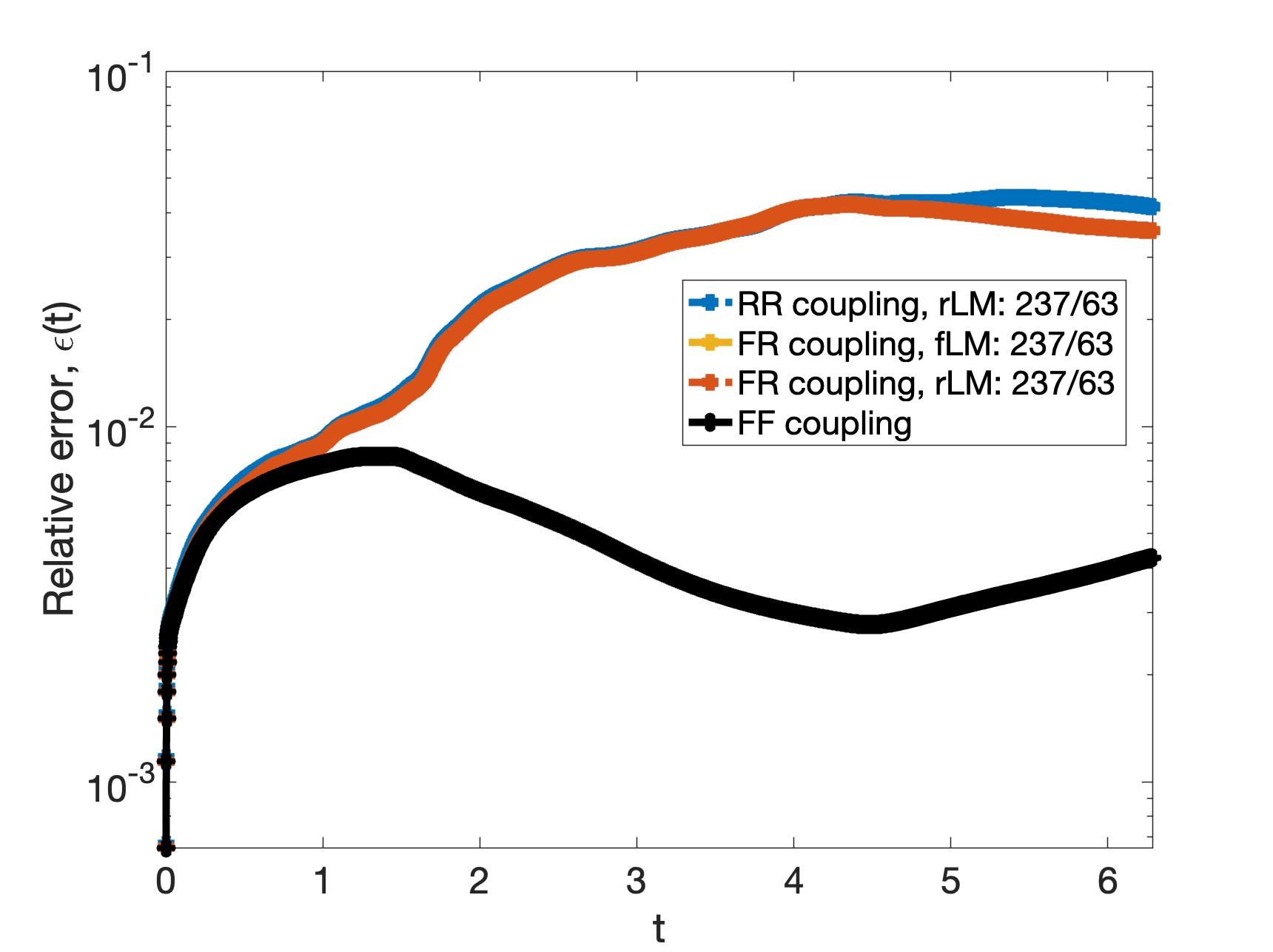}\label{AdC:fig:errVTime23763Multi}} 
\subfigure[$d_{i,0} = 1953, d_{i,\gamma} = 63$]{\includegraphics[scale=.10]{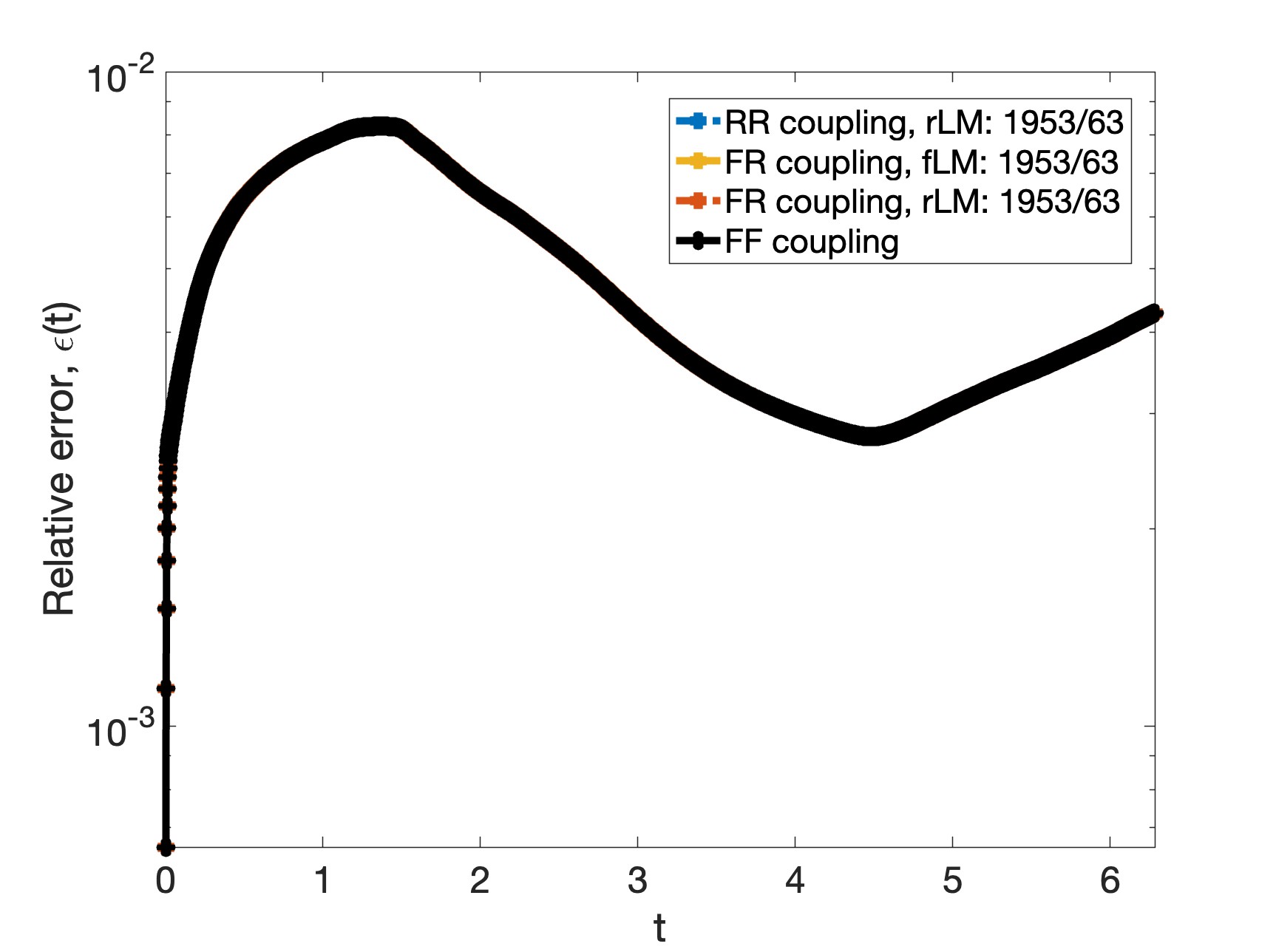}\label{AdC:fig:errVTime195363Multi}} 
\subfigure[Single domain ROM]{\includegraphics[scale=.105]{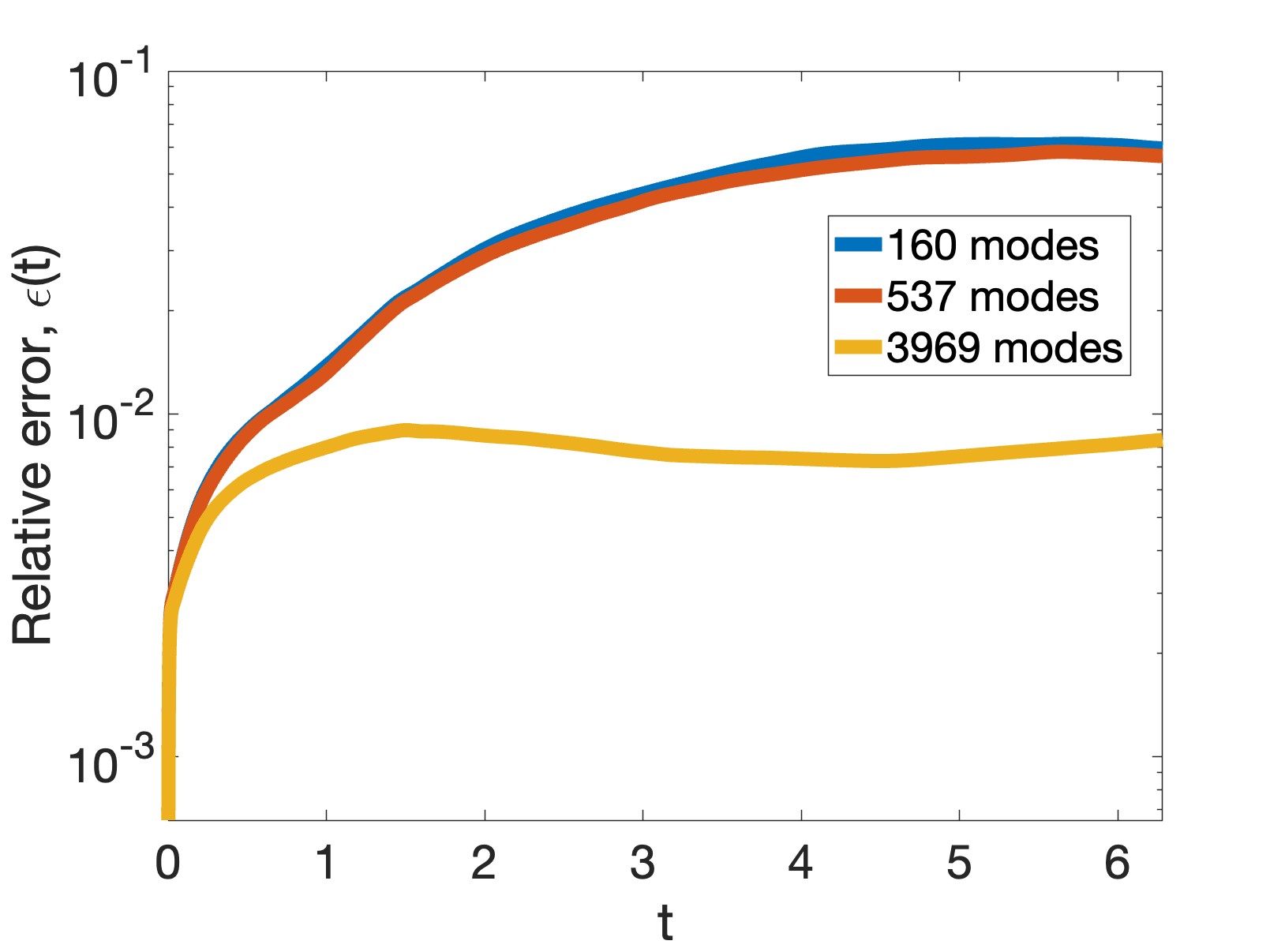}\label{AdC:fig:errVsTime_SingleROM_PredMultiPhysics}} 
\caption{\REV{Relative error \eqref{AdC:eq:relError} of the partitioned solution for select basis sizes as a function of time in the predictive regime for the TP example. The error plots for the FR-rLM and FR-fLM formulations in subfigures (c,d), and the error plots for all couplings in subfigure (e), are indistinguishable.}}\label{AdC:fig:L2errVsTime_BasisSizesMulti}
\end{figure}

\begin{figure}[t!]
\centering
	\subfigure[All coupled problems]{\includegraphics[scale=.13]{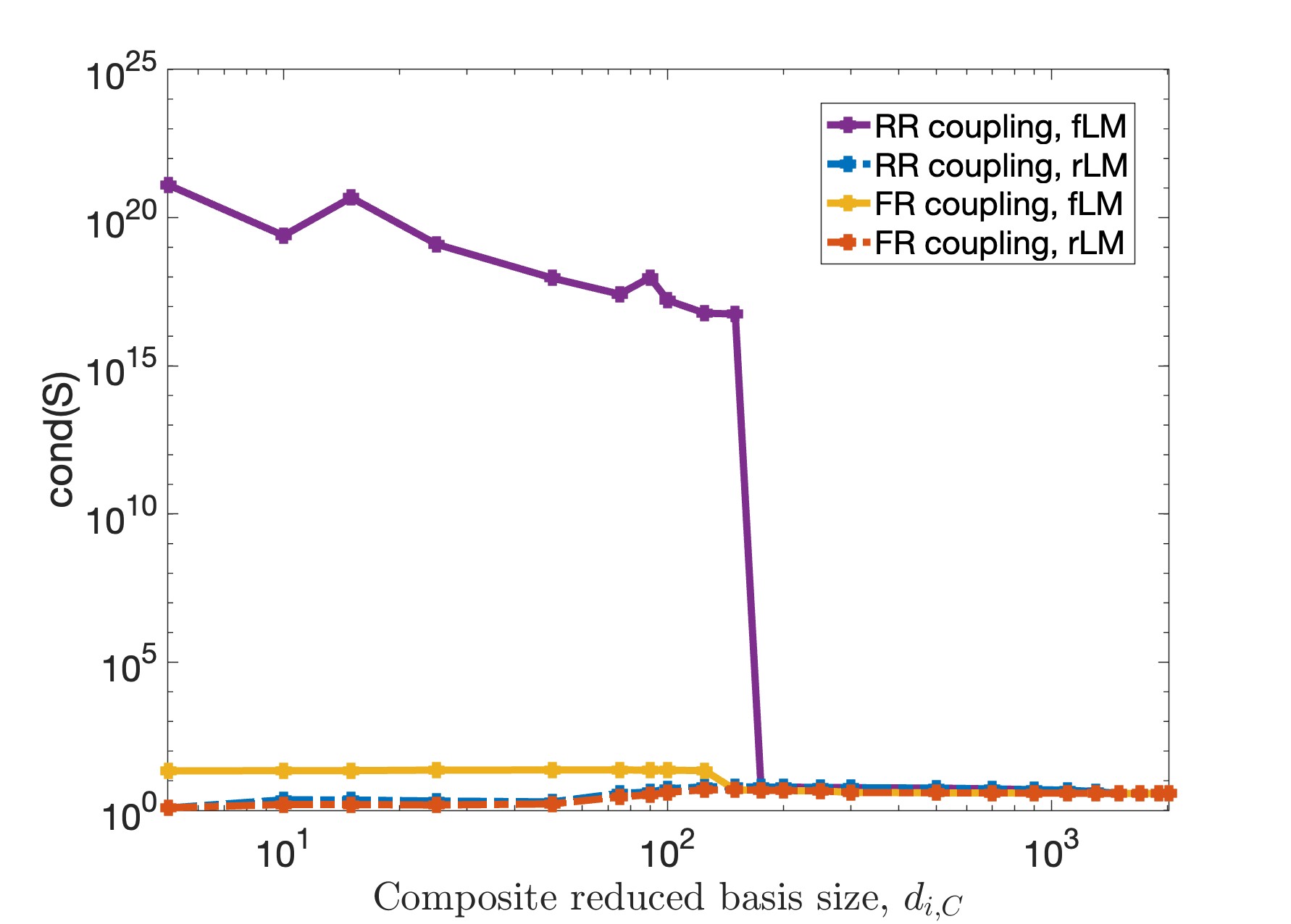}\label{AdC:fig:cond_RRfLM_reprodMulti}} 
\subfigure[Provably well-posed Schur complement]{\includegraphics[scale=.13]{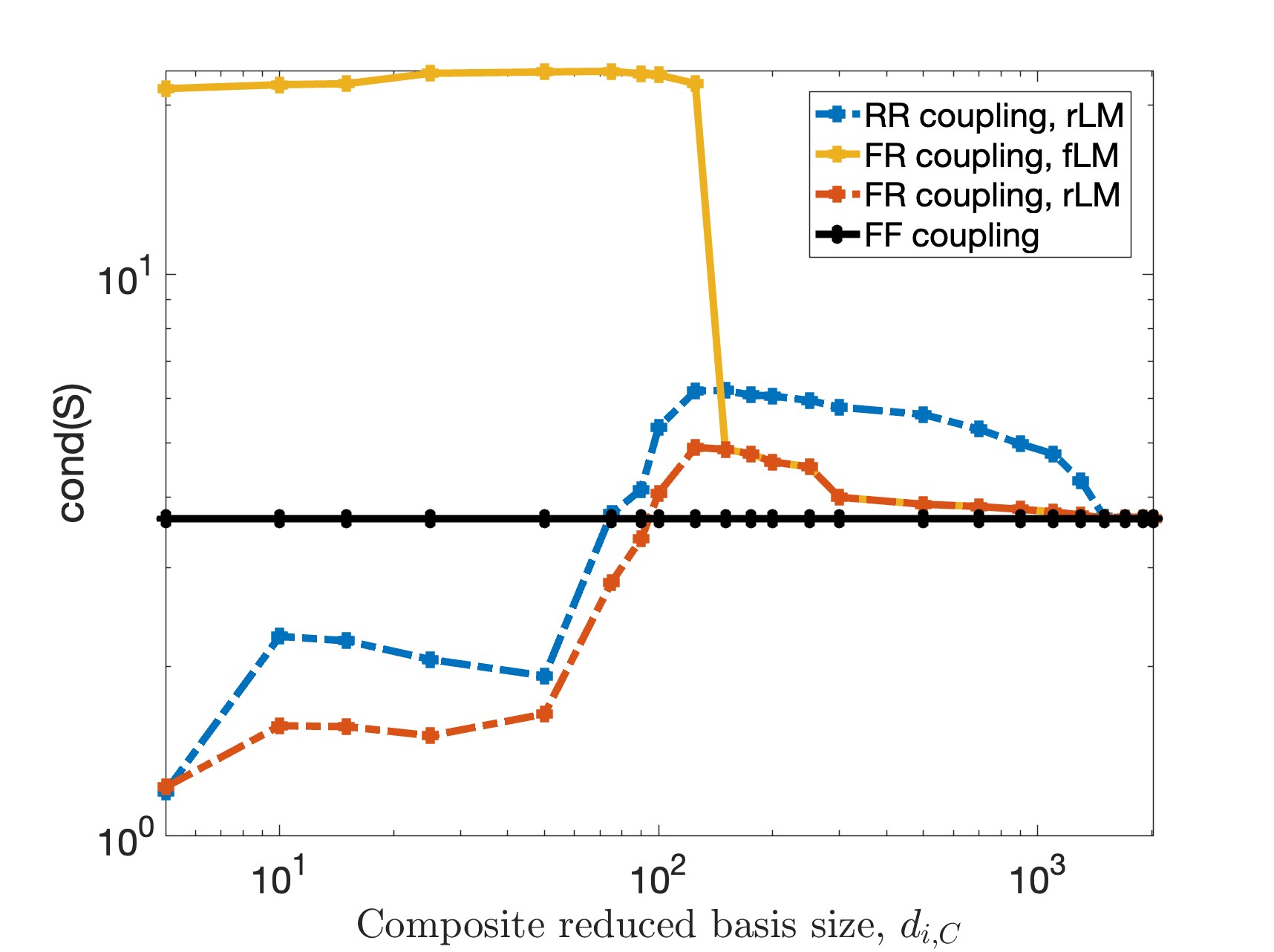}\label{AdC:fig:cond_FFfLM_reprodMulti}}  
\caption{\REV{Condition number of Schur complement matrix for each coupled formulation as a function of the composite reduced basis size $d_{i,C}=d_{i,0}+d_{i,\gamma}$ size  in the predictive regime for the TP example.  Subfigure (a) reports results for all methods evaluated, whereas subfigure (b) focuses only 
on methods with provably well-posed Schur complements.}}\label{AdC:fig:condSMulti}
\end{figure}

\begin{figure}[!ht]
\centering
\subfigure[RR-rLM, 15/10 modes]{\includegraphics[scale=.070]{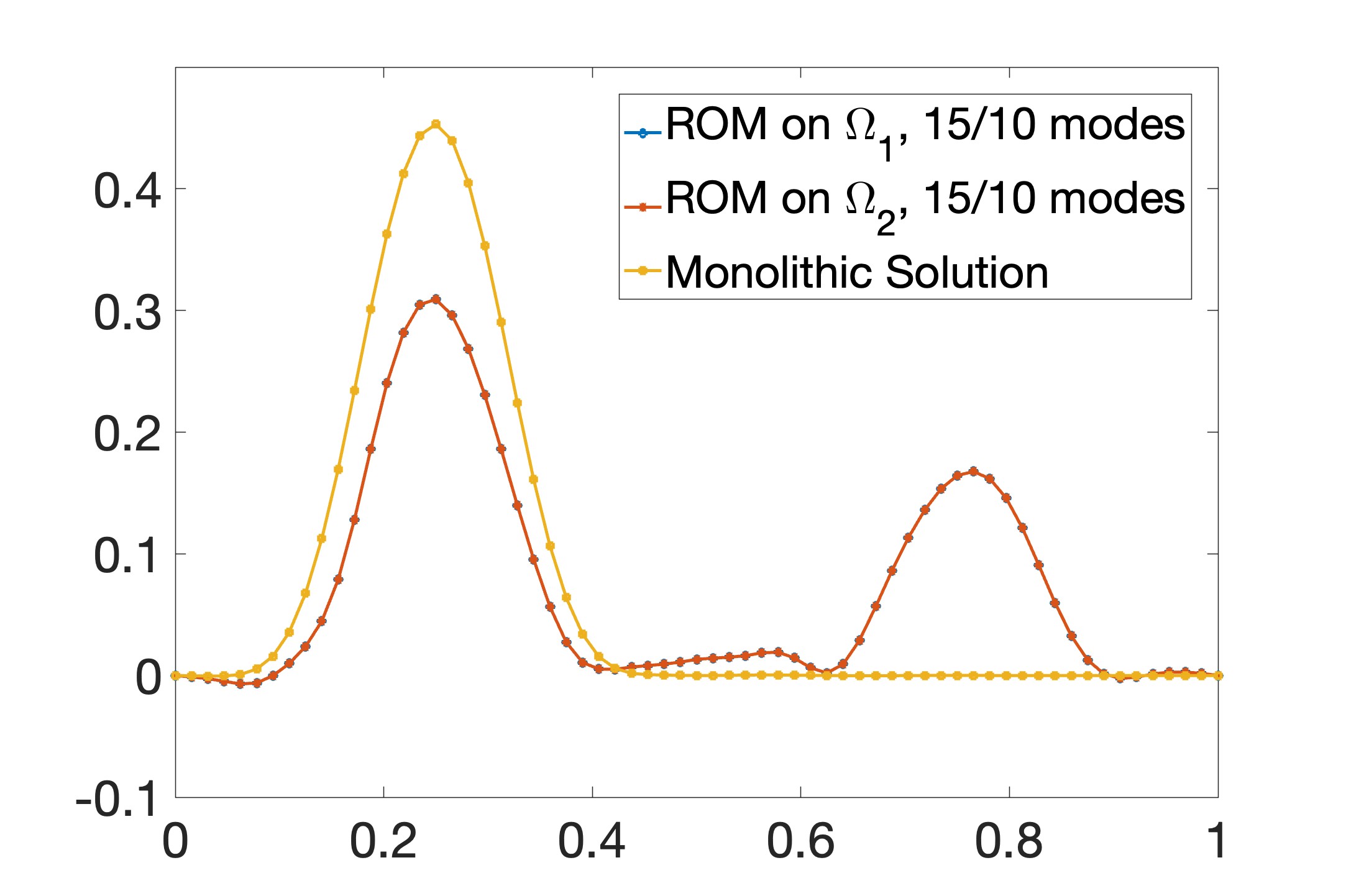}\label{AdC:fig:RRrLM_Inter_15_10_multi}} 
\subfigure[FR-fLM, 15/10 modes]{\includegraphics[scale=.070]{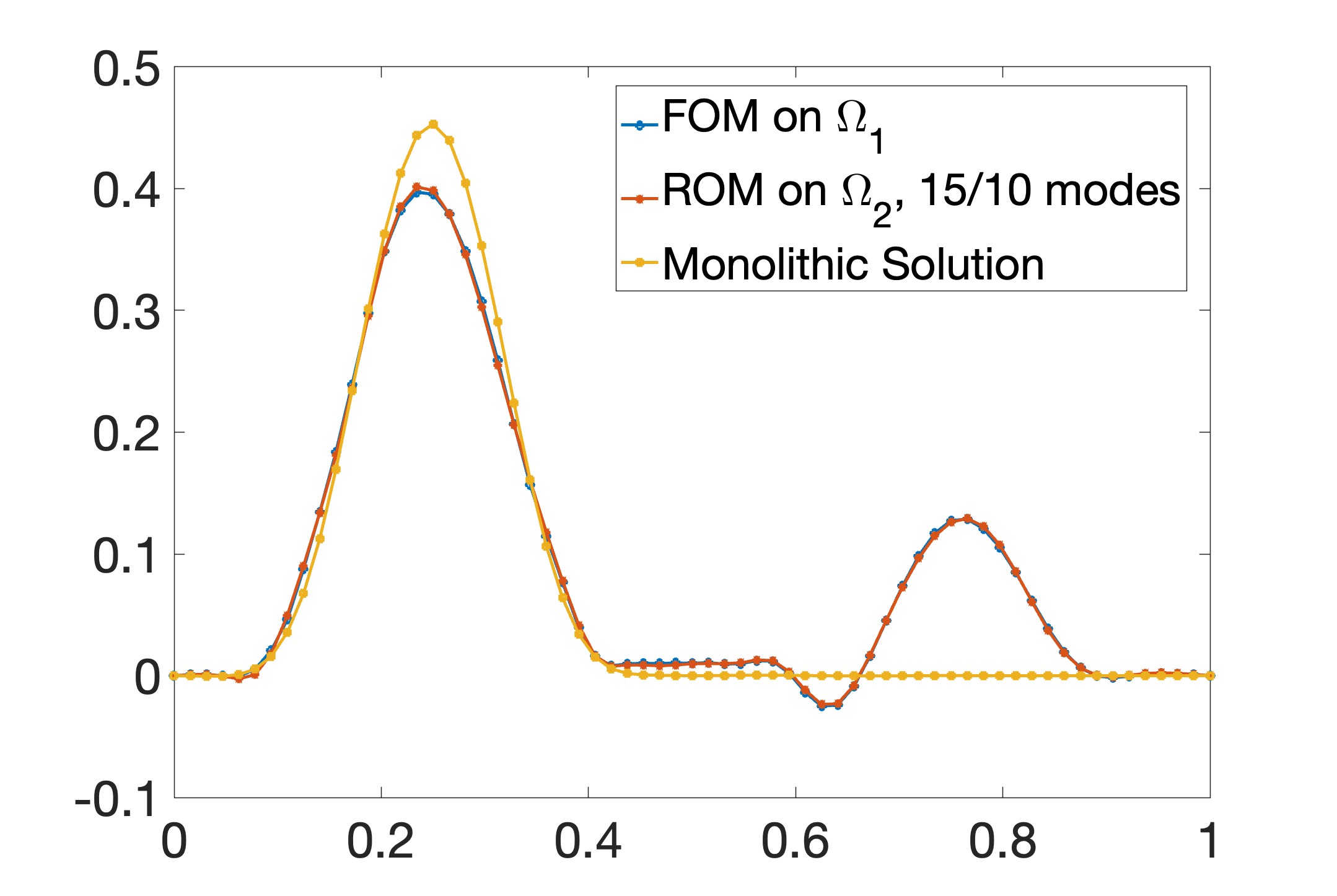}\label{AdC:fig:FRfLM_Inter_15_10_multi}} 
\subfigure[FR-rLM, 15/10 modes]{\includegraphics[scale=.070]{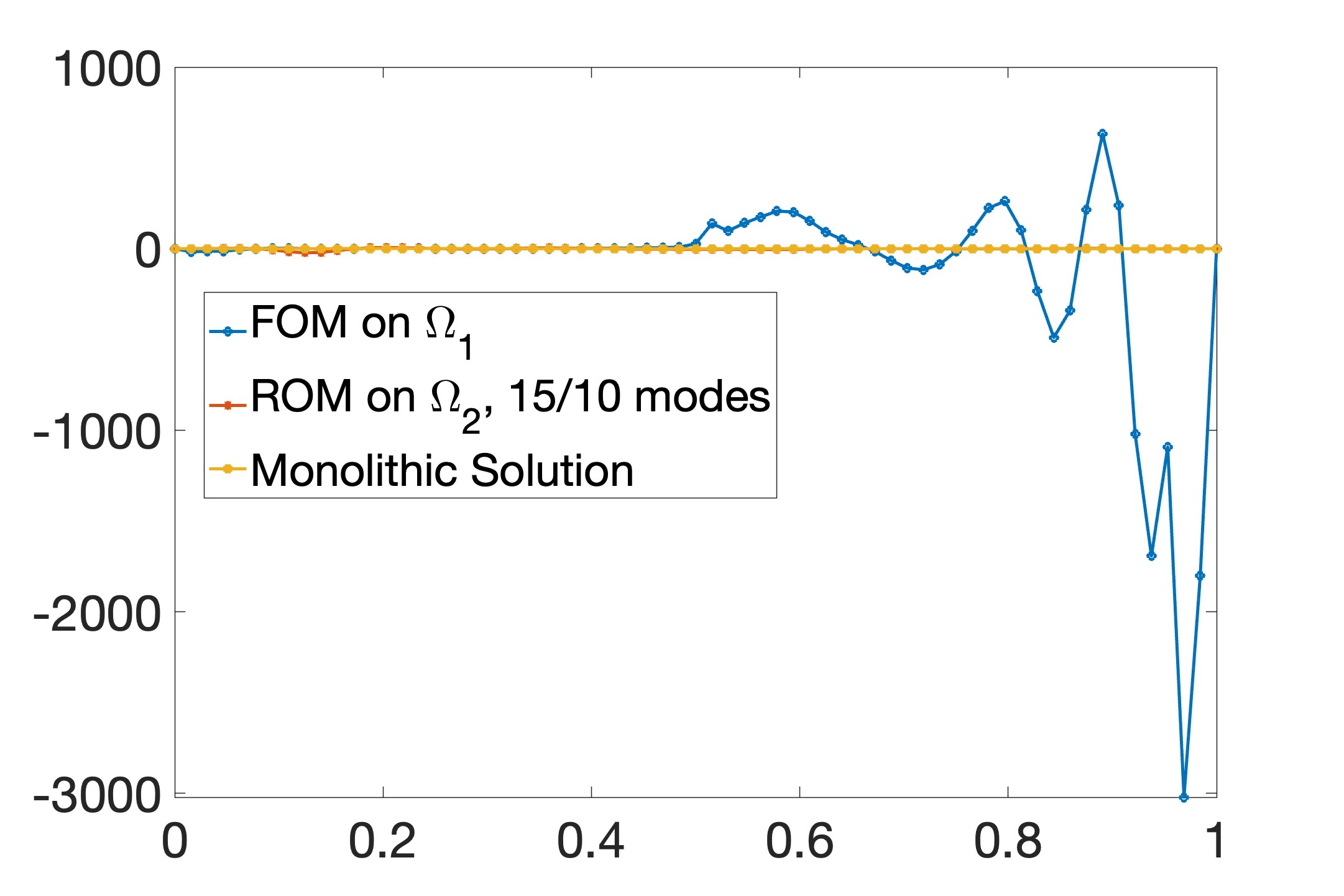}\label{AdC:fig:FRrLM_Inter_15_10_multi}} 
\subfigure[RR-rLM, 60/40 modes]{\includegraphics[scale=.070]{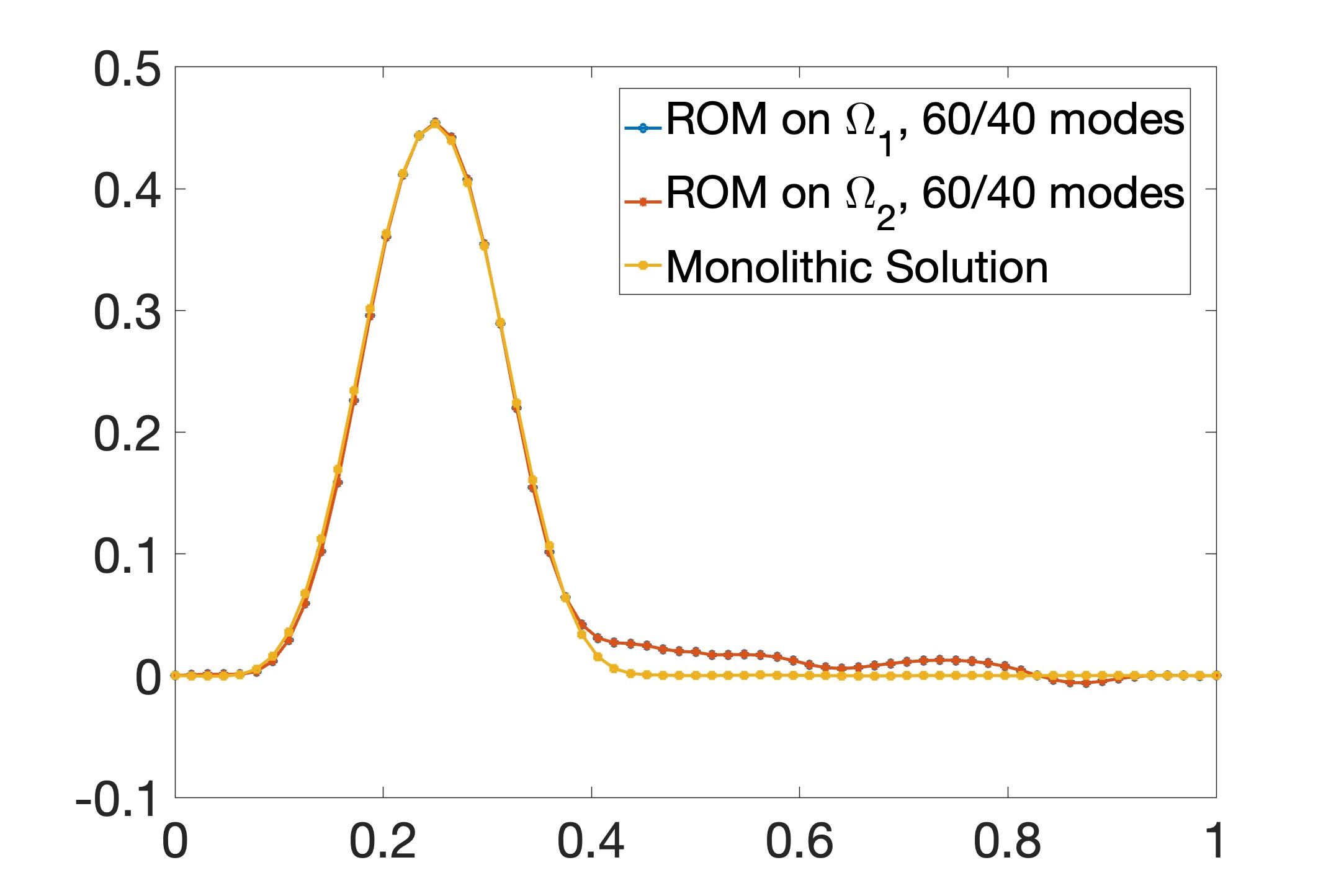}\label{AdC:fig:RRrLM_Inter_60_40_multi}} 
\subfigure[FR-fLM, 60/40 modes]{\includegraphics[scale=.070]{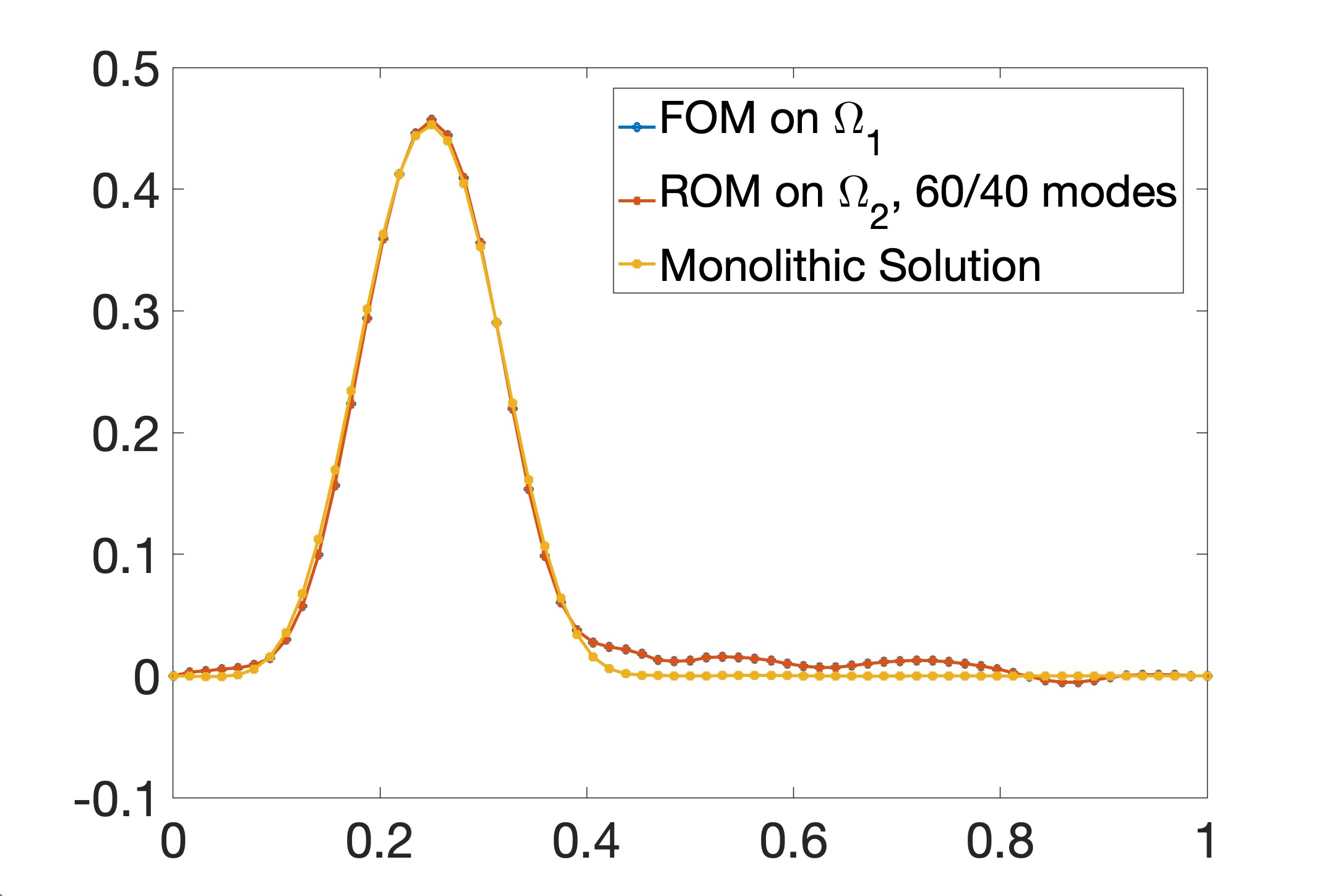}\label{AdC:fig:FRfLM_Inter_60_40_multi}} 
\subfigure[FR-rLM, 60/40 modes]{\includegraphics[scale=.070]{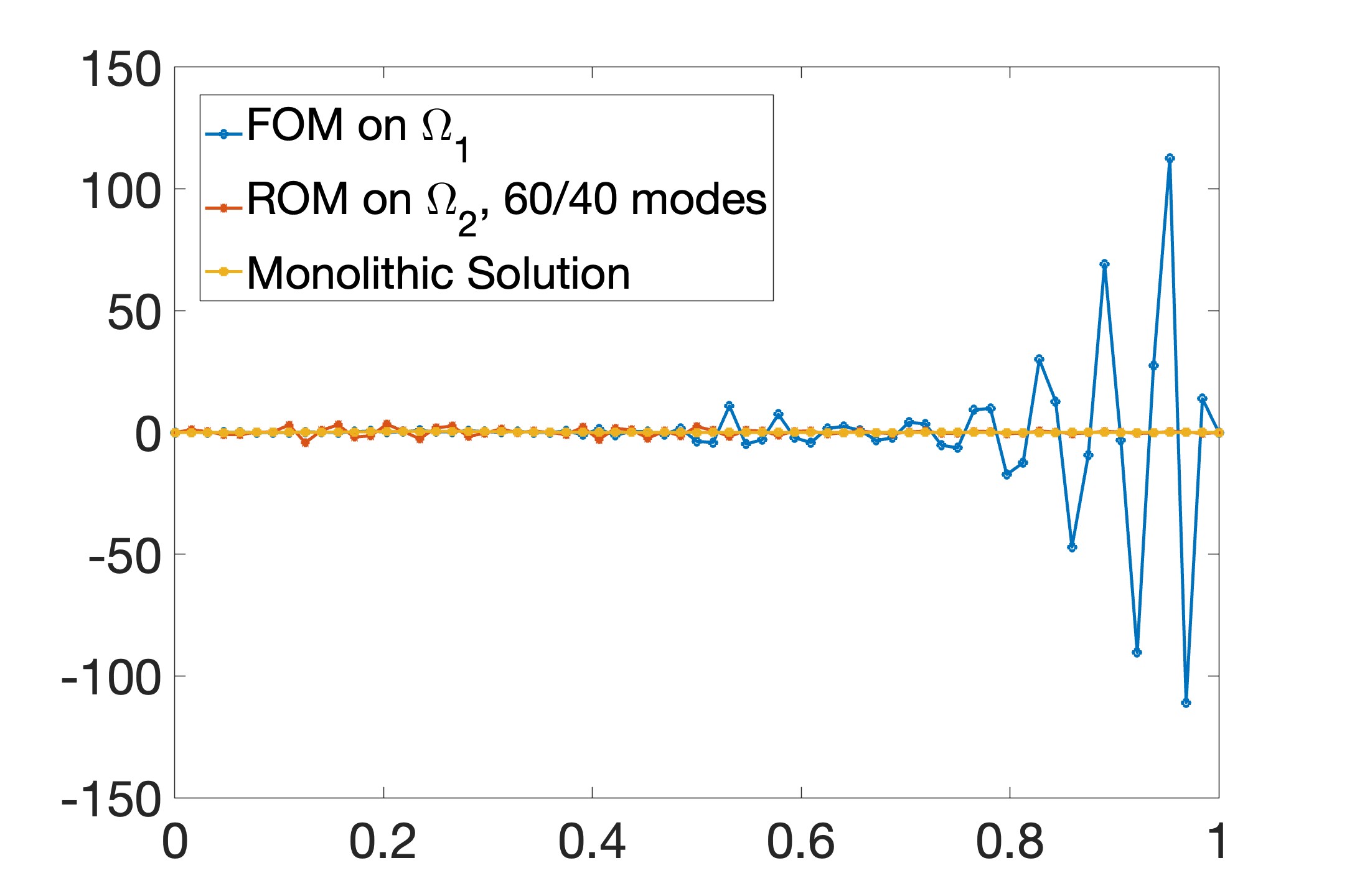}\label{AdC:fig:FRrLM_Inter_60_40_multi}} 
\subfigure[\REV{RR-rLM, 90/60 modes}]{\includegraphics[scale=.070]{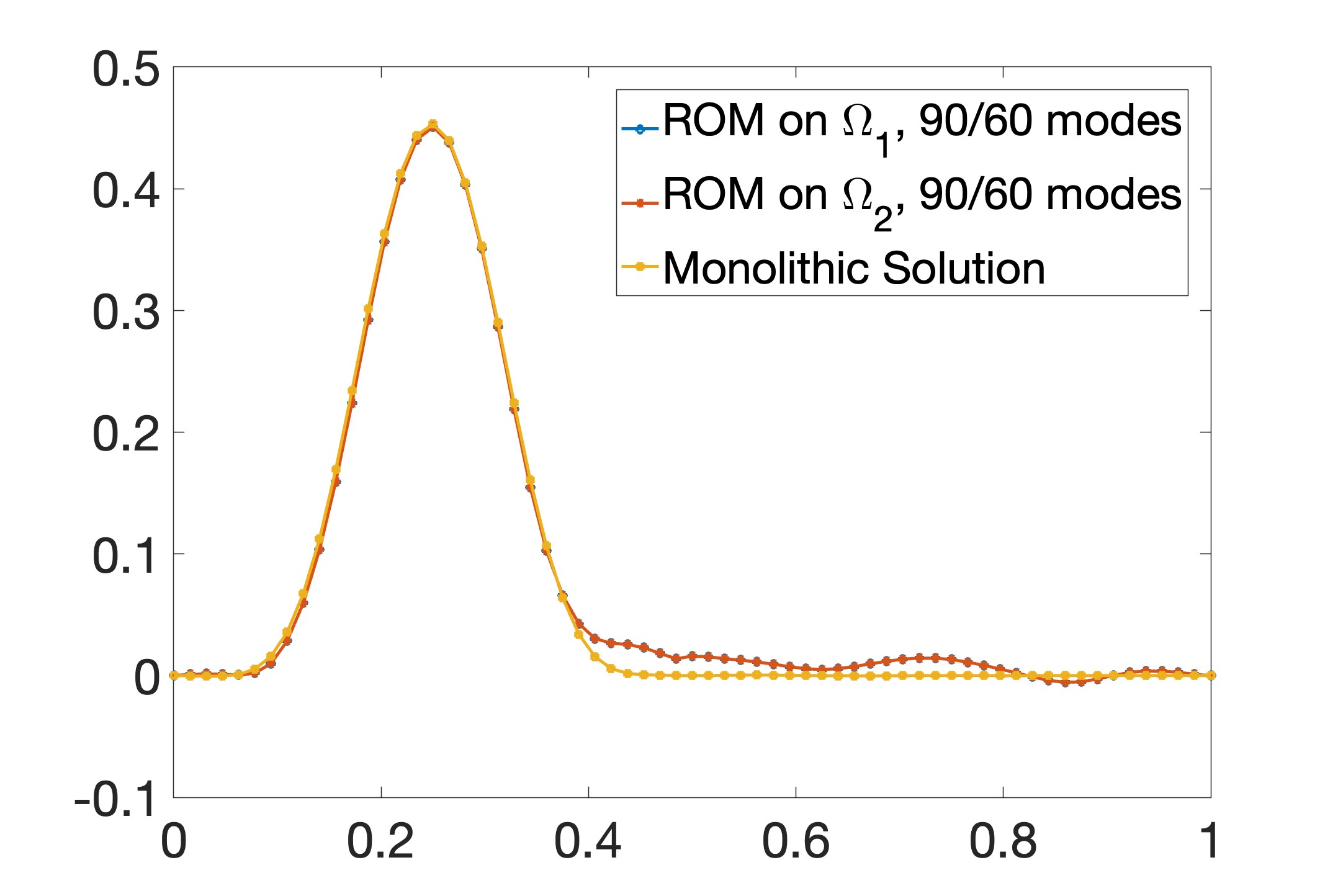}\label{AdC:fig:RRrLM_Inter_90_60_multi}} 
\subfigure[\REV{FR-fLM, 90/60 modes}]{\includegraphics[scale=.070]{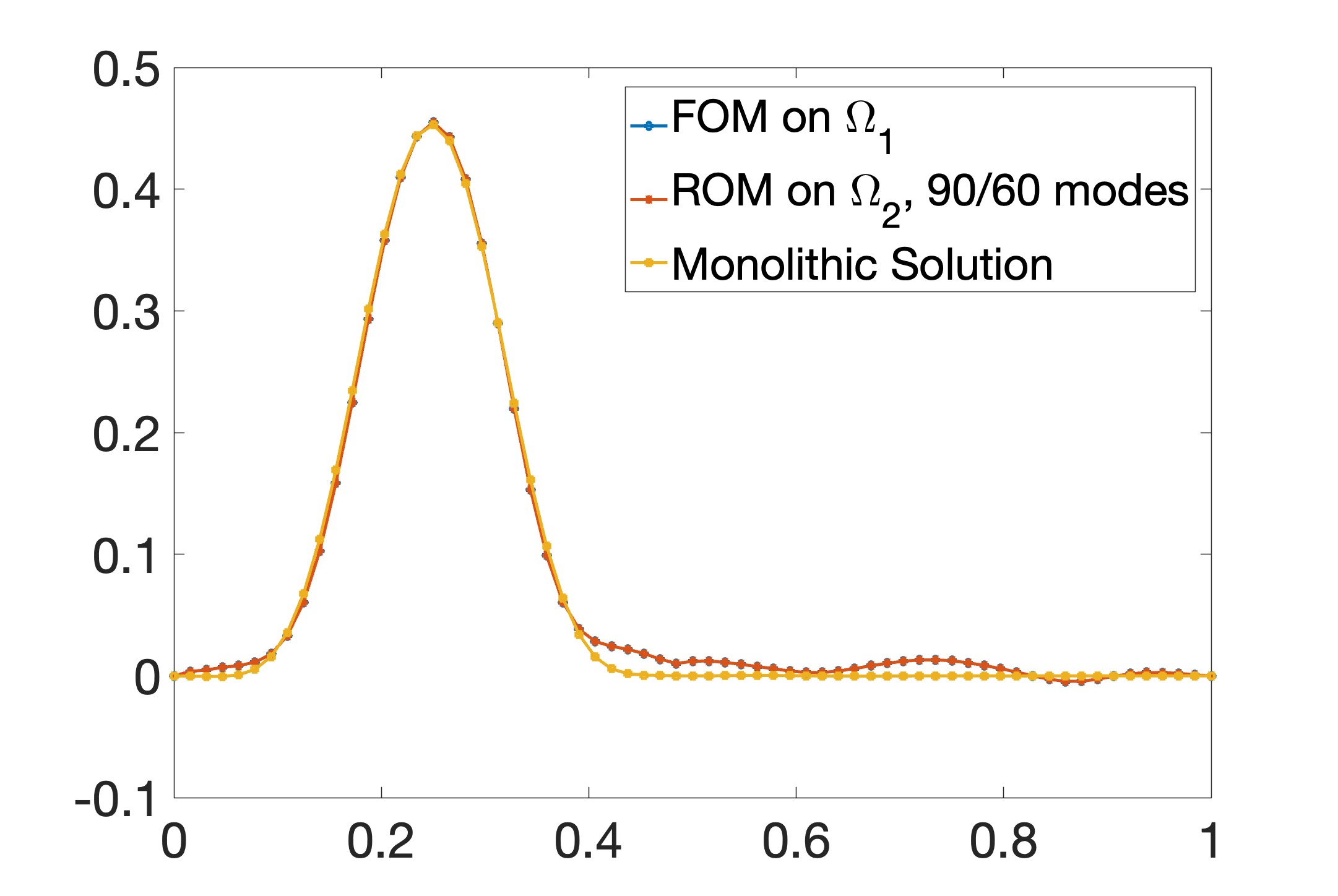}\label{AdC:fig:FRfLM_Inter_90_60_multi}} 
\subfigure[FR-rLM, 90/60 modes]{\includegraphics[scale=.070]{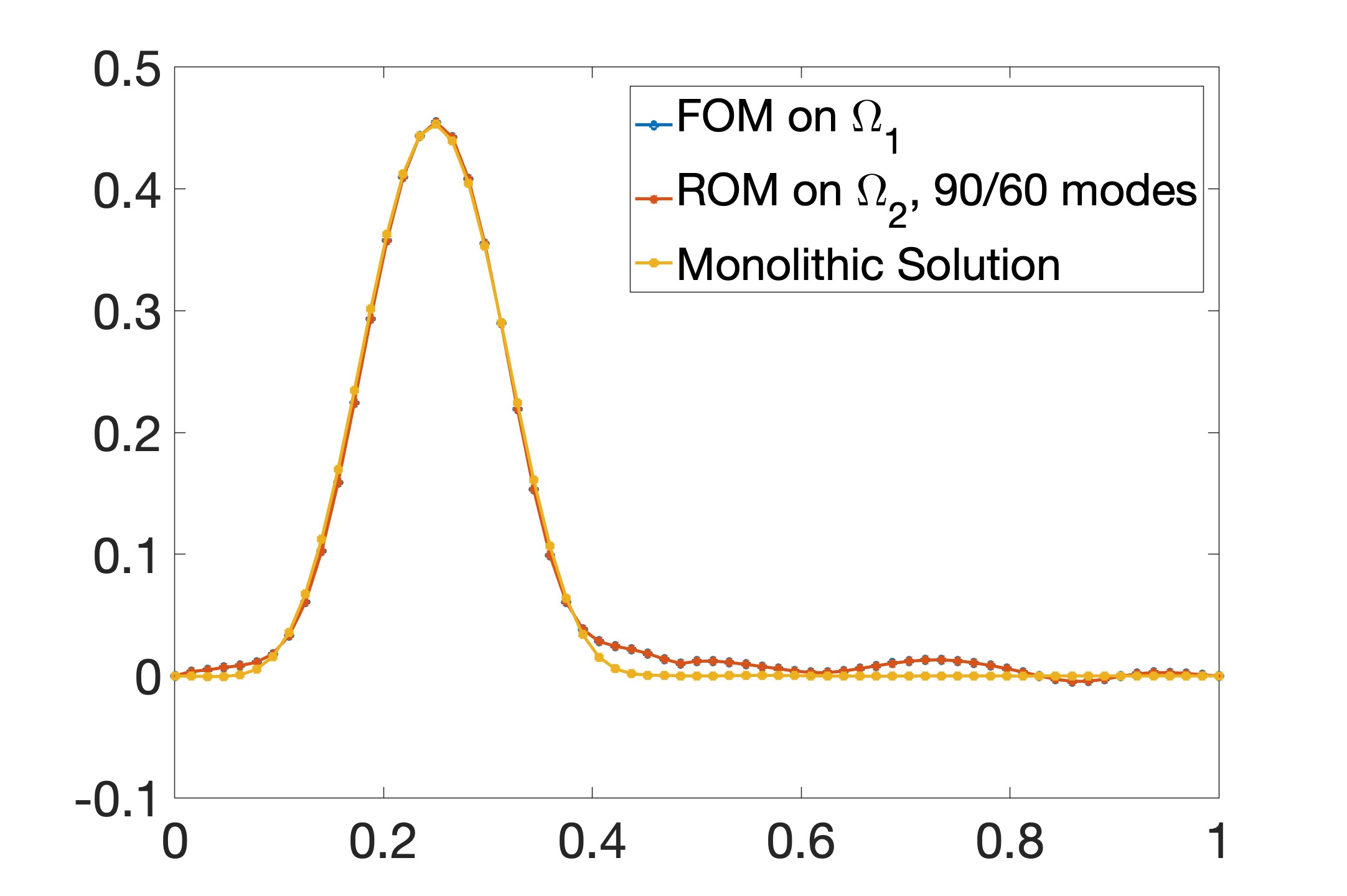}\label{AdC:fig:FRrLM_Inter_90_60_multi}} 
\caption{\REV{Comparison of the interface states at $T_f$ for the partitioned schemes  with provably well-posed Schur complements vs. the single domain (monolithic) solution of the model problem. Predictive test in the TP setting. The oscillations in the FR-rLM formulations with ``small'' RB sizes are due to accumulation of interface errors during the time integration caused by the approximate enforcement \eqref{eq:inequality} of the coupling condition. The legend ``$m/n$ modes'' corresponds to $m$ interior and $n$ interface modes.}}\label{AdC:fig:InterPlots_MultiPhysics}
\end{figure}

%% file: sec_Conclusion.tex
\section{Conclusions} \label{sec:conc}
\REV{The main contributions of this paper are (i)} extension 
 of the IVR scheme to the partitioned solution of coupled ROM-ROM and ROM-FOM problems, and
 \REV{(ii) identification of the trace compatibility condition as a key factor for the well-posedness of the IVR extensions.}
The IVR extension requires each coupled problem to have a non-singular Schur complement in order to compute accurate estimates of the interface flux needed to independently solve the subdomain equations at each time step. Moreover, the Schur complement must be well-conditioned for the partitioned scheme to be robust and accurate. 

Our previous numerical studies \cite{DeCastro_23_INPROC} revealed that the Schur complement of a coupled ROM-ROM problem, based on full subdomain bases and a full Lagrange multiplier space, can become severely ill-conditioned and compromise the accuracy of the coupling.
The key issue with this formulation is the lack of trace compatibility for the Lagrange multiplier space. However, the full subdomain basis does not provide a satisfactory setting to secure this property. In this paper, we demonstrated that a robust and effective solution to ensure well-posed Schur complements is to consider an alternative composite reduced basis comprising independently defined sets of basis vectors for the interface and interior variables, respectively. The interface reduced basis then provides a trace-compatible Lagrange multiplier space for coupled ROM-ROM problems. For coupled ROM-FOM problems, one can use either the reduced interface basis from the ROM side or the interface finite element space from the FOM side. 

Using variational techniques, we proved rigorously that these choices of the Lagrange multiplier space lead to coupled problems with non-singular Schur complements whose condition numbers are independent of the underlying FOM mesh size and the dimension of the composite reduced basis.
\REV{We performed numerical experiments in two distinct simulations settings. In both cases, numerical results confirmed the theoretical analysis and demonstrated the accuracy of the partitioned schemes. The consistent behavior of the well-posed formulations in both simulation settings further underscores the  importance of the trace compatibility condition and the robustness of the schemes.}

Our future work will focus on extension of the partitioned schemes in this paper to nonlinear problems where one has to consider techniques such as DEIM \cite{AdC:DEIM}, gappy POD \cite{AdC:Gappy} or the 
ECSW method \cite{AdC:ECSW} to handle the nonlinear terms.